\documentclass{amsart}
\usepackage{amssymb,amscd,amsfonts,mathabx,mathrsfs}
\usepackage[arrow,matrix,curve]{xy}
\usepackage{fullpage}
\usepackage{graphicx}
\usepackage{slashed}

\usepackage{color}
\definecolor{darkgreen}{cmyk}{1,0,1,.2}
\definecolor{m}{rgb}{1,0.1,1}
\definecolor{green}{cmyk}{1,0,1,0}

\definecolor{test}{rgb}{1,0,0}
\definecolor{cmyk}{cmyk}{0,1,1,0}

\newtheorem{Equation}{}[section]

\newtheorem{theorem}[Equation]{Theorem}
\newtheorem{proposition}[Equation]{Proposition}
\newtheorem{lemma}[Equation]{Lemma}
\newtheorem{corollary}[Equation]{Corollary}

\newtheorem{definition}[Equation]{Definition}
\newtheorem{conjecture}[Equation]{Conjecture}
\newtheorem{remark}[Equation]{Remark}

\def\Id{\operatorname{Id}}
\def\Ch{\operatorname{Ch}}
\def\Thom{\operatorname{Thom}}

\def\I{\operatorname{I}}

\def\Id{\operatorname{Id}}

\def\oH{\operatorname{H}}
\def\HP{\operatorname{HP}}

\def\Supp{\operatorname{Supp}}

\def\Sign{\operatorname{Sign}}

\def\Tr{\operatorname{Tr}}
\def\TR{\operatorname{TR}}

\def\tr{\operatorname{tr}}
\def\sign{\operatorname{sign}}

\def\C{\mathbb C}

\def\N{\mathbb N}
\def\R{\mathbb R}

\def\LL{\mathbb L}
\def\Z{\mathbb Z}

\def\maA{{\mathcal A}}
\def\maB{{\mathcal B}}

\def\maE{{\mathcal E}}
\def\maF{{\mathcal F}}
\def\maM{{\mathcal M}}
\def\maG{{\mathcal G}}

\def\maG{{\mathcal G}}

\def\maH{{\mathcal H}}

\def\maK{{\mathcal K}}
\def\maL{{\mathcal L}}
\def\maN{{\mathcal N}}

\def\what{\widehat}

\def\tm{{\widetilde m}}
\def\tv{{\widetilde v}}
\def\tV{{\widetilde V}}

\def\hD{\what{D}} 
\def\hA{\what{A}} 
\def\hU{\what{U}} 
\def\hV{\what{V}} 
\def\hT{\what{T}} 
\def\hP{\what{P}}

\def\hN{\what{N}} 
\def\hmaF{\what{\maF}}

\def\tM{\widetilde{M}}

\def\tV{\widetilde{V}}

\def\hm{\what{m}} 
\def\hM{\what{M}}
\def\om{\mathfrak{m}}

\def\hat{\what}

\def\dd{\displaystyle}
\def\pa{\partial}

\def\ep{\epsilon}

\begin{document}



\title[Transverse  NCG of  foliations \today]
{Transverse  noncommutative geometry of  foliations \\
\today}

\author[M-T. Benameur]{Moulay-Tahar Benameur}
\address{UMR 5149 du CNRS, IMAG, Universit\'{e} Montpellier, France}
\email{moulay.benameur@univ-montp2.fr}
\author[J.  L.  Heitsch \today]{James L.  Heitsch}
\address{Mathematics, Statistics, and Computer Science, University of Illinois at Chicago} 
\email{heitsch@uic.edu}


\maketitle

\bigskip
 \begin{center}
\dedicatory{\em{Dedicated with respect and admiration to  Alain Connes \\ on the occasion of his seventieth birthday}}
\end{center}

\begin{abstract}\
We define an $L^2$-signature for  proper actions on spaces of leaves of  transversely oriented foliations with bounded geometry. This is achieved  by using the Connes  fibration to reduce the problem to the case of Riemannian bifoliations where we show  that any transversely elliptic first order operator in an appropriate Beals-Greiner calculus, satisfying the usual axioms, gives rise to a semi-finite spectral triple over the crossed product algebra of the foliation by the action, and hence a periodic cyclic cohomology class through the Connes-Chern character. The Connes-Moscovici hypoelliptic signature operator yields an example of such a triple and  gives the differential definition of our ``$L^2$-signature''.  For Galois coverings of bounded geometry foliations,  we also define an Atiyah-Connes semi-finite spectral triple which generalizes to Riemannian bifoliations the Atiyah approach to the $L^2$-index theorem. The compatibility of the two spectral triples with respect to  Morita equivalence is proven, and by using an Atiyah-type theorem proven in \cite{BenameurHeitsch17}, we deduce some integrality results for Riemannian foliations with torsion-free monodromy groupoids. 
\end{abstract}

\tableofcontents

\noindent

\section{Introduction} 

The  local index theorem is  a central result in Connes' noncommutative geometry which applies to all non (necessarily) commutative manifolds, more briefly  called {\em{noncommutative manifolds}}. The main ingredient in this general index formula is  the noncommutative Guillemin-Wodzicki residue which allows the production of all the needed local invariants. The characteristic homology classes, such as the Atiyah-Singer ${\what A}$-class, now have to be replaced by their corresponding classes in (periodic) cyclic cohomology, via the Connes-Chern character. 
When the noncommutative manifold is commutative, the local index theorem is equivalent to the famous Atiyah-Singer formula by using standard relations between   the Guillemin-Wodzicki residues for Dirac operators appearing in the local formula  and the Atiyah-Singer characteristic classes of the given manifold. In particular, one recovers the ${\what A}$-class for the Dirac operator and the $\LL$-class for the signature operator. When the noncommutative manifold describes  the (possibly highly singular) space of leaves of a smooth foliation, the local index theorem gives  a totally new index formula, which was one of the main initial motivations for the local index theorem. See \cite{ConnesBook, CM95, CM98, CM01}.
In this latter situation and when the foliation is transversely oriented and almost Riemannian, Connes and Moscovici used a hypo-elliptic Riemannian operator and a noncommutative pseudodifferential calculus similar to  the Beals-Greiner calculus \cite{BealsGreiner}, to build the noncommutative manifold which suitably describes this transverse Riemannian structure \cite{CM95}.  The local index formula for this particular noncommutative manifold led to a deep Hopf index theorem \cite{CM98, CM01} and allowed Connes and Moscovici to relate their local residues with some more involved topological classes appearing in the work of Gelfand-Fuks.

\medskip

In this paper, we use the Connes formalism of noncommutative manifolds   to investigate a larger class of foliations. More precisely, we consider the case of a singular (transversly oriented) foliated manifold which is the quotient of  a smooth regular foliation  $({\what M}, \hmaF)$ under a proper action of a countable discrete group  $\Gamma$.  Then, based on a semi-finite version of the Connes-Moscovici construction, we obtain a periodic cyclic cohomology class representing the transverse $L^2$-signature for proper  actions on foliations. This cohomology class reduces to the Connes-Moscovici class when the group $\Gamma$ is trivial.  On the other hand, for general $\Gamma$ and when the foliation is zero dimensional, we obtain the signature class for proper actions, a cyclic cohomology class of the algebra $C_c^\infty (\hM)\rtimes \Gamma$, which is a differential answer to the signature problem for proper actions adressed in topological terms in   \cite{Mischenko}, and which generalizes the classical description in terms of an index class  \cite{AtiyahSingerIII}. For free actions on manifolds, our class corresponds to the $\LL$-class in the quotient manifold $M=\hM/\Gamma$.  For general $\Gamma$ and general foliations, we thus get a  semi-finite Riemannian index map 
$$
{\Sign}_{\hM, \hmaF; \Gamma} ^\perp\; : \; K_* (C^*({\what M}, \hmaF)\rtimes \Gamma) \longrightarrow \R.
$$
Here $C^*({\what M}, \hmaF)$ can be taken to be  Connes $C^*$-algebra of the foliation $({\what M}, \hmaF)$ or the maximal $C^*$-algebra of the monodromy groupoid $\what{G}$, and $C^*({\what M}, \hmaF)\rtimes \Gamma$ is the crossed product $C^*$-algebra for the $\Gamma$-action. In order to define such a morphism in general, our method  relies on Connes reduction method to the almost Riemannian case where  the transverse Connes-Moscovici signature operator is well defined and embodies  a {\em{semi-finite}} noncommutative manifold.  This is a notion which slightly generalizes Connes' noncommutative manifolds to encompass the semi-finite index theory, see   \cite{BenameurFack}.  It is worth pointing out that  the extra action of $\Gamma$ requires working on the ambient manifold $\what M$ with its foliation,  without restricting to a $\Gamma$-compatible  transversal. This introduces some extra technicalities which are due to the non-integrability of the normal bundle, while for trivial $\Gamma$, the restriction to a transversal is more natural, see   \cite{CM95}. 

\medskip

 For smooth  bifoliations of bounded geometry, and in order to prove the axioms for our semi-finite spectral triples,   we use an adapted Beals-Greiner calculus. Many proofs are easy  extensions of Kordyukov's work \cite{K97, K07} to the class of bifoliations on the one hand and to the class of bounded geometry foliations on the other, with the extra data of a proper action of a countable group which is dealt with by using von Neumann algebras and semi-finite index theory. Exactly as in the case of Riemannian foliations \cite{K97}, our semi-finite spectral triple  is regular and can as well be defined using the monodromy groupoid $\what{G}$ and the corresponding convolution algebra.  We thus get a well defined real-valued index morphism on the $K$-theory of our algebra $C_c^\infty (\what\maG)\rtimes \Gamma$ which admits a well defined ``extension'' to the $K$-theory of the appropriate completion $C^*$-algebra. The Connes-Moscovici hypo-elliptic Riemannian operator then yields an important example of such a semi-finite spectral triple and gives an interesting definition of the singular signature for spaces of leaves with proper actions. 
More precisely, we define, for any elliptic first order (in the new calculus) operator $\what{D}$ satisfying the usual conditions with respect to our bounded geometry data, a regular semi-finite spectral triple based on a suitable   semi-finite von Neumann algebra $\maN$ with its faithful normal semi-finite positive trace $\TR$ (see Section \ref{ProperActions}) and we get

\medskip

{\bf{Theorem }}\ref{CM-triple}
{\em{Let ${\what D}$ be a first order  $C^\infty$-bounded uniformly supported pseudodifferential  operator with associated operator $\what{D}_\rtimes$, which is uniformly transversely elliptic in the Connes-Moscovici pseudo' calculus, and has a holonomy invariant transverse principal symbol $\sigma ({\what D})$.
Then the triple $(\maB, (\maN, \TR),  {\what D}_\rtimes)$, where $\maB$ is the convolution crossed product algebra of smooth compactly supported functions, is a  semi-finite spectral triple which is finitely summable of dimension equal to the Beals-Grieiner codimension of the bifoliation. }}

\medskip

Working with  compactly supported  functions on the involved groupoids is clearly restrictive in the study of bounded geometry bifoliations.   However, since we are mainly interested in the cocompact case, we have chosen to postpone this discussion to a forthcoming paper where we use an appropriate larger algebra of uniformly supported functions. 
We also focus in the second half of this paper on the special case of free and proper actions, and we prove the compatibility of our spectral triple with the Atiyah-Connes semi-finite spectral triple. 
If we assume that the action of $\Gamma$ is free, so that the quotient foliation $(M, \maF)$ is regular with bounded geometry, and we are in the situation of a Galois covering  $(\what M, \hmaF) \to (M, \maF)$ so that $\Gamma \simeq \pi_1M/\pi_1\hM$, then we can alternatively follow Atiyah's approach to the Galois index theorem.   We then obtain another semi-finite index map  based on the Atiyah von Neumann algebra $\maM$ of $\Gamma$-invariant operators with its semi-finite trace $\tau$, which we naturally call the Atiyah-Connes spectral triple, and which induces an $L^2$ Riemannian index map
$$
{\Sign}_{M, \maF; \Gamma} ^\perp\; : \; K_* (C^*(M, \maF)) \longrightarrow \R,
$$
where now $C^*(M, \maF)$ is the $C^*$-algebra  downstairs. Then we prove the following

\medskip

{\bf Theorem } \ref{typeII}
{\em{Let $(\hM, \what\maF\subset \what\maF')\to (M, \maF\subset \maF')$ be a Galois covering of bifoliations. Assume that ${\what D}$ is a transversely elliptic $\Gamma$-invariant pseudodifferential operator in the Connes-Moscovici sense for the foliation $\what\maF'$, which is essentially self-adjoint  with  the holonomy invariant {\em{transverse principal symbol}}. 
Then  the triple $(\maA, (\maM, \tau), \what{D})$, where $\maA$ is the convolution algebra of compactly supported smooth functions on the groupoid downstairs,   is a semi-finite spectral triple which is finitely summable of dimension equal to the Beals-Greiner codimension. }}

\medskip

If we work with the {\underline{monodromy groupoids}}, then  the crossed product $C^*$-algebra $C^*({\what M}, \hmaF)\rtimes \Gamma$ is Morita equivalent  to $C^*(M, \maF)$ and  we prove in this case the expected compatibility of  the  two semi-finite index maps.   That is, denoting by $\Phi_*: K_* (C^*(M, \maF))\stackrel{\simeq}{\to} K_* (C^*({\what M}, \hmaF)\rtimes \Gamma)$  the Morita isomorphism which is induced by a $C^*$-algebra Mishchenko homomorphism $\Phi$, we prove \\

{\bf{Theorem }} \ref{Versus}
{\em{Assume that ${\what D}$ is a transversely elliptic $\Gamma$-invariant pseudodifferential operator from the Connes-Moscovici calculus ${\Psi'}^1 (\hM, \what\maF'; \what\maE)^\Gamma$ for the foliation $\what\maF'$, which is essentially self-adjoint and has holonomy invariant {\em{transverse principal symbol}}. Then the Connes-Chern characters of the semi-finite spectral triple  $(\maA, \maM, \what{D})$   coincides with the pull-back under the Morita map $\Phi$ of the Connes-Chern character of the semi-finite spectral triple $(\maB, \maN, \what{D}_\rtimes)$. }}

\medskip

Therefore, we get in particular the following compatibility of $L^2$-signatures of the space of leaves
$$
{\Sign}_{\hM, \hmaF; \Gamma} ^\perp \circ\Phi_*\; = \; {\Sign}_{M, \maF; \Gamma} ^\perp.
$$
Since the Atiyah-Connes spectral triple is not defined for non-free actions,  the semi-finite Riemannian map ${\Sign}_{\hM, \hmaF; \Gamma} ^\perp$ is the precise generalization to proper non-free actions of Atiyah's $L^2$ $\Gamma$-signature for spaces of leaves.

\medskip

We state in this situation the following conjecture which generalizes a well-known  conjecture for groups which is due to Baum and Connes.

\medskip

{\bf{Conjecture }} \ref{Rational}\
{\em{Let $G$ be the monodromy groupoid of the foliation $(M, \maF)$ with maximal $C^*$-algebra $C^*G$. Then the transverse signature morphism  $
\Sign_{\hM, \hmaF; \Gamma}^\perp \; : \; K_0(C^*G) \longrightarrow \R$ is
always   rational.}}

\medskip

Notice that  the Connes-Moscovici noncommutative manifold associated in this free and proper case to the quotient foliation $(M, \maF)$ also yields a (type I)  Riemannian map which is now integer valued since it describes a type I Fredholm pairing and does not see the  Galois covering $\hM\to M$
$$
{\Sign}_{M, \maF} ^\perp\; : \; K_* (C^*(G)) \longrightarrow \Z.
$$
When the foliation $\maF$ is zero dimensional, it is well known, thanks to the Atiyah $L^2$-index theorem,  that the two maps $\Sign_{M, \maF; \Gamma}^\perp$ and ${\Sign}_{M, \maF} ^\perp$  coincide, see \cite{AtiyahCovering}.  So, in this case the map $\Sign_{M, \maF; \Gamma}^\perp$ is even integral. 
On the other hand if the foliation $\maF$ is top dimensional then   $C^*(G)$ is Morita equivalent to the maximal $C^*$-algebra of the fundamental group $\pi_1 M$, and the above map ${\Sign}_{M, \maF; \Gamma} ^\perp$ reduces to the morphism induced by the regular trace on $C^*(\pi_1M)$ while ${\Sign}_{M, \maF} ^\perp$ reduces to the morphism induced by the averaging trace on $C^* (\pi_1M)$, that is the trace given by 
the trivial $1$-dimensional representation of $\pi_1M$. Therefore, this is a well known conjecture due to Baum and Connes. Notice again that if $\Gamma$ is torsion free then the coincidence of these two maps is ensured on the range of the maximal Baum-Connes map for $\pi_1M$ and hence $\Sign_{M, \maF; \Gamma}^\perp$ is integral in this case as far as the Baum-Connes conjecture is known to be true. 

\medskip

The expected Atiyah $L^2$ index theorem for general  Galois coverings of spaces of leaves  is probably true.  We conjecture that if  the maximal Baum-Connes map for the monodromy groupoid is surjective and the fundamental groups of the leaves are torsion free, then the  map 
$
\Sign_{M, \maF; \Gamma}^\perp \; : \; K_0(C^*G) \longrightarrow \R,
$
is integral.  In \cite{BenameurHeitsch17}, we concentrated on the simpler case of Riemannian foliations and we used the results of \cite{BenameurHeitschSymbols} to deduce the Atiyah theorem in this case. In particular, we obtained under this Baum-Connes assumption that the Atiyah-Connes Riemannian map defined above using our semi-finite spectral triple, is integer valued in the torsion free case. 

\bigskip

{\em Acknowledgements.}  The authors would like to thank Alan Carey, Adam Rennie and Thierry Fack for several helpful discussions.

\section{Notations and preliminaries}

\subsection{Connes algebras of foliations}

We now review the relation between different algebras associated by Alain Connes with smooth foliations. Since the ambient manifold will not always be compact, the constructions need to be adapted, especially the parametrix construction for  bounded geometry foliations. We shall mainly be interested in a foliation $(\hM, \hmaF)$ which is endowed with a proper  action of a countable group $\Gamma$ and a simplifying assumption is that  the quotient space $M=\hM/\hmaF$ is a compact Hausdorff space with its possibly singular foliation $\maF$. So,  we see $(\hM, \hmaF, \Gamma)$ as a smooth proper realization of $(M, \maF)$ so that the resulting invariants will not depend on this choice of realization.    We set  $\om =\dim (\hM)$,  $p=\dim (\hmaF)$, and $q=\om -p$ for the codimension of the foliation.  We denote by $\what{\maG}$ the holonomy groupoid of $(\hM, \hmaF)$ and by $\what{G}$ its monodromy groupoid. The space $\what{\maG}$ (resp. $\what{G}$) is the space of holonomy (resp. homotopy) equivalence classes of leafwise paths with fixed end-points.  Recall that $\what{G}$ and $\what{\maG}$ are smooth manifolds which may fail to be Hausdorff. In order to  keep this paper in a reasonable and readable form, we assume that all manifolds are Hausdorff, so the manifold $\hM$ is identified with the closed submanifold of units of $\what{G}\simeq \what{G}^{(1)}$ and $\what{\maG}\simeq \what{\maG}^{(1)}$ respectively. The source and range maps on both groupoids are denoted $s$ and $r$ respectively. For $A, B\subset \hM$ we use the classical notations $\what{G}_A=s^{-1} (A)$, $\what{G}^B=r^{-1} (B)$ and $\what{G}_A^B=\what{G}_A\cap \what{G}^B$, and similarly for the holonomy groupoid $\what{\maG}$. 

The quotient map $p:\what{G}\to \what{\maG}$ is then a smooth ''covering''  map. More precisely, for any $\hm\in \hM$ and denoting by $\what{L}_{\hm}$ the leaf through $\hm$, we have Galois coverings
$$
\what{G}_{\hm} \longrightarrow \what{\maG}_{\hm} \stackrel{r}{\longrightarrow} \what{L}_{\hm}.
$$
The first map $\what{G}_{\hm}\rightarrow \what{\maG}_{\hm}$ is the universal cover of the holonomy-leaf $\what{\maG}_{\hm}$, while the composite map $ \what{G}_m\rightarrow \what{L}_{\hm}$ is the universal cover of the leaf $\what{L}_{\hm}$. If $\Gamma_{\hm}$ is the subgroup of $\what{G}_{\hm}^{\hm}$ of those classes with trivial holonomy, then $\Gamma_{\hm}$ is the normal subgroup of $\what{G}_{\hm}^{\hm} =\pi_1(\what{L}_{\hm}, \hm)$ corresponding to the first Galois cover $\what{G}_{\hm}\to \what{\maG}_{\hm}$, and $\what{\maG}_m$ can thus be identified with $\what{G}_{\hm}/\Gamma_{\hm}$. Notice that $\pi_1 (\what{\maG}_{\hm}, \hm)= \Gamma_{\hm}$. Finally, the holonomy cover $\what{\maG}_{\hm}\rightarrow \what{L}_{\hm}$ is the Galois cover corresponding to the holonomy group $\what{\maG}_{\hm}^{\hm}$. 

The spaces $C_c^{\infty} (\what{G})$ and $C_c^{\infty} (\what{\maG})$ of smooth compactly supported functions on $\what{G}$ and $\what{\maG}$ respectively, are endowed with their usual structures of involutive convolution algebras. More precisely, we choose a fixed Lebesgue-class measure $\alpha=(\alpha_{\hm})_{\hm\in \hM}$ on the leaf-manifold ($\hM$ viewed as  a discrete union of leaves, so with discrete transverse topology).   By pulling back, we get Haar systems $(\tilde\eta_{\hm})_{\hm\in \hM}$ on $(\what{G}_{\hm})_{\hm\in \hM}$ and $(\eta_{\hm})_{\hm\in \hM}$ on $(\what{\maG}_{\hm})_{\hm\in \hM}$ which are compatible. So  for instance the measure $\tilde\eta_{\hm}$ is $\what{G}_{\hm}^{\hm}$ invariant and induces the Lebesgue-class measure on $\what{L}_{\hm}$. In addition, $\tilde\eta_{\hm}$ is $\Gamma_{\hm}$-invariant and induces the measure $\eta_{\hm}$ on $\what{\maG}_{\hm}$. 

Given $k, k'\in C_c^{\infty} (\what{\maG})$, the rules are
$$
(kk') (\alpha) := \int_{\what{\maG}_{s(\alpha)}} k(\alpha{\alpha'}^{-1}) k'(\alpha') d\eta_{s(\alpha)} (\alpha'), \;\; k^* (\alpha):= {\overline{k(\alpha^{-1})}},\quad \alpha\in \what{\maG}.
$$
The same rules are defined on $C_c^{\infty} (\what{G})$ by replacing $\eta$ by $\tilde\eta$.  We may complete $C_c^\infty (\what{\maG})$ with respect to the standard $L^1$ norm for the groupoid with Haar system $(\what{\maG}, \eta)$ and obtain an involutive Banach convolution algebra $L^1(\what{\maG}, \eta)$, see for instance \cite{Renault}. A similar completion defines the involutive Banach convolution algebra $L^1(\what{G}, \tilde\eta)$.  We denote by $C^*(\what{G})$ and $C^*(\what{\maG})$ the {\underline{maximal}} completion $C^*$-algebras defined as usual using all $L^1$-continuous $*$-representations, see again \cite{Renault}.
There is a well defined continuous $*$-homomorphism of involutive Banach algebras
$$
\varphi: L^1 (\what{G}, \tilde\eta) \longrightarrow L^1 (\what{\maG}, \eta) \text{ given by } \varphi (k) (\alpha) := \sum_{p(\tilde\alpha) = \alpha} k(\tilde\alpha)
$$
which maps $C_c (\what{G})$ (resp. $C_c^{\infty} (\what{G})$) into $C_c (\what{\maG})$ (resp. $C_c^{\infty} (\what{\maG})$). It is clear that for any $k\in C_c(\what{G})$ the function $\varphi (k)$ is well defined and belongs to $C_c(\what{\maG})$, and a direct inspection shows that
$$
\vert\vert \varphi (k)\vert\vert_{L^1(\what{\maG},\eta)}\;  \leq\;  \vert\vert k\vert\vert_{L^1(\what{G}, \tilde\eta)}.
$$
 It is also clear that $\varphi (k^*)=\varphi (k)^*$ for any $k\in L^1(\what{G})$. Moreover, due to our choice of compatible Haar systems, we also have
$$
\varphi (kk') = \varphi (k) \varphi (k'), \quad k, k'\in L^1(\what{G}, \tilde\eta).
$$
Hence, we end up with a well defined $C^*$-algebra homomorphism that will be used in the sequel
$$
\varphi : C^* (\what{G}) \longrightarrow C^* (\what{\maG}).
$$

Using the Riemannian metric, we can define a {\underline{continuous and proper}} lengh function $\vert\cdot\vert: \what\maG\to [0, +\infty[$ defined by taking the infimum over  a given class in  $\what\maG$. This was proven in \cite{TuHyperbolic}[page 134] in the case of compact ${\what M}$ and Hausdorff $\what\maG$, but the proof extends immediately if we work in the category of bounded-geometry manifolds. Notice that 
$$
\vert \what\gamma\vert = 0 \Leftrightarrow \what\gamma\in \what\maG^{(0)}\simeq {\what M}, \quad
\vert\what\gamma\what\gamma '\vert \leq \vert\what\gamma\vert\times \vert \what\gamma '\vert\quad \text{ and } \quad \vert \what\gamma^{-1}\vert = \vert \what\gamma\vert. 
$$
A closed ball neighborhood of a subset $\hA$ in $\what\maG$  is given by  
$$
B(\hA, R) :=\{\what\gamma\in \what\maG\text{ such that } \exists \gamma_1\in \hA\text{ with }\vert \what\gamma\gamma_1^{-1}\vert \leq R\}.
$$
The properness of the length function means precisely that for any $R\geq 0$, the restriction 
$$
(s, r) : B({\what M}, R) \longrightarrow {\what M}\times {\what M} \text{ is a proper map}.
$$
It implies for instance that for any compact subset $\hA$ in $\what\maG$, the space $B(\hA, R)$ is also compact. In particular, given a compact subset $\hA$ in ${\what M}$ the space $\what\maG^{\hA} \cap B({\what M}, R) $ will always be a compact subspace of $\what\maG$. A smooth function or a section $k$ on $\what\maG$ will be said to be uniformly supported  if the support of $k$ is contained in a ball neighborhood $B({\what M}, R)$ of ${\what M}$ for some fixed $R$.  It is said to be smoothly bounded  if all derivatives of $k$ are (uniformly) bounded on $\what\maG$. The least $R$ such that the support of $k$ is contained in the $R$-ball around ${\what M}$ is sometimes called the propagation of $k$. It is then easy to check that the space $C^\infty_u(\what\maG)$ of smooth uniformly supported and smoothly bounded functions on $\what\maG$ is an involutive convolution algebra with the propagation being subadditive. Moreover, the subspace $C_c^\infty (\what\maG)$ of smooth compactly supported functions is a two-sided involutive ideal in $C^\infty_u (\what\maG)$. Notice also that when the manifold $\what M$ is compact, the two algebras $C^\infty_u(\what\maG)$ and $C_c^\infty (\what\maG)$ coincide. 

\subsection{Semi-finite spectral triples}

We briefly  review the notion of regular finite dimensional Connes-von Neumann spectral triples as introduced in \cite{BenameurFack}. In this paper, we shall  also call them semi-finite spectral triples. All von Neumann algebras considered here will be weakly closed $C^*$-subalgebras of the algebra $B(\maH)$ of bounded operators on a given (separable) Hilbert space $\maH$, and will never be type III von Neumann algebras. So they will always  have positive normal semi-finite faithful traces.  Clearly, finite von Neumann algebras are semi-finite. 

\begin{remark} (Dixmier)
If $\tau$ is a normal trace, then $\tau$ is semi-finite iff 
$$
\tau(T)=\sup \{\tau (S), 0\leq S\leq T, \tau(S) < +\infty\}, \quad \forall T \geq 0.
$$
\end{remark}

Suppose $\maM$ is a semi-finite von Neumann algebra  with a faithful semi-finite normal positive trace $\tau$.  For any $p\geq 1$, the non-commutative $L^p$-space $L^p(\maM, \tau)$, as well as the Dixmier ideal $L^{p, \infty} (\maM, \tau)$, is well defined \cite{BenameurFack}. Recall that a $\tau$-measurable operator $T$ belongs to $L^p(\maM, \tau)$ if $\tau \left((T^*T)^{p/2}\right) < +\infty$. The closure of $L^p(\maM, \tau) \cap \maM$ in $\maM$ is the ideal of $\tau$-compact operators and is denoted $\maK (\maM, \tau)$. 

Following \cite{BenameurFack}, see also \cite{Benameur2003}, we introduce the following generalization of the classical notion of Connes' spectral geometry.

\begin{definition}\cite{BenameurFack}\label{def.spectral}\
A $p$-summable ($p\geq 1$) semi-finite spectral triple is a triple $({\mathcal A}, \maM, D)$ where
\begin{enumerate}
\item $\maM\subset B(\maH)$ is a von Neumann algebra faithfully represented in the separable Hilbert space $\maH$ 
and endowed with a (positive) normal semi-finite faithful trace $\tau$;
\item ${\mathcal A}$ is a $*$-subalgebra of the von Neumann algebra $\maM$;
\item $D$ is a $\tau$-measurable (so $\maM$-affiliated) self-adjoint operator such that
\begin{itemize}
\item  $\forall a\in \maA$, the operator $a(D+i)^{-1}$ belongs to the Dixmier ideal 
$L^{p,\infty}(\maM,\tau)$;
\item  Every element $a\in {\mathcal A}$ preserves the domain of $D$ and the commutator $[D,a]$ 
extends to an element of $\maM$.
\end{itemize}
\end{enumerate}
When $\maM$ is $\Z_2$-graded with ${\mathcal A}$ even and $D$ odd, we say that $({\mathcal A}, \maM, D)$ is even and denote by $\gamma\in \maM$ the
grading involution. Otherwise, $({\mathcal A}, \maM, D)$ is called an odd triple.
\end{definition}

When $\maM=B(\maH)$ with its usual trace of operators, we recover the classical notion of spectral triple as introduced by A. Connes, see for instance \cite{ConnesNCG}. In \cite{BenameurFack}, many geometric examples of semi-finite {\em{regular}} spectral triples are described.  The spectral triple $({\mathcal A}, \maM, D)$ is regular  if the following extra-condition is satisfied 
\begin{itemize}
\item For any $a\in {\mathcal A}$, the operators $a$ and $[D,a]$ belong to $\bigcap_{n\in \N} Dom(\delta^n)$, where
$\delta$ is the unbounded derivation of $\maM$ given by $\delta(b)=[|D|,b]$.
\end{itemize}

Since $
L^{p, \infty} (\maM, \tau) \subset L^q(\maM, \tau)\cap \maM \subset \maK (\maM, \tau)$, for any $q\geq p$,
the operator $a(D+i)^{-1}$ is automatically $\tau$-compact for any $a\in \maA$. We did not insist in this defintion on the minimality of $p$, but one prefers to work with the least $p$ such that the axioms are satisfied. It is sometimes desirable to work with more general spectral triples which are not finite dimensional.  In that case,  the condition on $a(D+i)^{-1}$ can be weakened into the assumption that this operator satisfies some heat condition or is just $\tau$-compact when no explicit formula is needed. 

\section{Overview of the NCG of bounded-geometry bifoliations}

In this section we review the Connes-Moscovici spectral triple associated with the transverse structure of a bounded-geometry Riemannian bifoliation, also sometimes  called a bounded-geometry almost-Riemannian foliation according to the terminology used in \cite{ConnesTransverse}. We have chosen to work on the ambient manifold without the expansive choice of a complete transversal so that our constructions are natural and immediately  extend to the setting of proper actions of discrete groups, as will be explained in the next section.  In the case of Riemannian foliations of compact manifolds, the construction is due to Kordyukov, \cite{K97}, and we shall show that  his constructions can be extended with no serious obstacles. Note that the Connes-Moscovici pseudodifferential calculus ${\Psi'}^* ({\what M}, \what\maF'; \what\maE)$ \cite{CM95}, associated to a given  foliation $\what\maF'$ and with the coefficient Hermitian bundle $\what\maE$, can be recovered from our calculus defined below for a bifoliation, when this latter is just  $0\subset \what\maF'$.

\subsection{The $\Psi$DO' calculi on bounded geometry bifoliations}

Let   $\what\maF$ be a smooth foliation on a smooth connected Riemannian manifold $\what M$. The dimension of the foliation $\what\maF$ is $p$ and its codimension is $q$. 
This paragraph is devoted to the bifiltered  pseudodifferential calculus, as described for instance in \cite{K97}, taking into account an extra larger foliation of bounded geometry on ${\what M}$. We are mainly interested in almost-Riemannian foliations which correspond to holonomy invariant (orthogonal) triangular structures as described in \cite{ConnesTransverse} or in \cite{CM95}.   We shall sometimes call them Riemannian bifoliations. Notice though that  the calculus is valid for all bifoliated manifolds and follows usual constructions that we review below.  

We assume that the foliated manifold $({\what M}, \what\maF)$ has $C^\infty$ bounded geometry.   This means that the manifold ${\what M}$ has $C^\infty$ bounded geometry and that moreover, all the leaves satisfy the same bounded geometry assumption. Recall that this means that the injectivity radius associated with the given Riemannian metric is positive so that we have well defined barycentric coordinates which are then assumed to have $C^\infty$-bounded changes, and that the curvature tensor $R$ is $C^\infty$-bounded. The positive injectivity radius insures completeness, i.e. all geodesics can be extended indefinitely. Bounded geometry manifolds have open covers of finite multiplicity by Riemannian balls of a fixed radius, which are domains of injectivity of the exponential map. Also, the complete Riemannian distance 
$d({\what x}, {\what y})$  is then well defined for any $({\what x}, {\what y})\in {\what M}^2$ with the usual properties. Finally, there exist smooth $C^\infty$-bounded partitions of unity which are subordinate to such covers. All these properties are expanded in the seminal monograph \cite{Shubin92}. 

An important class of examples is furnished by Galois coverings of smooth foliated compact  manifolds, where one chooses an invariant metric and the induced metric on the leaves. The larger class of examples given by proper cocompact actions will be treated in Section \ref{ProperActions}, along  with Connes' spectral geometry. 

All vector bundles over ${\what M}$ will also have $C^\infty$ bounded geometry, and we shall denote by $C^\infty_b ({\what M}, \what\maE)$ the space of $C^\infty$-bounded sections of such a given bundle $\what\maE$ over ${\what M}$. We can then choose a $C^\infty$-bounded Hermitian structure and consider as well the space $L^2 ({\what M}, \what\maE)$ of $L^2$-sections of $\what\maE$. In fact, the Sobolev spaces associated with $\what\maE$ are also well defined, see again \cite{Shubin92}.
 In the sequel, for any such vector bundles $\what\maE$ and $\what\maE'$ over ${\what M}$, a differential operator 
$$
\hD: C^\infty ({\what M}, \what\maE) \longrightarrow C^\infty ({\what M}, \what\maE '),
$$
will always be assumed to have $C^\infty$-bounded coefficients. This is the so-called class of $C^\infty$-bounded differential operators and it means that in all barycentric coordinates (so with $C^\infty$-bounded change of coordinates), the matrix-coefficents of $\hD$ are $C^\infty$-bounded with the bounds independent of the local metric discs. Notice that all geometric operators do have $C^\infty$-bounded coefficients and we are mainly interested in these operators. \\

We assume from now on  that we have two smooth foliations $\what\maF$ and $\what\maF'$ such that $T\what\maF\subset T\what\maF'$. We  denote by ${\what V}$ the quotient bundle $T\what\maF' / T\what\maF$, which can conveniently be identified with a subbundle of the tangent bundle $T{\what M}$ which is strictly transverse to the foliation $\hmaF$ (the orthogonal subbundle for the fixed metric), and such that the direct sum bundle $T\hmaF \oplus {\what V}$ is the integrable subbundle $T\what\maF'$ of $T{\what M}$. The holonomy  groupoid of the foliation $\what\maF$ is denoted $\what\maG$. 
\begin{remark}
Although we shall only be concerned with   the case where $\what{V}$ is also integrable, we do not make this assumption. 
\end{remark}

 Without further choices or assumptions, the bundle ${\what V}$ is a subbundle of the transverse bundle ${\what \nu}=\nu_{\what\maF} := T{\what M}/T{\what\maF}$ and the quotient bundle ${\what N}:=T{\what M}/T\what\maF'$ is then also a ${\what\maG}$-equivariant bundle. To sum up, we have two foliations ${\what\maF}'\supset {\what\maF}$  on ${\what M}$ and the transverse bundle of $\what\maF$ in $\what\maF'$ is  preserved by the holonomy action associated with $\what\maF$. In the case where the foliation $\what\maF'$ is maximal with leaves the connected components of ${\what M}$, we recover the usual situation of a single foliation. Another situation which gives back the case of a single foliation corresponds to $\what\maF' = \what\maF$, however, our choice of the Beals-Greiner pseudodifferential bifiltration given below imposes  that the single foliation case rather corresponds to the first situation, i.e. $T\what\maF' = T{\what M}$. The interesting new situations occur though when $0 < v=\dim (\hV) < q$. 

We can modify the bifiltered pseudodifferential calculus of \cite{K97} so it fits, when $p=0$, with the classical ``Heisenberg-type'' calculus introduced in \cite{BealsGreiner} and used in the seminal paper \cite{CM95}. Our main objective is to allow transversely hypo-elliptic operators which are not transversely elliptic in the classical calculus associated with the foliation $\what\maF$. 

 In a local chart, we thus have a decomposition of $\R^q$ into $\R^v\times \R^n$ where $v=\dim ({\what V})$ and $n=q-v=\dim ({\what N})$, the dimension of the ambient manifold being $\dim (\hM)= \om = p+q=p+v+n$.  A similar decomposition holds on the covectors in $\R_q=(\R^q)^*$, and we decompose any transverse covector $\eta$ into $(\eta_v, \eta_n)$ accordingly. We introduce the new radial action of $\R^*_+$ on $ \R_p\times \R_v\times \R_n$ by setting 
$$
\lambda\cdot \xi = \lambda\cdot (\zeta, \eta)  = \lambda\cdot (\zeta, \eta_v, \eta_n) := (\lambda\zeta, \lambda\eta_v, \lambda^2\eta_n).
$$
Here $\zeta$ represents the covector along the leaves of $\what\maF$ and $\eta=(\eta_v, \eta_n)$ is the transverse covector to $\what\maF$, while $(\zeta, \eta_v)$ corresponds to the covector along the leaves of the larger foliation $\what\maF'$. We shall also denote $\lambda\cdot \eta$ the restricted action
$$
\lambda\cdot \eta =  (\lambda\eta_v, \lambda^2\eta_n).
$$
The ``homogeneous norm" for this action is defined as $\vert\eta\vert ' = \left(\vert \eta_v\vert^4 + \vert\eta_n\vert ^2\right)^{1/4}$. 

We shall also sometimes use the notation 
$$
{\vert\xi \vert '} ^2 =\vert\zeta\vert^2+  {\vert \eta\vert '}^2.
$$
Moreover, given a multi-index $\alpha$ we set 
$$
<\alpha>:=\sum_{i=1}^{p+v} \alpha_i + 2\, \sum_{i=1}^n \alpha_{p+v+i}.
$$
The bifiltered symbols can now be defined.

\begin{definition}
An element
$k(z,x,y,\sigma,\zeta,\eta) \in C^{\infty}(\I^p \times \I^p \times \I^q \times \R^p \times \R^p \times \R^q, M_{a }(\C))$ belongs to the class ${S'}^{m,\ell} (\I^p \times \R^n, \R^p, M_{a }(\C))$  if for any multi-indices $\alpha$, $\beta$,  and $\gamma$, there is a constant $C_{\alpha,\beta, \gamma} > 0$ so that 
\begin{equation}\label{LocalSymbol}
\| \pa^{\alpha}_{\zeta,\eta} \pa^{\beta}_{\sigma} \pa^{\gamma}_{z,x,y} k(z,x,y,\sigma,\zeta,\eta)\| \,\, \leq \,\, C_{\alpha,\beta, \gamma}(1 + |\zeta|+|\eta|')^{m- <\alpha>}(1+|\sigma|)^{\ell - |\beta|}.
\end{equation}
\end{definition}
Examples are given by homogeneous symbols. The smooth function $k(z,x,y,\sigma,\zeta,\eta)$ is (positively) homogeneous of bidegree $(m, \ell)$ if for $\vert \sigma\vert \geq 1$, $\vert\zeta\vert \neq 0$ and $\vert\eta\vert' >> 0$ for instance, we have
$$
k(z,x,y, \lambda_1 \sigma, \lambda_2\cdot (\zeta, \eta)) = \lambda_1^\ell \lambda_2^m \; k(z,x,y,\sigma,\zeta,\eta), \quad \text{ for any }\lambda_1 >0 \text{ and } \lambda_2 >0.
$$
As for the usual pseudodifferential calculus, we shall only consider in this paper $1$-step classical symbols, i.e. those which have an assymptotic expansion
$$
k \sim \sum_{j\geq 0, j'\geq 0} k_{m_0-j, \ell_0-j'},
$$
where each $k_{m, \ell}$ is homogeneous of bidegree $(m, \ell)$ and $\sim$ means that for any $m < m_0 , \ell< \ell_0$ the difference $k- \sum_{m< m'\leq m_0, \ell < \ell'\leq \ell_0} k_{m, \ell}$ satisfies the estimate \eqref{LocalSymbol} for the order $(m, \ell)$. From now on and for simplicity, ${S'}^{m,\ell} (\I^p \times \R^n, \R^p, M_{a }(\C))$ will denote the classical $1$-step polyhomogeneous symbols in the sense described above.

Any symbol $k\in {S'}^{m,\ell} (\I^p \times \R^n, \R^p, M_{a }(\C))$ defines an operator 
$A:C^{\infty}_c(\I^{\om}, \C^{a
}) \to C^{\infty}(\I^{\om}, \C^{a })$ by the quantization formula
\begin{equation}\label{theta}
 Au(x,y) \,\, = \,\, (2\pi)^{-2p-q}\; \int e^{i[(x-x'-z)\zeta +(y-y')\eta + z\sigma]} \; k(z,x,y,\sigma,\zeta,\eta)\, u(x',y') \; dz dx' dy' d\zeta  d\eta d\sigma.
\end{equation}
The distributional Schwartz kernel of $A$ is thus the oscillating integral
$$
K_A (x,y;x',y') \,\, = \,\, (2\pi)^{-2p-q}\;  \int e^{i[(x-x'-z)\zeta +(y-y')\eta + z\sigma]}\;  k(z,x,y,\sigma,\zeta,\eta)\;  dz d\zeta  d\eta d\sigma.
$$
The space of such operators with proper support  is denoted ${\Psi'}^{m,\ell}(\I^{\om},\I^p, \C^{a })$.
Recall that $K_A$ is properly supported if the restrictions of the two projections $\I^{\om} \times \I^{\om} \to \I^{\om}$ to the support of $K_A$ are proper. 

If we denote by $\Delta_{\R^p}$, $\Delta_{\R^{p+v}}$ and $\Delta_{\R^n}$ the Laplacian operators in $\R^p$, $\R^{p+v}$ and $\R^n$, then the local operator 
$$
\left(I + \Delta_{\R^{p+v}}^2 + \Delta_{\R^n}\right)^{m/4} \; \left(I + \Delta_{\R^p}\right)^{\ell/2},
$$
can be defined so as to belong to the class ${\Psi'}^{m,\ell}(\I^{\om},\I^p, \C^{a })$.
\\

The space ${\Psi'}^{m, -\infty}(\I^{\om},\I^p, \C^{a })$ is  defined as the intersection space
$$
{\Psi'}^{m, -\infty}(\I^{\om},\I^p, \C^{a }) \; := \; \bigcap_{\ell \in \Z} {\Psi'}^{m,\ell}(\I^{\om},\I^p, \C^{a }).
$$
We also have a convenient direct description of the operators from ${\Psi'}^{m, -\infty}(\I^{\om},\I^p, \C^{a }) $ as follows. We  consider the space 
${S'}^{m,-\infty}(\I^p \times \I^p \times \I^q  \times \R^q, M_{a }(\C))$  of $M_{a }(\C)$ valued functions $a(x,x',y, \eta)$,  so that for any multi-indices $\alpha$ and $\beta$,  there is a constant $C_{\alpha,\beta} > 0$ so that 
$$
\| \pa^{\alpha}_{\eta}\pa^{\beta}_{x,x',y} a(x,x',y,\eta)\| \,\, \leq \,\, C_{\alpha,\beta}(1 + |\eta|')^{m- <\alpha>}.
$$
Here $\alpha= (\alpha_1, \cdots, \alpha_q)$ and $<\alpha> = \sum_{i=1}^v \alpha_i + 2\sum_{i=1}^n \alpha_{v+i}$. The associated operator is  given by the quantization formula 
$$
Au(x,y) \,\, = \,\, (2\pi)^{-q} \int   e^{i(y-y')\eta}a(x,x',y,\eta)u(x',y')  dx'dy'd\eta, \text{ for  } u \in C^{\infty}_c(\I^{\om}, \C^{a }),  x, x' \in \I^p \text{ and }y,y' \in \I^q.
$$
Again, we may (and will only) consider the $1$-step classical symbols which are defined as expansions of (positively) homogeneous symbols, i.e. symbols $a$ as before with the condition
$$
a(x,x',y, \lambda\cdot \eta) = \lambda^m a(x,x',y,\eta), \quad \text{ for }\vert \eta\vert' >> 0 \text{ and }\lambda >0.
$$
We may as well consider the spaces of pseudodifferential operators as above but with the extra condition that  their Shwartz kernels are compactly supported in $\I^{\om}\times \I^{\om}$.  Then we get the classes 
$$
{\Psi'}_c^{m,\ell}(\I^{\om},\I^p, \C^{a })\quad\text{ and }\quad {\Psi'}_c^{m, -\infty}(\I^{\om},\I^p, \C^{a }).
$$
If $a \in {S'}^{m,-\infty}(\I^p \times \I^p \times \I^q  \times  \R^q, M_{a }(\C))$  is a classical symbol as above (with properly supported Schwartz kernel for $A$) then we write $A \in {\Psi'}^{m,-\infty}(\I^{\om},\I^p, \C^{a })$. Notice that if we identify  $u\in C^{\infty}_c(\I^{\om}, \C^{a })$ with an element ${\bar u}$ of $C_c^\infty (\I^p, C_c^\infty (\I^q, \C^a))$ then the operator $A$ becomes the usual leafwise smoothing operator with values in the Connes-Moscovici pseudodifferential calculus, i.e.
$$
A{\bar u} (x) := \int_{\I^p} A (x, x') {\bar u} (x') dx',
$$
where $A(x, x')$ is a  operator from the pseudodifferential calculus defined in \cite{CM95} and denoted there $\Psi$DO' of order $m$.
Given such $a\in {S'}^{m,-\infty}(\I^p \times \I^p \times \I^q  \times \R^q, M_{a }(\C))$, the principal symbol is defined as the top-degree homogeneous part of $a$, given as the limit as $\lambda\to +\infty$ of $\frac{a (x,x',y, \lambda\cdot \eta) }{\lambda^m}$.  We shall see that composition of principal symbols is then  given by convolution over $\I^p$ of composition of symbols in the sense of \cite{CM95}. 

\begin{proposition}\label{m,-infty}
The above two definitions of the class ${\Psi'}^{m,-\infty}(\I^{\om}, \I^p, \C^a)$  coincide.  A similar statement holds for compactly supported operators.
\end{proposition}

\begin{proof}\
Fix $A\in \bigcap_{\ell\in \Z} {\Psi'}^{m, \ell} (\I^{\om}, I^p, \C^a)$ with symbol $a$.   Then we need to show that the following expression  yields an operator in ${\Psi'}^{m, -\infty} (\I^{\om}, I^p, \C^a)$
$$
k(x, x', y; \eta) := \frac{1}{(2\pi)^{2p}} \int_{I^p\times \R^{2p}} a (s; x, y, \zeta, \eta; \zeta+\sigma) e^{i\left[ (x-x')\zeta + s\sigma \right]} \;\; ds \; d\zeta\; d\sigma.
$$
Note that $k$ is a well defined distribution as an oscillatory integral. Moreover, applying the dominated convergence theorem, we can write for any given multi-indices $\alpha_1, \alpha_2, \alpha_3$ and $\beta$,
\begin{eqnarray*}
\pa_x^{\alpha_1} \pa_{x'}^{\alpha_2} \pa_{y}^{\alpha_3} \pa_{\eta}^{\beta}  \;\; k(x, x', y ; \eta) & = & \int \pa_x^{\alpha_1} \pa_{x'}^{\alpha_2} \pa_{y}^{\alpha_3} \pa_{\eta}^{\beta} \left[a (s; x, y, \zeta, \eta; \zeta+\sigma) e^{i\left[ (x-x')\zeta + s\sigma \right]} \right]\;\; ds \; d\zeta\; d\sigma\\
& = & \sum_{\gamma \leq  \alpha_1} \int  (-i\zeta)^{ \alpha_2} \pa_x^{\gamma}  \pa_{y}^{\alpha_3} \pa_{\eta}^{\beta} a (s; x, y, \zeta, \eta; \zeta+\sigma) (i\zeta)^{\gamma} e^{i\left[ (x-x')\zeta + s\sigma \right]} \;\; ds \; d\zeta\; d\sigma.
\end{eqnarray*}
Hence there exists a constant $C\geq 0$ such that
$$
\left\vert \pa_x^{\alpha_1} \pa_{x'}^{\alpha_2} \pa_{y}^{\alpha_3} \pa_{\eta}^{\beta}  \;\; k(x, x', y ; \eta) \right\vert \; \leq \; C \sum_{\gamma \leq  \alpha_1} \int (1+\vert\zeta\vert)^{\vert\alpha_2\vert + \vert \gamma\vert} \; (1+\vert\zeta\vert+ \vert \eta\vert ')^{m-<\beta>} (1+\vert\zeta\vert+ \vert \sigma\vert )^{\ell} \; d\sigma d \xi.
$$
Notice now that for any $(m, m')\in \Z^2$, an easy exercise shows that there exists $\ell_0 (m, m')\in \Z$ such that for any $\ell \leq \ell_0(m, m')$, we have
$$
\sup_{t \geq 0}\;  \frac{1}{(1+t)^m}\,  \int_{\R^{2p} } (1+ \vert \zeta\vert)^{m'} (1+\vert\zeta\vert+ t)^{m} (1+\vert\zeta\vert+ \vert \sigma\vert )^{\ell} \; d\sigma d \zeta \; < +\infty.
$$
Applying this estimate, we obtain for $\ell$ negative enough
$$
\left\vert \pa_x^{\alpha_1} \pa_{x'}^{\alpha_2} \pa_{y}^{\alpha_3} \pa_{\eta}^{\beta}  \;\; k(x, x', y ; \eta) \right\vert \; \leq \;  C(\alpha, \beta, \ell) (1+\vert\eta\vert)^{m-< \beta>}. 
$$
The converse inclusion is proven similarly and is left as an exercise. 
\end{proof}

We now extend these definitions to the global situation of our bounded geometry bifoliation $({\what M}, \what\maF\subset\what\maF')$. Since we shall only need operators with finite propagation for a complete metric on $\what M$, we shall restrict to the class of finite propagation pseudodifferential operators. The general theory for properly supported operators is more involved, see \cite{Shubin87} for the non-foliated case.

Let $\hU \simeq \I^p\times \I^q$ be a distinguished foliation chart for $\what\maF$ so that the restriction $\what\maE | _{\hU} \simeq \hU\times \C^a$. Let $\hU \times_{\what\gamma} \hU'\simeq \I^{\om}\times \I^q$ be a  chart for the holonomy groupoid $\what\maG$ corresponding to $\what\gamma\in \what\maG_{\hU}^{\hU'}$ with $\hU'$ a distinguished chart with the same properties,  \cite{ConnesIntegration}. The definition of a bifoliation allows us to assume that any such local chart of the holonomy groupoid of $\what\maF$ is compatible  with  the larger foliation $\what\maF'$, so
$$
\hU \times_{\what\gamma} \hU'\simeq \I^{\om}\times \I^{v}\times \I^{q-v}.
$$    
Using these charts and trivializations, any element $A_0 \in {\Psi'}^{m,\ell}(\I^{\om},\I^p, \C^{a })$,  defines an operator  (recall that our local operators are assumed to be properly supported)
\begin{equation}\label{LocalOperator}
A_0 (\what\gamma) :C^{\infty}_c(\hU, \what\maE) \longrightarrow C^{\infty}_c(\hU' ; \what\maE).
\end{equation}
Such operator will be our local model and will be called an elementary (local) operator of class ${\Psi'}^{m, \ell}$ on the bifoliated manifold $({\what M}, \what\maF\subset\what\maF' )$.  The same construction for $(m,-\infty)$ classes yields elementary (local) operators of class ${\Psi'}^{m, -\infty}$. {{Recall that we only consider local open sets which satisfy the usual condition on bounded-geometry manifolds and our distinguished open covers will always be assumed to have bounded diameters. Also, all operators are assumed to have the $C^\infty$-bounded coefficients, with bounds independent of the chosen local charts. Recall that an operator $A$ acting on the sections of $\what\maE$ over ${\what M}$, has  finite propagation if there exists a constant $C>0$, such that for any $\varphi, \varphi' \in C^\infty ({\what M})$ with $d(\Supp (\varphi), \Supp (\varphi')) > C$, one has $
M_{\varphi'} \circ A \circ M_\varphi = 0.$
The operator $M_\bullet$ is as usual multiplication  by the smooth function $\bullet$ and $d$ is our fixed complete distance. The class ${\Psi'}^{m, \ell}$  (resp. ${\Psi'}^{m, -\infty}$) that we shall consider here will be composed of locally finite sums of such elementary operators of class ${\Psi'}^{m, \ell}$  (resp. ${\Psi'}^{m, -\infty}$), so will always be assumed to have  finite propagation and are also called uniformly supported. More precisely,  a finite propagation linear map $A: C_c^\infty ({\what M}; \what\maE) \rightarrow C_c^\infty ({\what M}; \what\maE)$ is  of class ${\Psi'}^{m, \ell}$ (resp. ${\Psi'}^{m, -\infty}$) if it coincides,  in any local charts $ \hU$  and $\hU'$ as above, with a  finite sum of elementary operators  of class ${\Psi'}^{m, \ell}$ (resp. of class ${\Psi'}^{m, -\infty}$), and  with a global bound on the number of these elementary operators. It will be compactly supported if it has compact support in ${\what M}\times {\what M}$. On the other hand, a uniform smoothing operator is an operator with smooth Schwartz kernel $k$ which is uniformly supported (or equivalently has finite propagation with respect to the complete distance) and such that $k$ is $C^\infty$-bounded, see \cite{Shubin92}. This latter property means that  we can estimate the derivatives of $k$ in local coordinates by constants which do not depend on the local chart and hence are uniform bounds over ${\what M}\times{\what M}$. Such a uniform smoothing  operator induces a bounded operator  between any {\underline{usual}} Sobolev spaces of the bounded-geometry manifold ${\what M}$, \cite{Shubin92}.   The space (obviously a $*$-algebra) of  uniform  smoothing operators is denoted by $\Psi^{-\infty} ({\what M}, \what\maE)$, while we denote by ${\Psi}_c^{-\infty} ({\what M}, \what\maE)$ the space of compactly supported smoothing operators.  Again see the seminal monograph \cite{Shubin92} for a complete exposition of all these properties and results.}}

\begin{definition}\label{BigradedPseudo}\cite{K97}
\begin{itemize}
\item Denote by ${\Psi'}^{m,\ell} ({\what M}, \what\maF\subset \what\maF'; \what \maE)$ the space of operators of the form $T=A + R$ where $A$ is a (finite propagation)   operator of class  ${\Psi'}^{m,\ell}$ and $R\in \Psi^{-\infty}({\what M}, \what\maE)$ is a uniform  smoothing operator. 
\item Denote by ${\Psi'}_{c}^{m,\ell} ({\what M}, \what\maF\subset \what\maF'; \what \maE)$ the space of operators of the form $T=A + R$ where $A$ is a compactly supported   operator of class ${\Psi'}^{m,\ell}$ and $R\in \Psi_c^{-\infty}({\what M}, \what\maE)$ is a compactly supported  smoothing operator. 
\end{itemize}
\end{definition}

So, an operator $T\in {\Psi'}^{m,\ell} ({\what M}, \what\maF\subset \what\maF'; \what \maE)$  preserves the space of compactly supported smooth sections. If  $A$ is associated with a distributional kernel $k$, then finite propagation is equivalent to  the existence of a constant $C>0$ such that 
$$
d_{\hM} (\hm, \hm ') > C \quad \Longrightarrow \quad k(\hm, \hm') = 0.
$$
Differential operators are automatically uniformly supported ($C=0$ works). Note that  composition of a compactly supported operator with a uniformly supported operator is compactly supported.  

\begin{proposition}
If $A \in {\Psi'}^{m,\ell}({\what M},  \what\maF\subset \what\maF'; \what\maE)$ then the formal adjoint $A^*$ belongs to ${\Psi'}^{m,\ell} ({\what M}, \what\maF\subset  \what\maF'; \what\maE)$. Moreover,  if $B \in {\Psi'}^{m',\ell'}({\what M}, \what\maF\subset \what\maF'; \what\maE)$, then  $A\circ B \in  {\Psi'}^{m+m',\ell+\ell'}({\what M},  \what\maF\subset \what\maF';  \what\maE)$. The corresponding statement holds for compactly supported operators.
 \end{proposition}
 
\begin{proof}\  
Using the first appendix in \cite{Shubin92}, the results of \cite{K97}, and Remark \ref{rhodelta} below, we deduce that for any $R\in \Psi^{-\infty} ({\what M}, \what\maE)$ and any $A\in {\Psi'}^{m,\ell}({\what M}, \what\maF\subset \what\maF'; \what\maE)$, the composite operators $A\circ R$ and $R\circ A$  are uniform smoothing operators. 
Therefore, we only need to show the proposition for locally finite operators of class ${\Psi'}^{m,\ell}$. Using a locally finite partition of unity of $\hM$ associated with an open cover as above (with bounded diameters), this is reduced to considering an elementary operator  $A$ from sections over $\hU$ to sections over $\hU'$ as above.  The proof is thus reduced to the same result in $\I^{\om}=\I^p\times \I^{v}\times  \I^{q-v}$, where this proposition is proven as an easy extension of the main result in \cite{AntonianoUhlmann}, see also Proposition 1.39 of \cite{GreenleafUhlmann}, as well as \cite{K97}. 
The same proof works for the formal adjoint $A^*$. 
The proof for compactly supported operators is the same and we need only note that  composition of compactly supported operators is compactly supported and the adjoint of a compactly supported operator is compactly supported. 
\end{proof}

When ${\what\maF}'=T{\what M}$ is the tangent bundle and ${\what M}$ is compact, we recover the class of pseudodifferential operators ${\Psi}^{m,\ell} ({\what M}, \what\maF; \what \maE)$ considered in \cite{K97}. When the foliation $\what\maF$ is the foliation by points (zero-dimensional), we recover the class ${\Psi'}^m ({\what M}, \what\maF'; \what\maE)$ of pseudodifferential operators on ${\what M}$ described in \cite{CM95}, based on the Beals-Greiner calculus for the splitting $T{\what M} = T\what\maF' \oplus \what\nu_{\what\maF'}$. So, we have
$$
{\Psi}^{m,\ell} ({\what M}, \what\maF; \what \maE) = {\Psi'}^{m,\ell} ({\what M}, \what\maF\subset T{\what M}; \what \maE)\text{ and }{\Psi'}^m ({\what M}, \what\maF'; \what\maE)={\Psi'}^{m,0} ({\what M}, 0\subset T{\what \maF}'; \what \maE).
$$
 Moreover, the Connes-Moscovici pseudodifferential operators  ${\Psi'}^m  ({\what M},  \what\maF'; \what \maE)$,  \cite{CM95}, belong to the class ${\Psi'}^{m, 0} ({\what M}, \what\maF\subset \what\maF'; \what \maE)$ for any subfoliation $\what\maF$ of $\what\maF'$. Any leafwise classical pseudodifferential operator of order $\ell$ on $({\what M}, \hmaF)$ acting on the sections of $\what\maE$ yields an operator in the class $\Psi^{0, \ell}  ({\what M}, \what\maF; \what \maE)$ but with the uniform support defined as an operator in ${\what M}$. Indeed, obvious general functoriality properties hold for our pseudodifferential operators, with respect to the category of bifoliations.

\begin{remark}\label{rhodelta}\
If we extend  the class ${\Psi}^{m,\ell}$ used in \cite{K97} by introducing the H\"ormander weights $(\varrho, \delta)$ for the global index $m$, then
$$
{\Psi'}^{m,\ell} \; \subset \; {\Psi}_{0, 1/2}^{m/2,\ell}  \text{ for } m <0 \; \text{ and } \; {\Psi'}^{m,\ell} \; \subset \; {\Psi}_{0, 1/2}^{m,\ell}   \text{ for } m \geq 0.
$$
\end{remark}

 \medskip
 
Recall that the  cotransverse bundle $\what\nu^*\subset T^*{\what M}$ is the annihilator of $T\what\maF$.  It is isomorphic to $\what{N}^*\oplus \what{V}^*$ where $\what{N}^*\subset T^*{\what M}$ is the annihilator of $T\what\maF'$.  The bundle $\what{V}^*$ is naturally identified with the subbundle  of $T^*\what\maF'$ which is the annihilator of  $T\what\maF$.   We identify it with the dual bundle to the quotient bundle $T\what\maF'/T\what\maF$.   The holonomy action generated by the foliation $\what\maF$ on the cotransverse bundle $\what\nu^*$ induces the holonomy action on the bundle $\what{N}^*\oplus \what{V}^*$. Note that the choice of a supplementary subbundle to $T\what\maF'$ in $T{\what M}$  yields the identification of $\hV^*$ with a subbundle of the transverse bundle $\what\nu^*$ which, in general,  is not preserved by the holonomy action.  So the holonomy action is triangular with respect to the  decomposition  $\what\nu^*\simeq \what{N}^*\oplus \what{V}^*$, i.e. it maps $\what{N}^*$ to itself.

We define similarly the space of leafwise smoothing operators as follows.

\begin{definition}\label{Psi}\
The space  of uniformly supported $C^\infty$-bounded leafwise smoothing order $m$ operators is
$$
{\Psi'}^{m, -\infty}({\what M},  \what \maF\subset \what\maF'; \what\maE) \,\, = \,\, \bigcap_{\ell} {\Psi'}^{m,\ell}({\what M},  \what\maF\subset \what\maF'; \what\maE).
$$
Similarly,  the space  of compactly supported leafwise smoothing order $m$ operators is 
$$
{\Psi'}_c^{m, -\infty}({\what M},  \what \maF\subset \what\maF'; \what\maE) \,\, = \,\, \bigcap_{\ell} {\Psi'}_c^{m,\ell}({\what M},  \what\maF\subset \what\maF'; \what\maE).
$$
\end{definition}

The spaces ${\Psi'}^{m, -\infty}({\what M}, \what\maF\subset \what\maF'; \what\maE)$ and ${\Psi'}_c^{m, -\infty}({\what M}, \what\maF\subset \what\maF'; \what\maE)$ can also be defined directly at the level of symbols as proven above for the local operators. Indeed, the space ${\Psi'}^{m,-\infty}({\what M}, \what\maF\subset \what\maF'; \what\maE)$ of Definition \ref{Psi} is easily identified with the space of   linear maps $T$ on $C_c^\infty ({\what M}, \what\maE)$ of the form $T=A + K$ with $A$  uniformly supported of type  ${\Psi'}^{m, -\infty}$ and $K \in \Psi^{-\infty} ({\what M}; \what\maE)$ is a uniform  smoothing operator.  A similar description holds for compactly supported operators.

The $*$-algebra of all bifiltered uniformly (respectively compactly) supported pseudodifferential operators is
$$
{\Psi'}^{\infty, \infty}({\what M}, \what\maF\subset \what\maF'; \what\maE):=\bigcup_{m, \ell \in \Z} {\Psi'}^{m, \ell}({\what M}, \what\maF\subset \what\maF'; \what\maE),  
$$
respectively
$${\Psi'}_c^{\infty, \infty}({\what M}, \what\maF\subset \what\maF'; \what\maE):=\bigcup_{m, \ell \in \Z} {\Psi'}_c^{m, \ell}({\what M}, \what\maF\subset \what\maF'; \what\maE)\;.
$$
The spaces
$$
{\Psi'}^{\infty, -\infty}({\what M}, \what\maF\subset \what\maF'; \what\maE):=\bigcup_{m\in \Z} {\Psi'}^{m, -\infty}({\what M}, \what\maF\subset \what\maF'; \what\maE)
$$ 
and  
$$
{\Psi'}_c^{\infty, -\infty}({\what M}, \what\maF\subset \what\maF'; \what\maE):=\bigcup_{m\in \Z} {\Psi'}_c^{m, -\infty}({\what M}, \what\maF\subset \what\maF'; \what\maE)
$$
 are also $*$-algebras.  The spaces
$$
{\Psi'}^{-\infty, -\infty}({\what M}, \what\maF\subset \what\maF'; \what\maE):=\bigcap_{m\in \Z} {\Psi'}^{m, -\infty}({\what M}, \what\maF\subset \what\maF'; \what\maE) 
$$
and  
 $$
 {\Psi'}_c^{-\infty, -\infty}({\what M}, \what\maF\subset \what\maF'; \what\maE):=\bigcap_{m\in \Z} {\Psi'}^{m, -\infty}({\what M}, \what\maF\subset \what\maF'; \what\maE)
$$
coincide respectively with the two-sided $*$-ideals ${\Psi'}^{-\infty}({\what M}; \what\maE)$ and ${\Psi'}_c^{-\infty}({\what M}; \what\maE)$.

The space  ${S'}_{hom}^m (\hN^*\oplus\hV^*; {\what\maE})$ is the space  of smooth sections $p$ of the pull-back of the bundle ${\what\maE}$ to $(r^*\hN^*\oplus r^*\hV^*) \smallsetminus \what\maG$ whose support projects to a uniform subset of $\what\maG$ and which are $m$-homogeneous in the sense that 
$$
p(\what\gamma, \lambda\cdot \eta) = \lambda^m p(\what\gamma, \eta), \quad\forall \lambda >0. 
$$
That such homogeneous symbols are defined globally as sections over $(r^*\hN^*\oplus r^*\hV^*) \smallsetminus \what\maG$ is proven in \cite{CM95}. By a uniform subset of $\what\maG$ we mean here a uniform neighborhood of ${\what M}$ in the Hausdorff groupoid $\what\maG$. We may define this notion, using a complete Riemannian lengh function $\vert\cdot\vert$ on the groupoid $\what\maG$ as explained in the previous section.  That is as a subset which is contained in some ball-neighborhood of ${\what M}$ in $\what\maG$ of the type $\{\vert\what\gamma\vert \leq R\}$. When ${\what M}$ is compact, a uniform subset is just a relatively compact subset.
Here $(\what\gamma, \eta)\in r^*\hN^*\oplus r^*\hV^*$ (i.e. $\eta\in \hN_{r(\what\gamma)}^*\oplus \hV_{r(\what\gamma)}^*$) and we only need to impose homogeneity for instance when $\vert\eta\vert' >> 0$.  As in the usual pseudodifferential calculus, we introduce the class ${S'}^m$ of classical symbols of order $\leq m$ as those symbols which have, in local charts, an asymptotic expansion into homogeneous ones. Composition of symbols $p_1$ and $p_2$ of types ${S'}^{m_1}$ and ${S'}^{m_2}$  gives a symbol of type ${S'}^{m_1+m_2}$.  See for instance \cite{CM95} for the details. Given such $p$ of order $m$, the principal symbol is given by the usual formula 
$$
\sigma (\what\gamma, \eta):=\lim_{\lambda\to +\infty} \frac{1}{\lambda^m} \; p(\what\gamma, \lambda \cdot \eta).
$$
So, the principal symbol is a well defined global section which lives in ${S'}_{hom}^m (\hN^*\oplus\hV^*; {\what\maE})$.
The formula for the composition of such homogeneous symbols is given, \cite{CM95, K97}, by
 $$
(\sigma_1\sigma_2) (\what\gamma, \eta)  = \int_{ {\what\gamma}' \in \what\maG^{ r({\what\gamma}) } } \sigma_1 ( {\what\gamma} ', \eta) \circ  W_{{\what\gamma}'} 
[\tilde{\sigma}_2 ( {{\what\gamma}'}\,^{-1} {\what\gamma},  h_{\what\gamma'} \eta)] d\what\eta^{r(\what\gamma)} ({\what\gamma}'), \text{ for }(\what\gamma, \eta)\in r^*\hN^*\oplus r^*\hV^*.
$$
While $W_{{\what\gamma}'}[\bullet]$ is  conjugation by the linear isomorphism from $\what\maE_{s(\what\gamma ')}$ to $\what\maE_{r(\what\gamma ')}$ given by the holonomy action that we have fixed on $\what\maE$, the transformation $h_{{\what\gamma}'}$ is the transpose of the holonomy isomorphism induced by ${\what\gamma'}\,^{-1}$, a linear isomorphism from $\hN_{s(\what\gamma ')}^*\oplus \hV_{s(\what\gamma ')}^*$ to $\hN_{r(\hat\gamma ')}^*\oplus \hV_{r(\hat\gamma ')}^*$, which in general is not diagonal and does not respect the homogeneity condition. So, $h_{{\hat\gamma}'}$ is given by a triangular matrix 
$$
A=\left(\begin{array}{cc} A_{nn} & A_{nv}\\ 0 & A_{vv}\end{array}\right),
$$
and the formula above means that we  replace $\sigma_2 ( {{\hat\gamma}'}\,^{-1} {\hat\gamma},  h_{\hat\gamma'} \eta)$ by the top-degree homegenous part in $\eta$, which is denoted  $\tilde{\sigma}_2 ( {{\hat\gamma}'}\,^{-1} {\hat\gamma},  h_{\hat\gamma'} \eta)$.  See again \cite{CM95}, page 18, for the precise expansion of $\sigma_2 ( {{\hat\gamma}'}\,^{-1} {\hat\gamma},  h_{\hat\gamma'} \eta)$ with respect to $A_{nv} \eta_v$  and which justifies the replacement.
To sum up,  the principal symbol of a given $A\in {\Psi'}^{m, -\infty}({\what M}, \hmaF, \hV; {\what\maE})$ can be defined using local charts and is  well  defined as a global section which belongs precisely to ${S'}_{hom}^m (\hN^*\oplus\hV^*; {\what\maE})$.  Moreover the expected compatibility with the composition of operators  holds as far as we write down a meaningful composition formula for the principal  symbols. Moreover, by reducing to the case of the local operators   $A_0(\hat\gamma)$ for some classes $\hat\gamma\in \what\maG_{\hU}^{\hU'}$, it is easy to extend the Connes-Moscovici proof \cite{CM95}[pages 18-19] and to deduce the following
  
\begin{proposition}
 The  principal symbol induces an involutive algebra homomorphism 
 $$
 \sigma: {\Psi'}^{m, -\infty}({\what M}, \what\maF\subset  \what\maF'; \what\maE) \longrightarrow {S'}_{hom}^m (\hN^*\oplus\hV^*; {\what\maE}).
 $$
 \end{proposition}
 
If we consider  the  Connes-Moscovici pseudo' calculus ${\Psi'}^* ({\what M}, \what\maF', \what\maE)$ as defined in \cite{CM95}, with respect to the splitting of the tangent bundle $T{\what M}$ into $T\what\maF '\oplus \what{N}$, then any $P\in {\Psi'}^m ({\what M}, \what\maF', \what\maE)$ has a well defined  principal symbol $\sigma (P)$ which is an $m$-homogeneous (with respect to the scaling $\lambda\cdot \xi$) section over $T^*\what{M}\simeq T^*\what\maF' \oplus \what{N}^*$, which is compatible with compositions of operators, and which can be restricted to an $m$-homogeneous section over $\what{V}^*\oplus \what{N}^*$. This was proven  in \cite{CM95}.  In addition, we have the following compatibility result.

\begin{proposition}
Given operators $P_1\in {\Psi'}^{m_1} ({\what M}, \what\maF'; \what\maE)$ and $P_2\in {\Psi'}^{m_2, -\infty}({\what M}, \what\maF\subset  \what\maF'; \what\maE)$, the operators $P_1P_2$ and $P_2P_1$ belong to $ {\Psi'}^{m_1+m_2, -\infty}({\what M}, \what\maF\subset  \what\maF'; \what\maE)$.  In addition, for any $\what\gamma\in \what\maG$ and any $\eta\in \what{V}_{r(\what\gamma)}^*\oplus \what{N}_{r(\what\gamma)}^*$,
$$
\sigma (P_1P_2) (\what\gamma, \eta) = \sigma(P_1)( r(\what\gamma), \eta) \; \sigma(P_2) (\what\gamma, \eta) \text{ and } \sigma (P_2P_1) (\what\gamma, \eta) =  \sigma(P_2) (\what\gamma, \eta) \; \sigma(P_1)( s(\what\gamma), ^t(h_{\what\gamma, *})\;\eta).
$$
\end{proposition}

\begin{proof}
The operator $P_1\in {\Psi'}^{m_1} ({\what M}, \what\maF'; \what\maE)$ acts on smooth compactly supported sections as an element of ${\Psi'}^{m_1,0} ({\what M}, \what\maF\subset \what\maF'; \what\maE)$, therefore the composite operators $P_1P_2$ and $P_2P_1$  act as operators from $ {\Psi'}^{m_1+m_2, -\infty}({\what M}, \what\maF\subset  \what\maF'; \what\maE)$. Reducing to local charts, if $P_1$ is the quantization of the classical symbol $p_1$ in the Connes-Moscovici calculus, then $P_1$ is the quantization in our calculus ${\Psi'}^{m_1,0}$ of the symbol 
$$
k_1 (z, x, y, \sigma, \zeta, \eta) = p_1 (x, y, \zeta, \eta).
$$
Using the quantization formulae, a straightforward computation then gives the formula for the principal symbol of the composition. 
\end{proof}

\begin{corollary}\label{Commutator}\
Assume that $P_1\in {\Psi'}^{m_1} ({\what M}, \what\maF'; \what\maE)$ has a holonomy invariant transverse principal symbol and that $P_2\in {\Psi'}^{m_2, -\infty}({\what M}, \what\maF\subset  \what\maF'; \what\maE)$.  Then
$$
\sigma [P_1, P_2] (\what\gamma, \eta) = [\sigma (P_1) (r(\what\gamma), \eta), \sigma (P_2)(\what\gamma, \eta)].
$$
In particular, if $P_1$ or $P_2$  is a scalar operator, then the commutator operator
$$
[P_1, P_2] = P_1P_2 - P_2P_1
$$
lives in $ {\Psi'}^{m_1+m_2 -1 , -\infty}({\what M}, \what\maF\subset  \what\maF'; \what\maE)$
\end{corollary}

\begin{proof}
Since the transverse principal symbol of $P_1$ is holonomy invariant, we have
$$
\sigma(P_1)( s(\what\gamma), ^th_{\what\gamma, *}\eta) = \sigma(P_1)( r(\what\gamma), \eta).
$$
The result follows. If $P_1$ or $P_2$ is a scalar operator, then the commutator $[\sigma (P_1) (r(\what\gamma), \eta), \sigma (P_2)(\what\gamma, \eta)]$ vanishes and hence the principal symbol of $[P_1, P_2]$ has a zero $m_1+m_2$-homogeneous component with respect to our action $\lambda\cdot \eta$. Then since we only deal with one-step polyhomogeneous classical symbols, we have that 
$$
[P_1, P_2] \; \in \;  {\Psi'}^{m_1+m_2 -1 , -\infty}({\what M}, \what\maF\subset  \what\maF'; \what\maE).
$$
\end{proof}

\subsection{Transversal Beals-Greiner order}

\begin{definition}
A classical uniformly supported  pseudodifferential  operator $P$ from the Connes-Moscovici calculus ${\Psi'}^\infty ({\what M}, \what\maF'; \what\maE)$ is uniformly elliptic if  there exists a constant $C>0$ such that the principal symbol $\sigma (\hP)$ of $\hP$, a section over $T^*\what\maF'\oplus \hN^*$, satisfies the estimate 
$$
\vert < \sigma (\hP) (\hm, \xi) (u) , u> \vert \;\; \geq \;\; C <u, u >, \quad \forall u \in \what\maE_{\hm} \text{ and }\forall \xi\not=0.
$$
Such  operator will be called {\underline{uniformly transversely elliptic}} if we only impose this condition under the assumption $\xi\in \hV^*_{\hm}\oplus \hN_{\hm}^*\smallsetminus \{0\}$. 
\end{definition}
Recall that the principal symbol $\sigma (\hP)$ is $m$-homogeneous with respect to the scaling $\lambda\cdot\eta$ defined before. 
So, uniformly elliptic operators are uniformly transversely elliptic, but the class of uniformly transversely elliptic operators is more interesting. The transverse principal symbol of an operator $P$ from the Connes-Moscovici pseudodifferential calculus ${\Psi'}^m ({\what M}, \what\maF'; \what\maE)$ will then be the restriction of the principal symbol $\sigma (P)$ to the subbundle $\hV^*\oplus \hN^*$ of $T^*\what\maF' \oplus \hN^*$.
Following \cite{K97} and \cite{CM95}, we define the transverse order of a pseudodifferential operator in the Connes-Moscovici calculus ${\Psi'}^\infty ({\what M}, \what\maF'; \what\maE)$ as follows.

\begin{definition}
An order $\ell$ uniformly  supported classical operator $P$ in the Connes-Moscovici calculus ${\Psi'}^\ell ({\what M}, \what\maF'; \what\maE)$  has transversal order $m\leq \ell$  (with respect to $\what\maF$) if $P$ has order $m$ in some conic neighborhood $U_\ep=\{\vert\zeta\vert < \ep \vert\eta\vert'\}$ of the total space of the cotransverse subbundle $\what\nu^*$ to $\what\maF$ in the cotangent bundle $T^*{\what M}$, with the lengh function being as above given by $\vert\zeta\vert + \vert \eta\vert'$. 

We denote by ${\Psi'}^m(\what\nu ^*, \what\maE)$ (resp. ${\Psi'}_c^m(\what\nu ^*, \what\maE)$) the space of such uniformly  (resp. compactly) supported pseudodifferential operators with transversal order $\leq m$ in some conic neighborhood of $\what\nu^*$. 
\end{definition}

For the convenience of the reader, we give a proof of the following lemma which will be used  in the sequel. Recall the adapted Sobolev spaces of Appendix \ref{Sobolev}.

\begin{lemma}\label{TransversalOrder}
Assume that $P\in {\Psi'}^\ell (\what M, \what\maF'; \what\maE)$ is a  uniformly supported order $\ell$ operator which has transversal order $m\leq \ell$. Then the operator $P$ belongs to our pseudodifferential  class  ${\Psi'}^{m, \ell-m} ({\what M}, \what\maF\subset \what\maF'; \what\maE)$. In particular,  for every $s, k$, the operator $P$ extends to a bounded operator
$$
P_{s,k}: {\oH'}^{s,k} ({\what M}, \what\maF\subset \what\maF';\what\maE) \longrightarrow {\oH'}^{s-m,k-\ell+m} ({\what M}, \what\maF\subset \what\maF' ;\what\maE).
$$
\end{lemma}

\begin{proof}
Since $P$ is uniformly supported (and neglecting uniform smoothing operators for which the lemma is obvious), we may use a partition of unity argument to reduce the proof to the case of a local operator.   But then we know that the  total symbol $p$ of $P$ satisfies  estimates  in every (relatively compact) local chart from a uniform atlas as in the previous section, namely
$$
\vert\pa_{\zeta, \eta}^\alpha \pa_{x,y}^\beta p(x,y;\zeta,\eta) \vert \leq C_{\alpha, \beta}  (1+\vert \zeta\vert+\vert\eta\vert')^{\ell-<\alpha>}.
$$
We need to  check  that the total symbol $p$  satisfies the estimates
\begin{equation}\label{estimate}
\vert\pa_{\zeta, \eta}^\alpha \pa_{x,y}^\beta p(x,y;\zeta,\eta) \vert \leq C_{\alpha, \beta}  (1+\vert \zeta\vert+\vert\eta\vert')^{m-<\alpha>} (1+\vert\zeta\vert)^{\ell-m}.
\end{equation}
This is obviously true if $(x,y;\zeta,\eta)$ belongs to the conic neighborhood $U_\ep=\{\vert\zeta\vert < \ep \vert\eta\vert'\}$  where $P$ has the transverse order $m$, by simply using the fact that $\ell-m\geq 0$. Outside $U_\ep$ the estimate also holds easily by using the inequality 
$$
1+\vert \zeta\vert+\vert\eta\vert' \leq \left(1+\frac{1}{\ep}\right) \left(1+\vert\zeta\vert\right).
$$ 
Since any such symbol which satisfies the estimates \eqref{estimate}, defines by quantization an element of class ${\Psi'}^{m, \ell-m}$ by simply setting 
$$
a(s, x, y; \zeta, \eta, \sigma) := p(x,y;\zeta,\eta),
$$
the proof is complete. 
\end{proof}

\begin{lemma}\label{TraceClass}\
Assume that $m<-(v+2n)$ and that $\ell < -p$. Then any compactly supported operator $P$ in ${\Psi'}_c^{m, \ell} ({\what M}, \what\maF\subset\what\maF'; \what\maE)$ extends to a bounded operator  on the Hilbert space $L^2({\what M}; \what\maE)$, which is  trace class. 
\end{lemma}

\begin{proof}\
That $P$ extends to a bounded operator has already been proven for $m\leq 0$ and $\ell\leq 0$ and even for $P$ uniformly supported. So we need to prove the trace-class property. We shall first prove that if $m<-(v+2n)/2$ and $\ell < -p/2$, then any compactly supported operator in  ${\Psi'}_c^{m, \ell} ({\what M}, \what\maF\subset\what\maF'; \what\maE)$ yields a Hilbert-Schmidt operator on the Hilbert space $L^2({\what M}; \what\maE)$. Using again a partition of unity argument, we can reduce to the local situation and the operator is associated with a symbol $k(z, x, y; \zeta, \eta, \sigma)$ satisfying the $(m,\ell)$ pseudodifferential estimates. Its Schwartz kernel is then given by 
$$
K(x,y;x',y') = (2\pi)^{-2p-q} \int e^{i[(x-x'-z)\zeta+(y-y')\eta+z\sigma]} k(z, x, y; \zeta, \eta, \sigma) dz d\sigma d\zeta d\eta,
$$
so we need to show that $\dd \int \vert K(x,y;x',y')\vert ^2 dx dy dx' dy' < +\infty$. Using the pseudodifferential estimates, for $M$ and $N$ large enough,
$$
\vert \partial_z^M \partial_{x, y}^N k(z, x, y; \zeta, \eta, \sigma)\vert \leq C_{ M, N} (1+\vert \xi\vert')^m (1+\vert\sigma\vert)^\ell.
$$
Then, we easily deduce the existence of a constant $C\geq 0$ depending on the parameters such that 
\begin{multline*}
\int \vert K(x,y;x',y')\vert ^2 dx dy dx' dy' \leq C \int (1+ \vert\xi-\xi'\vert)^{-N} (1+\vert\sigma-\zeta\vert)^{-M}  \\(1+\vert\sigma'-\zeta'\vert)^{-M'}
(1+\vert \xi\vert')^m (1+\vert\sigma\vert)^\ell (1+\vert \xi'\vert')^m (1+\vert\sigma'\vert)^\ell \; d\sigma d\sigma' d\xi d\xi'.
\end{multline*}
Using Petree's inequality, one deduces the existence of a constant $C'\geq 0$ such that
$$
\int \vert K(x,y;x',y')\vert ^2 dx dy dx' dy' \leq C' \int (1+ \vert \xi-\xi'\vert)^{-N+\vert m\vert+ \vert\ell\vert} (1+\vert\zeta\vert)^{2\ell} (1+\vert\xi\vert')^{2m} d\xi' \,d\xi.
$$
Therefore, it remains to show that $\int (1+\vert\zeta\vert)^{2\ell} (1+\vert\xi\vert')^{2m} d\xi <+\infty$. 
But,
\begin{eqnarray*}
\int (1+\vert\zeta\vert)^{2\ell} (1+\vert \zeta\vert + \vert \eta\vert')^{2m} d\zeta d\eta  & = & \int (1+\vert\zeta\vert)^{2\ell+2m} \left( 1+\frac{\vert \eta\vert'}{1+\vert\zeta\vert} \right)^{2m} d\zeta d\eta\\
& = & \int (1+\vert\eta\vert')^{2m} d\eta \; \times \; \int (1+\vert\zeta\vert)^{2m+2\ell+v+2n} d\zeta.
 \end{eqnarray*}
The integral $\dd \int (1+\vert\eta\vert')^{2m} d\eta$ behaves like the integral $\dd \int (1+{\vert\eta\vert'}^4)^{m/2} d\eta$ and one has
\begin{eqnarray*}
\int (1+{\vert\eta\vert'}^4)^{m/2} d\eta & = & \int (1+\vert\eta_v\vert^4+\vert\eta_n\vert^2)^{m/2} d\eta_v d\eta_n\\
&= & \int (1+\vert\eta_v\vert^4)^{m/2} \left(1+\left\vert \frac{\eta_n}{(1+\vert\eta_v\vert^4)^{1/2}}\right\vert^2\right)^{m/2} d\eta_v d\eta_n\\
& = & \int (1+\vert\eta_v\vert^4)^{m/2+n/2)} d\eta_v \; \times \; \int (1+\vert\eta_n\vert^2)^{m/2} d\eta_n.
\end{eqnarray*}
Therefore, it converges if and only if $-2m-2n > v$ and $-m> n$, i.e. $m < -(n + v/2)$. 
If we also assume that $\ell <-p/2$, we have 
$$
2m+2\ell+v+2n < -(2n+v)-p + v + 2n = -p.
$$
Therefore, $\dd \int (1+\vert\zeta\vert)^{2m+2\ell+v+2n} d\zeta < +\infty$. 

If  $m<-(v+2n)$ and $\ell < -p$, then classical arguments show that there exist $P_1$ and $P_2$ in ${\Psi'}_c^{m_1, \ell_1} ({\what M}, \what\maF\subset\what\maF'; \what\maE)$ such that $P=P_1P_2$, with $m_1<-(v+2n)/2$ and $ \ell_1 <-p/2$. This can be justified by reducing to the local picture again and by using the powers, as pseudodifferential operators in our calculus, of the local Lapacian operators to get $L^2$-invertible operators of any bi-order $(m', \ell')$ as follows. In local coordinates we may consider the operators
$$
A_{m', \ell'} := (I+\Delta_{\R^{p+v}}^2 + \Delta_{\R^n})^{m'/4} \; (I+\Delta_{\R^p})^{\ell'/2} \;\; \in {\Psi'}^{m', \ell'}.
$$
Since $P$ is compactly supported, we can in particular find a smooth compactly supported function $\varphi$  on $\what M$ such that
$$
P \; =  P \, M_{\varphi}.
$$
Moreover, we can find a smooth compactly supported function $\psi$ on $\what M$ such that $P\, A_{m', \ell'} = P\, A_{m', \ell'}\, M_\psi$. Now we can write 
$$
P = (P\; A_{m', \ell'}) \; (M_\psi \; A_{-m', -\ell'} \; M_\varphi).
$$
Choosing appropriately $m'$ and $\ell'$, we obtain the claimed decomposition $P=P_1P_2$. Therefore, using the corresponding bounded operators on the Hilbert space $L^2 ({\what M}; \what\maE)$, we see that $P$ is the composition of two Hilbert-Schmidt operators, and hence that $P$ is a trace-class operator.
\end{proof}

\begin{corollary}
If $P\in {\Psi'}_c^{m,\ell} ({\what M}, \what\maF\subset\what\maF'; \what\maE)$ is compactly supported with $m < 0$ and $\ell < 0$, then $P$ extends to a compact operator on $L^2  ({\what M}; \what\maE)$. 
\end{corollary}

\begin{proof}
Using Remark \ref{Comparison} of the Appendix, we deduce that $P$ extends to a bounded operator on the Hilbert space  $L^2  ({\what M}; \what\maE)$. Moreover, the operator $(P^*P)^N$ belongs to ${\Psi'}_c^{2mN, 2\ell N} ({\what M}, \what\maF\subset\what\maF'; \what\maE)$ and hence for $N$ large enough and using Lemma \ref{TraceClass}, the operator $(P^*P)^N$ is a bounded trace-class operator which is therefore a compact operator. Using the spectral theorem together with the polar decomposition shows that $P$ is itself a compact operator on the Hilbert space $L^2  ({\what M}; \what\maE)$. 
\end{proof}

The following corollary will be used in the sequel. 

\begin{corollary}\label{Traceability}\
Assume that $P\in {\Psi'}^{m,-\infty} ({\what M}, \what\maF\subset\what\maF'; \what\maE)$. Then
\begin{enumerate}
\item If $m\leq 0$ then the operator $P$ extends to a bounded operator on $L^2 ({\what M}; \what\maE)$.
\item If $m < -(v+2n)$ (resp. $m < 0$) and $P$ is compactly supported, then the bounded extension of $P$ to  $L^2 ({\what M}; \what\maE)$ is a trace-class operator (resp. a compact operator). 
\end{enumerate}
\end{corollary}

\subsection{Transversely elliptic Connes-Moscovici operators}

Recall that the Hermitian bundle $\what\maE$ is assumed to be holonomy equivariant with respect to the foliation $\what\maF$ with the unitary action $W_{\what\gamma}: \what\maE_{s(\what\gamma)}\to \what\maE_{r(\what\gamma)}$ for any $\what\gamma\in \what\maG$.  
Then the representation  $\pi$ of the algebra $C_u^\infty (\what\maG)$ of smooth uniformly supported functions for the monodromy groupoid $\what\maG$ of $({\what M}, \what\maF)$ is  involutive.   For $\xi\in L^2({\what M}; \what\maE)$, it is given by
$$
[\pi (k) \xi] (\hm) := \int_{\what\gamma \in  \what\maG^{\hm}}  k(\what\gamma) \; W_{\what\gamma} [\xi (s(\what\gamma)] \; d\what\eta^{\hm} (\what\gamma).
$$

\begin{lemma}\label{pi(k)}
For any $k\in C_u^\infty (\what\maG)$, the operator $\pi (k)$ belongs to ${\Psi'}^{0, -\infty} ({\what M}, \what\maF\subset \what\maF'; \what\maE)$. Moreover, if $k\in C_c^\infty (\what\maG)$ is compactly supported then $\pi (k)$ belongs to ${\Psi'}_c^{0, -\infty} ({\what M}, \what\maF\subset \what\maF'; \what\maE)$.
\end{lemma}

\begin{proof}
As above, we may  reduce to the local picture and  assume that the support of $k$ is contained in an open set $(\hV, \what\gamma_0, \hV') \simeq \I^p\times \I^p\times \I^q$, where $\hV$ and $\hV'$ are charts for the foliation in ${\what M}$ and $\what\gamma_0\in \what\maG_{\hV}^{\hV'}$, \cite{ConnesIntegration}. We denote by $\hm_0=s(\what\gamma_0)\in \hV$ and $\hm'_0=r(\what\gamma_0)\in \hV'$ the source and range of the class $\what\gamma_0$. We can also assume that the bundle $\what\maE$ is trivilized as $\hV\times \C^a$ over $\hV$. Using the coordinates  $(x', x, y)\in \I^p\times \I^p\times \I^q$, where $(x,y) \in \hV$ and $(x',y) \in \hV'$, $\pi (k) : C_c^\infty (\hV, \C^a) \rightarrow C_c^\infty (\hV', \C^a)$ is given by
$$
\pi (k) (u) (x',y) = \int_{\I^p} k(x', x, y)  u(x,y) dx.
$$
Viewing the smooth uniformly  supported and smoothly bounded function $k$ on $(V, \gamma_0, V')$ as a function on $\simeq \I^p\times \I^p\times \I^q$, we see that it satisfies the estimate for a symbol of class $(0, -\infty)$ (it does not depend on the convector variable $\eta$).  For fixed $x\in \I^p$, and writing
$$
u(x,y)=\frac{1}{(2\pi)^q} \int_{\I^q\times \R^q} u(x, y') e^{i(y-y')\eta} dy' d\eta,
$$
we see that $\pi (k)$ is a pseudo differential operator of class ${\Psi'}^{(0,-\infty)}$ and in fact also $\Psi^{0, -\infty}$. If $k$ is compactly supported then obviously $\pi(k)$ is a finite sum of compactly supported operators of class  ${\Psi'}^{(0,-\infty)}$ and hence is compactly supported. 
\end{proof}

The above representation $\pi$ will only be  used for $k\in \maA=C_c^\infty (\what\maG)$.
The restriction of the principal symbol of a uniformly transversely elliptic operator to some punctured conic neighborhood of the cotransverse subbundle $\what\nu^*$  to $\what\maF$ as before, is invertible with uniformly bounded inverse. Applying the  parametrix construction (see \cite{Hormander}), we thus get

\begin{theorem}\label{Parametrix}\
Let $P\in {\Psi'} ^\ell({\what M}, \what\maF'; \what\maE)$ be a uniformly supported pseudodifferential operator which is uniformly transversely elliptic with respect to $\what\maF$. Then there exists $Q \in {\Psi'} ^{-\ell} ({\what M}, \what\maF'; \what\maE)$  such that
$$
R=I-QP \text{ and  }S=I-PQ \;\; \in  {\Psi'} ^{-\infty} (\what\nu ^*, \what\maE) \cap {\Psi'} ^0({\what M}, \what\maF'; \what\maE).
$$
\end{theorem}

\begin{proof}
This is the classical construction of the parametrix away from the characteristic variety of the pseudodifferential operator $P$. See for instance \cite{Hormander}. 
\end{proof}

We can now deduce the following.

\begin{corollary}\label{TraceParametrix}\
Let $P\in {\Psi'}^\ell ({\what M}, \what\maF'; \what\maE)$ be a  uniformly supported pseudodifferential operator from the Connes-Moscovici calculus associated with the foliation $\what\maF'$. Assume that $P$ is uniformly transversely elliptic with respect to $\what\maF$ and let $Q \in {\Psi'}^{-\ell} ({\what M}; \what\maF'; \what\maE)$ be a parametrix as in Theorem \ref{Parametrix}. Then for any $k\in C_c^\infty (\what\maG)$ the operators $\pi (k) (I-QP)$ and $\pi (k) (I-PQ)$ are trace-class operators in the Hilbert space $L^2({\what M}, \what\maE)$.
\end{corollary}

\begin{proof}
Recall that  the codimension of $\what\maF$ is $q$ and that its dimension is $p$, so $p+q= \om =\dim {\what M}$. Moreover, we have the decomposition $q=v+n$ where $v$ is the rank of $\what V\simeq T\what\maF ' /T\what\maF$ and $n=q-v$ is the rank of $\what N\simeq T\what M/T\what\maF '$. By Proposition \ref{Parametrix}, we know in particular that
$$
I-QP \text{ and  }I-PQ \;\; \in   {\Psi'}^0({\what M}; \what\maE) \cap {\Psi'}^{-n-2v-1} (\what\nu ^*, \what\maE).
$$
On the other hand, by Proposition \ref{TransversalOrder}, we deduce that 
$$
I-QP \text{ and  }I-PQ \;\; \in   {\Psi'}^{-(n+2v)-1, n+2v+1} ({\what M}, \what\maF\subset \what\maF' ; \what\maE).
$$
But for any $k\in C_c^\infty (\what\maG)$ we proved in Lemma \ref{pi(k)} that $\pi (k) \in {\Psi'}_c^{0, -\infty} ({\what M}, \what\maF\subset\what\maF' ; \what\maE)$. Therefore, 
$$
\pi (k) \in {\Psi'}_c^{0, -(p+v+2n)-2} ({\what M}, \what\maF \subset \what\maF'; \what\maE).
$$
As a consequence, we obtain
$$
\pi (k) (I-QP) \text{ and  }\pi (k) (I-PQ) \;\; \in   {\Psi'}_c^{-(v+2n)-1, -p-1} ({\what M}, \what\maF\subset \what\maF' ; \what\maE).
$$
The proof is completed using Lemma  \ref{TraceClass}.
\end{proof}

\begin{proposition}
Assume that $P\in {\Psi'}^1 ({\what M}, \what\maF'; \what\maE)$ is a uniformly supported pseudodifferential operator from the Connes-Moscovici calculus. Assume that $P$ is uniformly transversely elliptic and that it induces an invertible operator on $L^2 ({\what M}, \what{\maE})$ (injective with dense range and bounded inverse). Then for any $k\in C_c^\infty (\what\maG)$, the operator $\pi (k) P^{-1}$ is a compact operator on  the Hilbert space $L^2({\what M}, \what\maE)$. More precisely, it belongs to the Schatten ideal $\maL^r(L^2({\what M}, \what\maE))$ for any $r > v+2n$.
\end{proposition}

\begin{proof}
Let $Q \in {\Psi'}^{-1} ({\what M}; \what\maF'; \what\maE)$ be a parametrix of $P$ as in Proposition \ref{Parametrix}.  Then we get
$$
\pi (k) P^{-1} = \pi (k) R P^{-1} + \pi (k) Q \text{ where } R=I-QP.
$$
From the previous corollary, we know that $\pi (k) R $ (and hence also $\pi (k) R P^{-1}$) is a bounded trace class operator. 

Since $Q\in {\Psi'}^{-1, 0} ({\what M}, \what\maF\subset \what\maF'; \what\maE)$, we have $\pi (k) Q\in {\Psi'}^{-1, -\infty} ({\what M}, \what\maF\subset \what\maF'; \what\maE)$. Applying  (2) of Proposition \ref{Traceability}, we deduce that  $\pi (k) Q$ extends to a bounded operator on $L^2({\what M}, \what\maE)$ which  belongs to the claimed Schatten ideal. In particular, it is compact.
\end{proof}

\begin{corollary}\label{Resolvent}
Let ${\what D} \in {\Psi'}^{1} ({\what M}, \what\maF'; \what\maE)$ be a uniformly supported uniformly transversely elliptic pseudodifferential operator from the Connes-Moscovici calculus as before. Assume that ${\what D}$  induces an essentially self-adjoint operator on the Hilbert space  $L^2({\what M}, \what\maE)$. Then for any $k\in C_c^\infty(\what\maG)$, the operator $\pi_{\what\maE} (k) (\what{D}+i)^{-1}$ belongs to the Schatten ideal $\maL^r(L^2({\what M}, \what\maE))$ for any $r > v+2n$.  In particular,  it is  a compact operator.
\end{corollary}

\begin{proof}
Simply apply the previous proposition to $P=D+i$.
\end{proof}

\subsection{The CM operator for (strongly) Riemannian bifoliations}\label{CM-operator}

An important example of a uniformly supported uniformly transversely elliptic pseudodifferential operator satisfying the assumptions of  Corollary \ref{Resolvent}  is the transverse signature operator associated with Riemannian bifoliations. In this case, an interesting transversely hypo-elliptic Dirac-type operator is given by the Connes-Moscovici construction that can be adapted to foliations as we now explain.   See \cite{CM95}.   In particular, we return to the general situation of smooth bounded geometry bifoliations.

\begin{definition}
A smooth bifoliation $({\what M}, \what\maF\subset\what\maF')$  is a Riemannian bifoliation if there exists a  metric $g$ on the transverse bundle $\nu_{\what\maF} = T{\what M}/T{\what\maF}$ such that
\begin{enumerate}
\item The restriction of $g$ to the subbundle $T\what\maF'/ T\what\maF$ is a holonomy invariant metric.
\item The induced metric on the quotient bundle $T{\what M}/T{\what\maF}'$ is a holonomy invariant metric.
\end{enumerate}
The bifoliation $({\what M}, \what\maF\subset\what\maF')$ will be called strongly Riemannian if it is Riemannian and there exists an integrable subbundle $\what{V}$ which is such that $T\what\maF'=T\what\maF\oplus \what{V}$ and the  metric induced from that on $T\what\maF'/ T\what\maF$ is holonomy invariant. 
\end{definition}

In the above definition, holonomy invariance is  understood with respect to the foliation $\what\maF$. 
It means that   the action of the holonomy transformation associated with  some $\gamma\in {\what\maG}_x^y$,  is given  by a matrix of the form
$$
\left(
\begin{array}{cc}
\psi^{11} & 0 \\
\psi^{21} & \psi^{22}\\
\end{array}
\right)
$$
where $\psi^{22}$ is an orthogonal transformation from ${\what V}_x$ to ${\what V}_y$ and $\psi^{11}$
corresponds to the induced action on ${\what N}=T{\what M}/T{\what \maF}'$,  and is also orthogonal.

The vector bundle ${\what V}\oplus {\what N}$ is thus (non-canonically) isomorphic, as a ${\what\maG}$-equivariant vector bundle, to the transverse bundle $\nu_{\what\maF}=T{\what M}/T{\what\maF}$. 

Suppose that $(\hM, \what\maF\subset \what\maF')$ is a strongly Riemannian smooth bifoliation  as before with the integrable bundle $\what{V}$. Assume that the bundles $\what{V}$ and $\what{N}$ are  oriented and even dimensional. We fix adapted Riemannian structures on $T\hM$, $\what{V}$ and $\what{N}$ and hence adapted Hermitian structures on all the exterior powers of these bundles. By adapted, we mean that they are holonomy invariant.  So, we have two transverse foliations $\what\maF$ and $\what{V}$ whose direct sum $T\what\maF'$ is still integrable with the previously explained conditions on the holonomy action of $\what\maF$. 
 We thus have the Riemannian gradings $\gamma_{\what{V}}$ and $\gamma_{\what{N}}$ on the exterior powers of $\what{V}^*$ and $\what{N}^*$ respectively as well as a volume element which trivializes the bundle $\Lambda^v\what{V}^*\otimes \Lambda^n\what{N}^*\simeq \Lambda^q\nu_{\what\maF}^*$. 

Take for $\what\maE$ the holonomy equivariant (w.r.t. $\what\maF$) Hermitian bundle 
$$
\what\maE\; = \; \Lambda^\bullet (\what{V}^*\otimes \C) \otimes \Lambda^\bullet (\nu_{\what\maF'}^*\otimes \C).
$$
The de Rham differential along the leaves of the foliation $\what{V}$  is denoted $d_{\what{V}}$. Since $\nu_{\what\maF'}$ and $\Lambda^\bullet (\nu_{\what\maF'}^*\otimes \C)$ are flat bundles along the leaves of $\what\maF'$, they are also flat along the leaves of $\what{V}$.  So the differential $d_{\what{V}}$ is a well defined first order differential operator with $d_{\what{V}}^2=0$ acting on smooth sections of $\what\maE$.
With respect to the bi-grading of forms 
$$
\Lambda^{a, b}=  \Lambda^a (\what{V}^*\otimes \C) \otimes \Lambda^b (\nu_{\what\maF'}^*\otimes \C),
$$
the differential $d_{\what{V}}$ is the component of bidegree $(1, 0)$, so the component which sends $\Lambda^{a, b}$ to $\Lambda^{a+1, b}$. We denote by $Q_{\what{V}}$ the second order essentially self-adjoint {\em{signature  operator}}  along the leaves of $\what{V}$.    By definition, see \cite{CM95}, this is given by
$$
Q_{\what{V}} \; := \; d_{\what{V}}\, d_{\what{V}}^*\; - \; d_{\what{V}}^*\, d_{\what{V}}.
$$
Here $d_{\what{V}}^*$ is the formal adjoint of $d_{\what{V}}$ as an operator on $\what\maH = L^2(\hM, \what\maE)$. The operator $Q_{\what{V}}$ is then  an elliptic operator along the leaves of the foliation generated by $\what{V}$.  Indeed  one can easily check that
$$
Q_{\what{V}} \; \sim \; [\gamma_{\what{V}} d_{\what{V}} \gamma_{\what{V}} \, , \, d_{\what{V}}] \; = \; \gamma_{\what{V}} d_{\what{V}} \gamma_{\what{V}} d_{\what{V}} - d_{\what{V}}\gamma_{\what{V}} d_{\what{V}} \gamma_{\what{V}} \quad \text{ up to zero-th order operators}.
$$
On the other hand, using our choice of a normal bundle $\what{N}$ to $\what\maF'$ and the isomorphism $\nu_{\what\maF'}\simeq \what{N}\subset T\hM$, we have a well defined transverse component to $\what\maF'$ of the de Rham differential on the ambient manifold $\hM$ which acts on $\what\maE$. More precisely, this is the restriction to the smooth sections of $\what\maE$ of the component of tridegree $(0, 0, 1)$ corresponding to the decomposition of forms induced by the isomorphism
$$
T\hM \simeq T\what\maF \oplus \what{V} \oplus \nu_{\what\maF'}.
$$
This latter component thus  corresponds to differentation in the transverse directions to $\what\maF'$, which we view as acting on the smooth sections of $\what\maE$.   It is denoted $d_{\what{N}}$ to indicate the dependence on the choice of $\what{N}$. The formal adjoint of $d_{\what{N}}$ as an operator on $C_c^\infty (\hM; \what\maE)$ is denoted $d_{\what{N}}^*$.    Set 
$$
Q_{\what{N}} \; := \; d_{\what{N}}\, + \,d_{\what{N}}^*,
$$
and  (see \cite{CM95})
$$
Q\; :=\; Q_{\what{V}} (-1)^{\partial_{\nu_{\what\maF'}}} \; + \; Q_{\what{N}},
$$
where $\partial_{\nu_{\what\maF'}}$ is the form degree of the $\Lambda^\bullet\nu_{\what\maF'}^*$ components. 

\begin{lemma}
The operator $Q$ belongs to the Connes-Moscovici pseudo' calculus ${\Psi'}^{2} (\hM, \what\maF'; \what\maE)$ associated with the foliation  $\what\maF'$ and is a uniformly transversely elliptic operator (with respect to the foliation $\what\maF$).
\end{lemma}

\begin{proof}
Any differential operator belongs to the Connes-Moscovici calculus although the order and principal symbols are defined in a different way. So $Q$ belongs to  ${\Psi'}^{2} (\hM, \what\maF'; \what\maE)$.   Note that the principal symbol of $Q^2$ only depends on the transverse covectors and is given by 
$$
\sigma_4 (Q^2) (\hm, \eta) = \sigma_2(Q)(\hm, \eta)^2 = \vert \eta_v\vert ^4 + \vert \eta_n\vert^2.  
$$
Since all our geometric data are $C^\infty$-bounded, we deduce that $Q$ is uniformly transversely elliptic.
\end{proof}

Since the manifold $\hM$ as well as the foliations have bounded geometry, classical arguments \`a la Chernoff \cite{Chernoff} (see also \cite{Shubin92} and \cite{Shubin87}) show that the operator $Q$ has uniformly bounded coefficients.  Moreover by classical arguments, the operator $Q^2$ can be extended to a non-negative self-adjoint operator that we still denote by $Q^2$ for simplicity. 

\begin{definition}\cite{CM95}\
The CM transverse signature operator $D^{\sign}$ is defined by the spectral formula
$$
D^{\sign}\; := \; \frac{1}{\pi\sqrt{2}}\int_0^\infty  Q (Q^2+\lambda)^{-1} \frac{d\lambda}{ \lambda^{1/4}}.
$$
\end{definition}
\noindent
Following the same proof as in \cite{CM95}, one can show that the operator $D^{\sign}$ belongs to ${\Psi'}^1 (\hM, \what\maF'; \what\maE)$ and is a uniformly transversely elliptic operator (with the holonomy invariant principal symbol). 

The above constructions on (strongly) Riemannian bifoliations  actually allow one to deduce  all the important topological results  on general smooth foliations of bounded geometry. We now explain the idea behind this reduction construction, which is due to Alain Connes.   It goes back to his reduction method from type III to type II von Neumann algebras as intensively exploited in his breakthrough results on the classification problem of type III von Neumann algebras.  See \cite{ConnesTransverse}.  Fix a smooth foliation $(M, \maF)$ and denote by ${\what M}_x$,  $x\in M$,  the set of positive definite quadratic forms on the transverse bundle $\nu_x=T_xM/T_x\maF$. If $q$ is the codimension of the foliation $\maF$, then ${\what M}_x$ can be identified with the homogeneous space $GL^+_q(\R)/SO_q(\R)\simeq \R\times SL_q(\R)/SO_q(\R)$. The tangent space to $GL^+_q(\R)/SO_q(\R)$ can in turn be easily identified with the space  ${\mathcal S}$ of symmetric $q$-matrices. There is a $GL^+_q$-invariant metric on the manifold $GL^+_q(\R)/SO_q(\R)$ given for $A\in GL^+_q (\R)$ and $(B,C)\in {\mathcal S}^2$ by
$$
g_{[A]}(B,C)=<A^{-1}B,A^{-1}C>_{HS} \text{ with } <\bullet, \bullet>_{HS} \text{ the Hilbert-Schmidt scalar product.}
$$
The family ${\what M}= ({\what M}_x)_{x\in M}$ is then a smooth fibration over $M$.  Indeed $\pi:{\what M}\to M$ is  the fiber bundle of Euclidean metrics on the transverse bundle $\nu=TM/T\maF$.   So, the fibers of this fibration 
are contractible manifolds of nonpositive sectional curvature. In fact, it is easy to see that  the fibers of ${\what M}$
are all diffeomorphic to some $\R^N$.  The group $GL^+_q (\R)$ acts on the left on ${\what M}$ by fibre-preserving diffeomorphisms. 

There is a lift of the foliation $\maF$ to ${\what M}$, of the same dimension, denoted ${\what\maF}$.  If $(x, [A]) \in {\what M}_x$, its leaf consists of all $(y, [B]) \in {\what M}$ where $y \in L_x$, the leaf of $x$,  and there is a path $\gamma$ in $L_x$ starting at $x$ and ending at $y$, so that the induced action of the holonomy along $\gamma$ on the fibers of ${\what M}$ takes $[A]$ to $[B]$.

There is a second foliation, denoted ${\what\maF}'$, containing ${\what\maF}$, whose leaves are the  inverse images of the leaves of $F$ under $\pi:{\what M}\to M$.  It may also be described as the foliation associated to the subbundle $T{\what\maF} \otimes \what{V}$ of $T{\what M}$,  which is the kernel of $p\circ \pi_*$.   Here $p:TM\to \nu$ is the projection to  $\nu$, and $\what{V} = \ker(\pi_*) \subset T{\what M}$.  

\begin{lemma}\label{ConnesFibration}\cite{ConnesTransverse}\
Let $(M, \maF)$ be a smooth bounded-geometry  foliation and ${\what M}$  the fiber bundle of all the Euclidean metrics on $\nu$, the normal bundle of $\maF$. Then $\hM$ is endowed with a strongly Riemannian bifoliation consisting of  ${\what\maF}$ and $\what\maF'$.
 \end{lemma}

\begin{proof}\
The normal bundle to ${\what\maF}$ is isomorphic to the bundle
$$
\nu_{\what\maF} \cong \pi^*\nu  \oplus {\what V}.
$$
The Hilbert-Schmidt scalar products on the fibres of ${\what M}$  give a smooth metric on the bundle ${\what V}$.  
On the other hand any element $[A]$ of ${\what M}_x$ is itself a metric on $\nu_x$ and hence yields a scalar product on $\nu_x$.  The vertical bundle ${\what V}$ is clearly integrable and strictly transverse to the foliation. It is also clear that ${\what V}$ is preserved by the holonomy action of $\what\maF$ and that the direct sum subbundle $T\what \maF\oplus {\what V}$ of the tangent bundle $T{\what M}$ is integrable and generates the foliation $\what \maF'$. Therefore,   $({\what M}, \what\maF \subset \what\maF')$ is a smooth bifoliation. Moreover, an easy inspection shows that $\what M$ as well as its foliations $\what\maF$ and $\what\maF'$ do have bounded geometry since  $(M, \maF)$ has bounded geometry.

The action of the holonomy transformations of $\what\maF$ on the transverse bundle  $\nu_{\what\maF}$ can be described as the action induced on the metrics over $\nu$ from that on $\nu$. More
precisely, any holonomy transformation $\psi$  corresponds to a holonomy transformation $\varphi$ in $(M, \maF)$. So, if $g$ is a metric  on $\nu_x$  then the holonomy transformation $\psi$ associated with an element of $\what\maG_{g}^{g'}$  acts by
$$
\psi(g)(X,Y)= g(\varphi_*^{-1}(X),\varphi_*^{-1}(Y)),\quad \forall (X,Y)\in \nu_x^2,
$$
where $\varphi$ is the holonomy transformation associated with the projected element in $\maG_{\pi(g)}^{\pi(g')}$. 

The action on the vertical bundle is thus given by the differential of the above action and we have $\forall X\in T{\what V}_g$ and with respect to the metric defined on ${\what V}$,
$$
\|\psi_*(X)\|^2=\|D_x(\varphi)B_g^{-1}XD_x(\varphi)^{-1}\|_{HS}=\| B_g^{-1}X\|_{HS}=\|X\|^2,
$$
where $B_g$ is an element of $GL^+_q(\R)$ representing a class corresponding to $g$.
Whence the action of $\psi_*$ on the transverse bundle to $({\what M},\what\maF)$ decomposes with respect to the vertical bundle, and any supplementary bundle, in the required triangular form
$$
\left( 
\begin{array}{cc}
\psi^{11} & 0 \\
\psi^{21} & \psi^{22}\\
\end{array}
 \right)
$$
where $\psi^{22}$ is an orthogonal transformation from $TV_g$ to $TV_{g'}$. That $\psi^{11}$
is also isometric is in fact obvious and is a tautology.   In particular, if $\pi(g) = x$ and $Y\in  (\pi^* \nu)_{g} \cong \nu_x$, then  $\psi_*(Y)=(\psi^{11}(Y),\psi^{21}(Y))$ and we have
$$
\|\psi^{11}(Y)\|^2=\|\pi_*(D_g(\psi)(Y))\|^2=\psi(g)(\pi_*(D_g(\psi)(Y)),\pi_*(D_g(\psi)(Y))).
$$
But $\pi_*\circ D_g(\psi)=D_x(\varphi)\circ \pi_*$ and we finally obtain
$$
\|\psi^{11}(Y)\|^2=\psi(g)(D_x(\varphi)(\pi_*(Y)),D_x(\varphi)(\pi_*(Y)))
$$
$$
=
g((D_x(\varphi)^{-1}\circ D_x(\varphi)\circ \pi_*)(Y),(D_x(\varphi)^{-1}\circ D_x(\varphi)\circ \pi_*)(Y))=\|Y\|^2.
$$
\end{proof}

\section{The NCG of  proper bifoliated  actions} \label{ProperActions}

\subsection{Algebras associated with  (bi)foliated actions}
 
Let $(\what M, \what\maF\subset  \what\maF')$ and $\what\maE$ be as above, of bounded geometry, together with a right  action of a countable discrete  group $\Gamma$ by $T\what\maF$-preserving diffeomorphisms which also preserves $T\what\maF'$.  So  $\Gamma$ acts by diffeomorphisms of $\what M$ which send leaves of $\what\maF$ to leaves of $\what\maF$ and also leaves of $\what\maF'$ to leaves of $\what\maF'$. In particular,  all leaves (of $\what\maF$ as well as of $\what\maF'$) in a given orbit are diffeomorphic.  We  will assume that $\hM$ is endowed with a $\Gamma$-invariant Riemannian metric and that there is a $\Gamma$-invariant Hermitian structure on $\what\maE$. We then consider  the space $L^2 (\hM, \what\maE)$ of $L^2$-sections of $\what\maE$, which is  defined with respect to these $\Gamma$-invariant structures, so that it furnishes a unitary representation of $\Gamma$.

 The normal bundle $\what\nu$ to the foliation $\what\maF$, 
will  be identified with the orthogonal bundle to $T\what\maF$ with respect to the $\Gamma$-invariant metric, and is an example of bounded geometry $\Gamma$-equivariant vector bundle over $\hM$ and it is also of bounded geometry over each leaf. This normal bundle is  endowed with the linear action of the holonomy pseudogroup of the foliation $\what\maF$ which commutes with the  action of $\Gamma$. The same properties hold for all functorially associated bundles such as its dual bundle $\what\nu^*$ and its exterior powers. Also, when the normal  bundle $\what\nu$ is $K$-oriented with a $\Gamma$-invariant spin$^c$ structure, the associated spinor bundle is endowed with the action of the holonomy pseudo group of the foliation $\what\maF$ and is $\Gamma$-equivariant again with commuting actions. This spinor bundle is then of bounded geometry over $\hM$ as well as over all leaves.

We fix  a Hausdorff Lie groupoid $\what\maG$ which generates the foliation $(\hM, \what\maF)$ and which is a quotient of the monodromy groupoid and a covering of the holonomy groupoid. For simplicity, the reader may assume that the groupoid $\what\maG$ coincides with the holonomy groupoid and that this latter is Hausdorff. The classical convolution $*$-algebra associated with the groupoid $\what\maG$ is $\maA:=C_c^\infty (\what\maG)$.  Recall that we also have the larger algebra $C_u^\infty (\what\maG)$.  Observe that the group $\Gamma$ also acts on $\what\maG$ by groupoid isomorphisms so that the source and range maps $s$ and $r$ are $\Gamma$-equivariant. Indeed, $\Gamma$ acts on the monodromy groupoid and this action descends to an action on the holonomy groupoid since it obviously respects the holonomy equivalence relation. We denote by $\what\maG\rtimes \Gamma$ the crossed product groupoid which is obviously a Lie groupoid with the same unit space $\hM$ and with the rules
$$
s (\what\gamma, g) = s (\what\gamma) g, \quad r (\what\gamma, g) = r (\what\gamma), \quad\text{ and } \quad (\what\gamma, g) (\what\gamma ', g') = (\what\gamma (\what\gamma' g^{-1}), gg') \: \text{ if }\: r(\what\gamma ') = s (\what\gamma)g.
$$
The fibers of this Lie groupoid are the cartesian products of the fibers of $\what\maG$ with the group $\Gamma$ and are hence endowed with the invariant Haar system $\what\eta\otimes \delta$ where $\what\eta$ is the $\what\maG$-invariant Haar system (given by lifting  a Lebesgue measure on the leaves of $(\hM, \what\maF)$) and $\delta$ is the $\Gamma$-invariant counting measure on $\Gamma$. The convolution $*$-algebra of smooth compactly supported functions on the groupoid  $\what\maG\rtimes \Gamma$ is denoted $\maB$. 
 Given $\varphi, \psi \in \maB$, recall that
$$
(\varphi \psi) (g; \what\gamma) \,\, := \sum_{g'\in \Gamma} \int_{\what\maG^{{\what r} (\what\gamma)}} \varphi (g'; {\what\gamma} ') \psi ({g'}^{-1} g; (({\what\gamma}')^{-1} \what\gamma)g') d\what\eta^{{\what r} (\what\gamma)} \quad\text{ and } \quad \varphi^* (g; \what\gamma)\,\, := \,\,{\overline{ \varphi(g^{-1}; {\what\gamma}^{-1} g)}}.
$$
The algebra $\maB$ may be identified with  the convolution algebra of finitely supported functions on $\Gamma$ with values in the convolution $*$-algebra $\maA$. Then we can rewrite the algebra structure as
$$
(\varphi*\psi)(g) \,\,=\sum_{g_1g_2=g} \varphi (g_1) * g_1 (\psi (g_2)) \quad\text{ and }\quad \varphi^* (g) = [g \varphi (g^{-1})]^*.
$$
Notice that all these rules make sense for the larger algebra $C^\infty_u(\what\maG\rtimes \Gamma)$  of finitely supported functions on $\Gamma$ with values in $C_u^\infty (\what\maG)$.

An interesting situation occurs for proper  actions of countable discrete groups. In this case, the quotient $M=\hM/\Gamma$ is Hausdorff and  we call any such pair $(\hM, \Gamma)$ a proper smooth presentation for the  space $M$. 

Since $\Gamma$ acts by isometries of $\hM$, preserves the bifoliation $\what\maF\subset \what\maF'$, and preserves the Hermitian structure on $\what\maE$, it acts by filtration-preserving automorphisms on the spaces ${\Psi'}^{m,\ell} (\hM, \what\maF\subset  \what\maF'; \what\maE)$ and ${\Psi'}_c^{m,\ell} (\hM, \what\maF\subset  \what\maF'; \what\maE)$. Moreover, this is an action by $*$-automorphisms of ${\Psi'}^{\infty,\infty} (\hM, \what\maF\subset  \what\maF'; \what\maE)$.  In particular, $\Gamma$ also acts on each ${\Psi'}^{m,-\infty} (\hM, \what\maF\subset  \what\maF'; \what\maE)$ and ${\Psi'}_c^{m,-\infty} (\hM, \what\maF\subset  \what\maF'; \what\maE)$. For simplicity, we shall concentrate on the action on the compactly supported operators and  we introduce the (algebraic) crossed product class ${\Psi'}_c^{m,\ell} (\hM, \what\maF\subset  \what\maF'; \what\maE)\rtimes \Gamma$. This  is the space of finitely supported functions on $\Gamma$ which take values in ${\Psi'}_c^{m,\ell} (\hM, \what\maF\subset  \what\maF'; \what\maE)$. When $\Gamma$ acts properly and cocompactly on $\hM$,  this is the natural space of operators for our study here, but in general one might need a slightly larger algebra of vertically compactly supported operators. For simplicity, we shall avoid this discussion here and leave the easy extension to the interested reader. 
Composition of such elements involves the action of $\Gamma$ on ${\Psi'}_c^{\infty,\infty} (\hM, \what\maF\subset  \what\maF'; \what\maE)$. More precisely, 
$$
(T\circ S) (g):= \sum_{kh=g} [h^{-1} T(k)] \circ S (h), \text{ for } T, S \in {\Psi'}_c^{\infty,\infty} (\hM, \what\maF\subset  \what\maF'; \what\maE)\rtimes \Gamma. 
$$
Indeed,  the elements of ${\Psi'}_c^{m,\ell} (\hM, \what\maF\subset  \what\maF'; \what\maE)\rtimes \Gamma$ act on $C_c(\Gamma, C_c^\infty (\hM, \what\maE))$ as 
$$
T (\xi) (g)  := \sum _{g_1g_2=g} [T (g_1) \circ U_{g_1}] (\xi (g_2)),
$$
where $U_g$ is the unitary on $L^2 (\hM, \what\maE)$ corresponding to the action of $g\in \Gamma$, and which preserves $C_c^\infty (\hM, \what\maE)$. The action of $U_g$, for $\xi\in L^2 (\hM, \what\maE)$, is given by 
$$
[U_g \xi] (\hm) := g \xi(\hm g), \text{ where we use the action of $\Gamma$ on $\what\maE$.}
$$
Similar definitions give the spaces
$$
{\Psi'}_c^{m,-\infty} (\hM, \what\maF\subset  \what\maF';  \what\maE)\rtimes \Gamma,  \quad{\Psi'}_c^{\infty,-\infty} (\hM, \what\maF\subset  \what\maF'; \what\maE)\rtimes \Gamma 
\quad \text{ and } \quad  \Psi_c^{-\infty} (\hM; \what\maE)\rtimes \Gamma.
$$

Denote by $\what{\oH}_\Gamma^{m,\ell}$ the space $\ell^2(\Gamma, {\oH'}^{m,\ell}(\hM, \what\maF\subset  \what\maF'; \what\maE))$ of $\ell^2$ functions on $\Gamma$ with values in the bigraded Sobolev space ${\oH'}^{m,\ell}(\hM, \what\maF\subset  \what\maF'; \what\maE)$ associated with the  bifoliation $\what\maF\subset \what\maF'$ (see in Appendix \ref{Sobolev}). Set 
$$\Psi_{c,\Gamma}^{m,\ell} \,\, = \,\, {\Psi'}_c^{m,\ell} (\hM, \what\maF\subset  \what\maF'; \what\maE)\rtimes \Gamma
\quad \text{ and } \quad \Psi_{\Gamma}^{m,\ell} \,\, = \,\, {\Psi'}^{m,\ell} (\hM, \what\maF\subset  \what\maF'; \what\maE)\rtimes \Gamma,
$$  where the crossed products are algebraic.

\begin{proposition}\label{Bounded}
\begin{enumerate}
\item If $A\in \Psi_{\Gamma}^{m_1,\ell_1}$ and $B \in \Psi_{\Gamma}^{m_2,\ell_2}$, then $A\circ B\in  \Psi_{\Gamma}^{m_1+m_2,\ell_1+\ell_2}$.   Moreover, if $A$ or $B$ belongs to $\Psi_{c,\Gamma}^{*,*}$ then so does the composition. 
\item If $A\in \Psi_{\Gamma}^{m,\ell}$ then for any $(s,k)\in \R^2$, the operator $A$ extends to a bounded operator 
$$
A_{s,k}:\what\oH_\Gamma^{s,k}  \longrightarrow \what\oH_\Gamma^{s-m,k-\ell}.
$$
\end{enumerate}
\end{proposition}

\begin{proof}
We have already explained why  the isometric action of $\Gamma$ automatically preserves ${\Psi'}^{\infty, \infty}(\hM, \what\maF\subset  \what\maF'; \what\maE)$ and ${\Psi'}_c^{\infty, \infty}(\hM, \what\maF\subset  \what\maF'; \what\maE)$ and their bifiltrations. Let $A\in \Psi_{\Gamma}^{m_1,\ell_1}$ and $B \in \Psi_{\Gamma}^{m_2,\ell_2}$.   Then for any $g_1, g_2\in \Gamma$, we deduce that 
$$
A(g_1) \in {\Psi'}^{m_1,\ell_1} (\hM, \what\maF\subset \what\maF'; \what\maE) \quad \text{ and } \quad g_1 B(g_2) \in {\Psi'}^{m_2,\ell_2} (\hM, \what\maF\subset \what\maF'; \what\maE).
$$
Therefore, $A(g_1)\circ g_1 B(g_2)\in {\Psi'}^{m_1+m_2,\ell_1+\ell_2} (\hM, \what\maF\subset\what\maF'; \what\maE)$. Since for any $g\in \Gamma$, the operator $(AB)(g)$ is a finite sum of operators of this kind, we get that $(AB)(g)\in {\Psi'}^{m_1+m_2,\ell_1+\ell_2} (\hM, \what\maF\subset\what\maF'; \what\maE)$. But $AB$ is then obviously finitely supported, hence it belongs to $  \Psi_{\Gamma}^{m_1+m_2,\ell_1+\ell_2}$. The second statement of (1) is also clear since composition of a compactly supported operator with a uniformly supported operator is a compactly supported operator. 

For (2),  fix $A\in \Psi_{\Gamma}^{m,\ell}$ and $s,k$ as above. We know that for any $g\in \Gamma$, $\vert\vert A(g)\vert\vert_{s,k}^{s-m, k-\ell} < +\infty$.  Set $C_{s,k}(A) := \sum_{g\in \Gamma} \vert\vert A(g)\vert\vert_{s,k}^{s-m, k-\ell}$ where the sum is of course finite.  Then a straightforward estimate shows that for any $\xi\in \what\maH_\Gamma^{s, k}$
$$
\vert\vert A(\xi)\vert\vert_{s-m, k-\ell} \leq C_{s,k}(A) \times \vert\vert \xi\vert\vert_{s,k}.
$$
\end{proof}

As explained above, if we   assume that the Hermitian structure on $\what\maE$ is $\Gamma$-invariant and that $\Gamma$ acts by isometries on $\hM$, then  the induced action $U$ of $\Gamma$ on the Hilbert space $\what\maH:=L^2(\hM, \what\maE)$ is a unitary representation. For any operator $T$ on the Hilbert space  $\what\maH$ and any $g\in \Gamma$, we denote by $g T$ the operator obtained by conjugation, i.e. $g T:= U_g T U_g^{-1}$.  Any finitely supported operator-valued function $A:\Gamma \to B(\what\maH)$ acts on the Hilbert space  of $\ell^2$ functions from $\Gamma$ to $\what\maH$, which is as usual identified with the spacial tensor product $\ell^2\Gamma\otimes \what\maH$, through the operator $\lambda (A)$ defined by
 $$
 \lambda (A) (\delta_g\otimes \xi) = \sum_{k\in \Gamma} \delta_{kg} \otimes [U_k\circ A(k)] (\xi)\quad\text{ for } g\in \Gamma\text{ and } \xi\in \what\maH.
 $$ 
As usual $\delta_g$ is the delta function at $g\in \Gamma$, that is the characteristic function of $\{g\}$.  Notice that the space $\C (\Gamma, B(\what\maH))$ of such operator-valued finitely supported functions $A$ is a unital $*$-algebra for the rules
 $$
 (AB) (g) := \sum_{kh=g} [h^{-1} A(k)] \circ B(h) \,\, \in B( \what\maH)\quad \text{ and } \quad A^* (g):= g^{-1} A(g^{-1})^*.
 $$ 
 Then, $\lambda$ is a $*$-representation in $\ell^2\Gamma\otimes \what\maH$ as can be checked easily. 
 
\begin{definition}
The weak closure of the $*$-algebra $\lambda \left(\C (\Gamma, B(\what\maH))\right)$ in $B(\what\maH\otimes \ell^2\Gamma)$ is denoted $\maN$.  
\end{definition}

 $\maN$ is the crossed product von Neumann algebra which will be used in the sequel. 

The group $\Gamma$ acts on $\ell^2\Gamma\otimes \what\maH$ through the representation 
$$
V:= L\otimes U, \,\,  \text{ where } L \text{ is the left regular representation},
$$
and
$$
V_k (\xi) (g) := U_k (\xi (k^{-1}g)) \quad \text{ or } \quad V_k (\xi) (g; \hm) :=  k\cdot \xi (k^{-1} g; \hm k)  \in E_{\hm}.
$$
On the other hand, any operator $T\in B(\what\maH)$ acts on $\ell^2\Gamma\otimes \what \maH$ as $I\otimes T$. 
The von Neumann algebra  $\maN$ coincides by definition with the weak closure of the $*$-algebra generated in $B(\ell^2\Gamma \otimes \what \maH)$ by the operators $I\otimes T$ and all unitaries $V_g$ for $T\in B(\what\maH)$ and $g \in \Gamma$. In particular,  all operators in $\maN$ commute with the right representation $R$ defined on $\what\maH$ by $R_g (\xi) (k, \hm) = \xi (kg, \hm)$. 

Denote by  $W$  the action of $\what\maG$ given by the holonomy groupoid action on $\what\maE$.  Recall that $W_{\what\gamma}: \what\maE_{{\what s}(\what\gamma)} \to \what\maE_{{\what r}(\what\gamma)}$ is a linear isomorphism with the usual functorial properties.  
The $*$-algebra $\maB:=C_c^\infty (\what\maG\rtimes \Gamma)$ is represented in the Hilbert space $\ell^2\Gamma\otimes \what\maH$ by the formula
$$
\what\pi (\varphi) (\xi) (g, \hm) :=  \sum_{g_1\in \Gamma} \int_{\what\gamma\in \what\maG^{\hm}} \varphi (g_1, \what\gamma ) W_{\what\gamma} \left[ (U_{g_1}  \xi_{g_1^{-1}g} ) ({\what s}({\what\gamma})) \right] d\what\eta^{\hm} (\what\gamma).
$$
For any $g \in \Gamma$ and any $\what\gamma\in \what\maG$, the compatibility of the action of $\Gamma$ with the foliation means that $(g \cdot)\circ W_{\what\gamma k} = W_{\what\gamma} \circ (g\cdot )$. This representation is  obviously well defined on the bigger algebra of continuous compactly supported functions but this will not be needed here. When the holonomy action $W$ preserves the Hermitian structure of $\what\maE$, the above representation $\what\pi $ is an involutive representation. We shall describe examples of such bundles provided by bundles functorially associated  with the transverse bundle to the foliation when this latter is Riemannian, or more generally functorially associated  with the splitting of the normal bundle when this latter is almost Riemannian.  We shall in fact only need this construction in the almost Riemannian case with a bundle which is holonomy invariant (and $\Gamma$-invariant) as a Hermitian bundle, so we assume from now on that the representation $\what\pi$ is involutive. 

\begin{remark}
Recall the averaging representation $\pi$ of $C_u^\infty (\what\maG)$ in $\what\maH$ defined   for $k\in C_u^\infty (\what\maG)$ and $v\in \what\maH$ by
$$
\pi (k) (v)( \hm) := \int_{\what\maG^{\hm}} k (\what\gamma) W_{\what\gamma} v(s(\what\gamma))  d\what\eta^{\hm} (\what\gamma).
$$
This representation $\what\pi$ can be reinterpreted as 
$$
\left[\what\pi (\varphi) (\xi)\right]_g:=\sum_{g_1g_2=g} \pi (\varphi _{g_1}) (U_{g_1} \xi_{g_2}).
$$
\end{remark}

Notice that if $\varphi=f\delta_\alpha$ then $[\what\pi (\varphi) (\xi)]_g= \pi(f) \circ [V_\alpha \xi]_g$ and hence $\what\pi (\varphi)$ belongs to $\maN$. This shows that $\what\pi (\varphi)$ belongs  to $\maN$ for any $\varphi\in \maB$.  Viewed as an infinite matrix of operators in $\what\maH$, indexed by  $\Gamma\times \Gamma$, one  checks that
$$
\what\pi (\varphi)_{g',g} =  \pi (\varphi _{g'g^{-1}}) \circ U_{g'g^{-1}}, \quad  \forall g', g\in \Gamma.
$$
We actually have the following 
\begin{lemma}
For any $\varphi \in \maB=C^\infty_c (\what\maG)\rtimes \Gamma$,  the operator $\what\pi (\varphi)$ belongs to the von Neumann algebra $\maN$.
\end{lemma}

In general,  a direct computation shows that for any $T\in \maN$, we have, with the obvious notation,
$$
T_{g',g} =  T _{g'g^{-1}} \circ U_{g'g^{-1}} , \quad \forall g', g\in \Gamma.
$$

Recall that we have assumed that $\Gamma$ acts by unitaries on the Hilbert space $\what\maH$.
\begin{definition}
Let $(f_i)$ be an orthonormal basis  of the Hilbert space $\what\maH$. For any nonnegative operator $T\in \maN$,  set
$$
\TR (T):= \sum_i <T(\delta_e \otimes f_i), \delta_e\otimes f_i> \quad \in \,\, [0, +\infty].
$$
\end{definition}

It is easy to check that the functional  $\TR$ does not depend on the choice of the $(f_i)$.

\begin{proposition}
$\TR$ extends by linearity to a normal semi-finite faithful positive trace on $\maN$.  
\end{proposition}

\begin{proof}
Note that for any $T\in \maN$, the entries of the infinite matrix $(T_{\alpha, \beta})_{\alpha, \beta\in \Gamma}$, which are bounded operators in $\what\maH$, satisfy
$$
T_{\alpha, \beta} = T_{e, \beta\alpha^{-1}},
$$
and hence only depend on $\beta\alpha^{-1}$. In particular, $T_{\alpha, e}= T_{e, \alpha^{-1}}$.  In addition, if $T=S^*S$ with $S\in \maN$ then
$$
\TR (T)=\sum_{\alpha\in \Gamma} \Tr ((S_{\alpha,e})^* S_{\alpha,e}).
$$
Now $(S^*)_{e, \alpha} = (S_{\alpha,e})^*$, but $\Tr ((S_{\alpha,e})^* S_{\alpha,e}) = \Tr (S_{\alpha,e}(S_{\alpha,e})^*) = \Tr(S_{e, \alpha^{-1}}(S_{e, \alpha^{-1}})^*)$. This proves that $\TR (S^*S) \geq 0$ and that  $\TR (S^*S)=\TR (SS^*)$.  Moreover, if  $\TR(T)=0$ then we get $\Tr ((S_{\alpha,e})^* S_{\alpha,e})=0$ for any $\alpha$. Since the usual trace $\Tr$ is faithful, we deduce that $S_{\alpha,e}=0$ for any $\alpha\in \Gamma$ and hence since $S\in \maN$, $S=0$. Therefore, $\TR$ is a faithful positive tracial functional as claimed. 
Normality of the trace $\TR$ is a consequence of the normality of the usual trace $\Tr$ since $\TR (T)=\Tr (T_{e,e})$.
\end{proof}

\begin{proposition} 
Suppose that $T$ is an element of $\maN$ such that for any $g',g$ and any $k\in \maA$, the operator $\pi (k) T_{g', g}$ is trace class in $\what\maH$. Then for any $\varphi\in \maB$, the operator $\what\pi (\varphi) T$ belongs to the Schatten ideal of $\TR$-trace class bounded operators in $\maN$ and we have
$$
\TR (\what\pi (\varphi) T) \,\, = \,\,  \sum_{g\in \Gamma} \Tr \left( \pi (g^{-1} \varphi_g) T_{g, e} \right)  \,\,  = \,\,  \sum_{g\in \Gamma} \Tr \left( \pi (\varphi_g) (g\cdot T_{g^{-1}})\right).
$$
\end{proposition}

\begin{proof}
We assume first that $T$ is an element of $\maN$ such that for any $g',g$ and any $k\in \maA$, the operator $\pi (k) T_{g', g}$ is Hilbert-Schmidt in $\what\maH$. Then we have
$$
(\what\pi (\varphi) T)_{g, g'} = \sum_{k\in \Gamma} \pi (\varphi_{gk^{-1}}) \circ U_{gk^{-1}} \circ T_{k{g'}^{-1}} \circ U_{k{g'}^{-1}}. 
$$
The Hilbert-Shmidt norm in $\maN$, with respect to the trace $\TR$,  is given by $
\vert\vert A \vert\vert_2^2 = \sum_{g} \vert\vert A_{g, e}\vert\vert_2^2$,
which gives the estimate
$$
\vert\vert \what\pi (\varphi) T \vert\vert_2^2 \leq \sum_{g\in \Gamma} \left( \sum_k \left \vert\left \vert \pi (\varphi_{gk^{-1}}) \circ U_{gk^{-1}} \circ T_{k} \circ U_{k}   \right\vert\right\vert_2\right) ^2.
$$
The sums being finite, we concentrate on each term and we can write
$$
\left \vert\left \vert \pi (\varphi_{gk^{-1}}) \circ U_{gk^{-1}} \circ T_{k} \circ U_{k}   \right\vert\right\vert_2 \leq \left \vert\left \vert \pi ((kg^{-1}) \varphi_{gk^{-1}}) \circ T_{k}   \right\vert\right\vert_2
$$
Thus $\what\pi (\varphi) T$ is a Hilbert-Schmidt operator with respect to the trace $\TR$ in the von Neumann algebra $\maN$ acting in the Hilbert space $\ell^2\Gamma\otimes\what\maH$. 

A classical argument then shows that if the operator $\pi (k) T_{g', g}$ is $\Tr$-class in $\what\maH$ for any $k\in \maA$, then $\what\pi (\varphi) T$ is $\TR$-class operator in the von Neumann algebra $\maN$. Moreover, 
$$
\left(\what\pi (\varphi) T\right)_{e,e} \,\,  =  \,\, \sum_g \what\pi (\varphi)_{e,g} \circ T_{g,e} \,\,  = \,\,  \sum_g  \pi (\varphi_g)\circ U_g \circ T_{g^{-1}}\circ U_{g^{-1}} \,\,  =  \,\, \sum_g  \pi (\varphi_g)  \circ (g \cdot T_{g^{-1}}). 
$$
The sum is of course finite.
\end{proof}

\subsection{A II$_\infty$ triple for proper actions on bifoliations}

  For $p\geq 1$, we denote by $\maL^p(\maN, \TR)$, or simply $\maL^p(\maN)$ the $p$-Schatten space, i.e. the space of $\TR$-measurable operators $A$ such that $(A^*A)^{p/2}$ has finite $\TR$-trace, \cite{Benameur2003, BenameurFack}. We point out that $\maL^p(\maN)\cap \maN$ is a two-sided $*$-ideal in $\maN$ whose closure is the two-sided ideal $\maK(\maN)=\maK (\maN, \TR)$ of $\TR$-compact operators.

\begin{lemma}\label{Schatten}\
Given a bounded operator $T$ on $\what\maH$, we have
\begin{itemize}
\item If $T$ belongs to the Schatten ideal $\maL^r(\what\maH)$ for some $r\geq 1$ then the operator $T\otimes \Id_{\ell^2\Gamma}$ belongs to Schatten ideal $\maL^r(\maN, \TR)\cap \maN$ in the von Neumann algebra $\maN$.
\item If $T$ is a compact operator then $T\otimes id_{\ell^2\Gamma}$ is a $\TR$-compact operator in $\maN$.
\end{itemize}
\end{lemma}

\begin{proof}\
Since the correspondence $T\mapsto T\otimes id_{\ell^2\Gamma}$ respects composition and adjoint, we can reduce the proof by standard arguments to the case of a nonnegative trace-class operator $T$ (so $r=1$). But then a straightforward computation gives$$
\TR (T\otimes id_{\ell^2\Gamma}) = \Tr (T).
$$
For the second item, we may assume that $T$ is  nonnengative with eigenvalues $(\lambda_n)_{n\geq 0}$ which satisfy $\lambda_n\to 0$ as $n\to +\infty$. The singular values $(\mu_t^{\TR})_{t\geq 0}$ of $ T\otimes id_{\ell^2\Gamma}$ with respect to the trace $\TR$ of $\maN$ are then given by \cite{FackKosaki} 
$$
\mu_t(T\otimes id_{\ell^2\Gamma}) =\lambda_{[t]} (T) \quad (= \mu_t(T)),
$$
where $[t]$ is the integral part of $t$. Therefore, we also have $\mu_t^{\TR}\to 0$ as $t\to +\infty$. Thus the operator $ T\otimes id_{\ell^2\Gamma}$ is $\TR$-compact in $\maN$.
\end{proof}

\begin{proposition}\label{Properties-maN}
\begin{enumerate}
\item For $\varphi\in \maB$, the operator ${\what{\pi}} (\varphi)$ belongs to $\Psi_{c,\Gamma}^{0, -\infty}$.
\item  If $T$ is an element of  $\Psi_{\Gamma}^{m, \ell} $ with $m\leq 0$ and $\ell\leq 0$, then the induced operator extends to a bounded operator which then belongs to the von Neumann algebra $\maN$.  
\item If $T$ is an element of  $\Psi^{m, \ell} _{c, \Gamma}$ with $m<-(v+2n)$ and $\ell<-p$, then the induced operator belongs to the Schatten 
ideal $\maL^1(\maN)\cap \maN$ of $\TR$-class operators in $\maN$. 
\item If $T$ is an element of  $\Psi^{m, \ell}_{c, \Gamma}$ with $m<0$ and $\ell<0$, then the induced 
operator belongs to the ideal $\maK (\maN)$ of $\TR$-compact operators in $\maN$.
\end{enumerate}
\end{proposition}

\begin{proof}
To prove (1), by the  definitions of $\maB$ and $\Psi_{c, \Gamma}^{0, -\infty}$, we only need to show that for any $k\in \maA$, the operator $\pi (k)$ belongs to the class ${\Psi'}_c^{0, -\infty} (\hM, \what\maF\subset\what\maF'; \what\maE)$. But this is the content of Lemma \ref{pi(k)}. 

Regarding (2), we can apply (2) of Proposition \ref{Bounded} to deduce that $T$ induces a bounded operator on the Hilbert space $\what\maH\otimes \ell^2\Gamma$.  On the other hand we can see such operator (or rather $\lambda (T)$) as an element of $\maN$ since it belongs to $\C(\Gamma, B(\what\maH))$. 

For (3), from the definition of the trace $\TR$ on the von Neumann algebra $\maN$, it is clear that we only need to show that for any $A\in {\Psi'}^{m, \ell} (\what{M}, \what\maF\subset\what\maF'; \what\maE)$, with $m<-(v+2n)$ and $\ell < p$, the corresponding operator $\lambda (A)$ is trace class as an operator on $\what\maH$. But this is precisely the content  of (2) of Corollary \ref{Traceability}.

The proof of (4) is similar since some power of $T^*T$ is then $\TR$-class in $\maN$ and this implies by standard arguments (left as an exercise)  that $T$ itself is $\TR$-compact. 
\end{proof}

We quote the following corollary which will be used in the sequel. 

\begin{corollary}
If $T$ is an element of  $\Psi^{m, \infty} _{ \Gamma}$ with $m<-(v+2n)$, then for any $\varphi\in \maB$, the operator ${\what{\pi}} (\varphi)\circ T$ belongs to the Schatten 
ideal $\maL^1(\maN, \TR)\cap \maN$. If we only assume that $m<0$, then the operator ${\what{\pi}} (\varphi)\circ T$ belongs to the ideal $\maK (\maN)$ of $\TR$-compact operators in the semi-finite von Neumann algebra $\maN$.
\end{corollary}

\begin{definition}
Given a densely defined operator $(\hD, Dom (\hD))$ in $\what\maH$, the densely defined operator $\hD_\rtimes $ with domain $\ell^2 (\Gamma, Dom (\hD))$ is 
$$
\hD_\rtimes (\xi) (g) := \hD (\xi (g)).
$$
\end{definition}

The space $\ell^2 (\Gamma, Dom (\hD))$ is  defined as the subset of $\ell^2 (\Gamma, \what\maH)$ composed of such $\ell^2$ maps  which are valued in $Dom (\hD)$. So, it is obvious that this is a dense subspace of $\ell^2 (\Gamma, \what\maH)$ whenever $Dom (\hD)$ is a dense subspace of $\what\maH$. Recall that  the polar decomposition holds for any closed densely defined operator. 

\begin{proposition}\label{rtimes}
Let $(\hD, Dom (\hD))$ be a closed densely defined operator and let $\varphi$ be a continuous compactly supported function on $\what\maG\rtimes \Gamma$, then
\begin{itemize}
\item The  operator $(\hD_\rtimes , \ell^2 (\Gamma, Dom (\hD)))$ is a closed operator which is affiliated with the von Neumann algebra $\maN$. 
\item If $\hD$ is essentially self-adjoint so is $\hD_\rtimes$.
\item Assume that for any $g\in \Gamma$, the operator $\pi (\varphi _g)$ preserves $Dom (\hD)$ and that $[\hD, \pi (\varphi _g)]$ extends to a bounded operator on $\what\maH$. Then the operator $\what\pi (\varphi)$ preserves the domain of $\hD_\rtimes$ and  $[\what\pi (\varphi), \hD_\rtimes]$ extends to a bounded operator on $\ell^2\Gamma\otimes \what\maH$, which then belongs to $\maN$. 
\end{itemize}
\end{proposition}

\begin{proof}
That $(\hD_\rtimes , \ell^2(\Gamma, Dom (\hD)))$ is closed is easy to verify. Moreover, the minimal (resp. maximal) domain of $\hD_\rtimes$ coincides with $\ell^2(\Gamma, Dom_{min} (\hD))$ (resp. with $\ell^2(\Gamma, Dom_{max} (\hD))$). The partial isometry in the polar decomposition of $\hD_\rtimes$ coincides with the operator $I \otimes \what U$ where $\what U\in B(\what\maH)$ is the partial isometry appearing in the polar decomposition of $\hD$. But by the definition of $\maN$ we know that $I \otimes \what U \in \maN$. The same argument works for all spectral projections of the modulus $\vert\hD_\rtimes\vert$ of the operator $\hD_\rtimes$, since we have $ \vert\hD_\rtimes\vert = \vert\hD\vert _\rtimes$. Hence,  these spectral projections belong to $\maN$.

For $\xi\in \ell^2\Gamma\otimes \what\maH$, we have
$$
[\what\pi (\varphi), \hD_\rtimes] (\xi)  = [\pi (\varphi (\bullet)), \hD] * \xi  \quad \text{ that is } \quad [\what\pi (\varphi), \hD_\rtimes] (\xi) _g = \sum_{g_1g_2=g} [\pi (\varphi _{g_1}), \hD] (\xi _{g_2}).
$$
Set $C(\varphi):= \sum_{k\in \Gamma} \vert\vert[\pi (\varphi _{k}), \hD] \vert\vert_{B(\what\maH)} <+\infty$, where the sum is  finite since $\varphi$ is compactly supported. Then a straightforward computation gives
$$
\vert\vert[\what\pi (\varphi), \hD_\rtimes] (\xi) \vert\vert_{\ell^2\Gamma\otimes \what\maH} \,\, \leq \,\, C(\varphi) \times  \vert\vert \xi \vert\vert_{\ell^2\Gamma\otimes \what\maH}.
$$
\end{proof}

Recall that all our $\Gamma$-equivariant vector bundles are holonomy equivariant and have $C^\infty$-bounded geometry. We are now in position to state the main theorem. 

\begin{theorem}\label{CM-triple}
Let ${\what D}$ be a first order  $C^\infty$-bounded uniformly supported pseudodifferential  operator acting on the smooth compactly supported sections of $\what\maE$ over ${\what M}$. Assume that
\begin{itemize}
\item ${\what D}$ is uniformly transversely elliptic in the Connes-Moscovici pseudo' calculus ${\Psi'}^{1} ({\what M}, \what\maF'; \what\maE)$, and essentially self-adjoint with initial domain $C_c^\infty ({\what M}, \what\maE)$. 
\item ${\what D}$ has a holonomy invariant transverse principal symbol $\sigma ({\what D})$, a section over the total space of $\what\nu_{\what\maF'}^*  \oplus (T^*\what\maF '\cap \what\nu_{\what\maF})\simeq \what{V}^*\oplus \what{N}^*$.
\end{itemize}
Then the triple $(\maB, \maN,  {\what D}_\rtimes)$ is a  semi-finite spectral triple which is finitely summable of dimension $v+2n$. 
\end{theorem}

\begin{remark}
We have implicitely used the trace $\TR$ on $\maN$ and the operator ${\what D}_\rtimes$ has domain $\ell^2(\Gamma, Dom (\what{D}))$ with $Dom (\what{D})$ being the domain of the self-adjoint extension of $\what{D}$. 
\end{remark}

\begin{proof}\
By the first and second items of Proposition \ref{rtimes}, we know that $(\ell^2(\Gamma, Dom({\what D}), {\what D}_\rtimes )$ is also self-adjoint and affiliated with the von Neumann algebra $\maN$. By Corollary \ref{Resolvent}, we know that for any $k\in C_c^\infty(\what\maG)$ the operator $\pi (k) (\what{D}+i)^{-1}$ belongs to the Schatten ideal $\maL^r(L^2({\what M}, \what\maE))$ for any $r > v+2n$. Since we obviously have
$$
\left[(\what{D}+i)^{-1}\right]_\rtimes = \left[(\what{D}_\rtimes+i)\right] ^{-1},
$$
we can use again Lemma \ref{Schatten} and Proposition \ref{rtimes} to deduce that for any $\varphi\in \maB$, we have
$$
{\what\pi} (\varphi) \left[(\what{D}_\rtimes+i)\right] ^{-1}\; \in \; \maL^r (\maN, \TR)\quad \forall r > v+2n.
$$

We also proved in Lemma \ref{pi(k)} that $\pi (k)$ belongs to ${\Psi'}_c^{0, -\infty} (\what{M}, \what\maF \subset \what\maF')$.   Therefore, using Corollary \ref{Commutator} since $\what{D}$ has a holonomy invariant transverse principal symbol, we have that the commutator
$$
\left[ \what{D}, \pi_{\what\maE} (k)\right] = \what{D}\pi_{\what\maE} (k) - \pi_{\what\maE} (k)\what{D},
$$
belongs to  ${\Psi'}^{0, -\infty} (\what{M}, \what\maF \subset \what\maF'; \what\maE)$ and yields a bounded operator on the Hilbert space $L^2(\what{M}; \what\maE)$. Applying the third item of Proposition \ref{rtimes}, we see that for any $\varphi\in \maB$ the commutator $[{\what\pi} (\varphi), \what{D}_\rtimes]$ is a well defined bounded operator on $\what\maH\otimes \ell^2\Gamma$ which belongs to $\maN$. It is easy to check using the local Laplacians that the dimension is precisely $v+2n$. This is achieved using Proposition \ref{Dixmier} below, where we compute the semi-finite Dixmier trace as defined in \cite{BenameurFack}.
\end{proof}

Denote by $R$ the vector field generator of the flow $s\mapsto (e^s\eta_v, e^{2s}\eta_n)$ and by $\sigma_{-(v+2n)} (P)$ the principal symbol of an operator $P\in {\Psi'} ^{v+2n, -\infty} (\hM, \what\maF\subset  \what\maF'; \what\maE)$.
Then $\sigma_{-(v+2n)} (P)$ is a homogeneous section over the total space of the bundle $r^*\hV^* \oplus r^* \hN^*$ over  the graph $\what\maG$ of $\what\maF$ for the modified dilations $\lambda\cdot \eta$ defined above. The following is an easy modification of Proposition I.2 in \cite{CM95}, and its proof is omitted.

\begin{proposition}\label{Dixmier}
Let $P\in {\Psi}_{c, \Gamma} ^{-(v+2n), -\infty}$. The induced operator on $\what\maH\otimes \ell^2\Gamma$, which also belongs to the von Neumann algebra $\maN$,  belongs to the Dixmier ideal $\maL^{1,\infty} (\maN, \TR)$ associated with the semi-finite von Neumann algebra $(\maN, \TR)$. Moreover,  for any Dixmier state $\omega$ as in  \cite{BenameurFack}, the Dixmier trace $\TR_\omega (P)$ does not depend on $\omega$ and is given by the following formula
$$
\TR_{\omega} (P) = C(p, v, n) \int_{\vert \eta\vert' = 1} \; \tr \sigma_{-(v+2n)} (P_e) (\hm, \eta)\; i_R (d\hm d\eta),
$$
 where $C(p, v, n)$ is a constant which does not depend on $\omega$ and $P_e$ is the evaluation of $P$ at the unit element $e\in \Gamma$. Note that we have restricted $\sigma_{-(v+2n)} (P_e)$ to $\hM$ and hence the formula only involves the symbol as a section over a ``sphere'' in  the total space of the bundle $\hV^*\oplus \hN^*$ over $\hM$. 
\end{proposition}

\begin{remark}
If we assume furthermore that ${\what D}^2$ is also essentially self-adjoint with the scalar principal symbol in the Connes-Moscovici calculus, then the semi-finite spectral triple $(\maB, \maN,  {\what D}_\rtimes)$ is regular with simple dimension spectrum contained in $\{k\in\N\vert\; k\leq v+2n\}$. This result  will not be used in the sequel and the proof is a straigtforward (although tedious) extension of the proof given for Riemannian foliations in \cite{K97}. 
\end{remark}

In the case of trivial $\Gamma$, Theorem \ref{CM-triple} extends the results of \cite{K97} to non compcat bounded geometry foliations which moreover are not necessarily Riemannian but do satisfy the almost Riemannian condition of Connes-Moscovici \cite{CM95}. For non compact $\what{M}$,  the algebra $C_c^\infty ({\what \maG})$ can  be replaced by a larger algebra and still produce such a spectral triple but we have restricted ourselves to this simplest situation which will be adapted to proper (cocompact) actions in the next sections. As explained in the previous section, replacing the original (bounded-geometry) foliated manifold by its Connes fibration of transverse metrics, we can construct, using Theorem \ref{CM-triple},  additive maps  from the $K$-theory of Connes' $C^*$-algebra of any foliation to the reals by using any operator $\what{D}$ which satisfies the assumptions of Theorem \ref{CM-triple} on this fibration. 
 Recall that for a strongly Riemannian bifoliation with $\what{V}$ and $\what{N}$ oriented and even dimensional,  the CM transverse signature operator $\what{D}^{\sign}$ associated with the Riemannian structures and the transverse orientations was defined in \ref{CM-operator}. Since the metric on $\hM$ and all structures on our bundles are $\Gamma$-invariant, the  operator $\what{D}^{\sign}$ is $\Gamma$-invariant. Thus, we have the following important corollary.

\begin{theorem}\label{CrossedSign}
Assume that $(\hM, \what\maF\subset \what\maF'=\what\maF\oplus \what{V}, \Gamma)$ is a strongly Riemannian bounded geometry bifoliation with the proper bifoliated action of the countable group $\Gamma$ as above.  Then the triple 
$$
(\maB=C_c^\infty (\what\maG) \rtimes \Gamma, (\maN, \TR), \what{D}^{\sign}_\rtimes),
$$
is a semi-finite spectral triple of finite dimension equal to $v+2n$. Moreover, this spectral triple is regular with simple dimension spectrum contained in $\{k\in \N \vert k\leq v+2n\}$.
\end{theorem}

The Connes-Chern character of this spectral triple, in periodic cyclic cohomology, is called the equivariant transverse signature class of the almost Riemannian foliation $\what\maF$ with its proper action $\Gamma$.  It is denoted
$$
[\what{D}^{\sign}_\rtimes] \; \in \; \HP^0 (C_c^\infty (\what\maG) \rtimes \Gamma). 
$$
Notice that the pairing of this class with $K$-theory automatically extends to the $K$-theory of the maximal $C^*$-completion
$$
\left< \Ch (\bullet)\ ,\; [\what{D}^{\sign}_\rtimes]\right> \; : \; K_0(C^*(\what\maG\rtimes \Gamma) ) \longrightarrow \R.
$$

\begin{proof}
We only need to show that the CM operator $\what{D}^{\sign}$ satisfies the assumptions of Theorem \ref{CM-triple}. But as already observed, the operator $\what{D}^{\sign}$ is a first order  $C^\infty$-bounded uniformly supported pseudodifferential  operator in the Connes-Moscovici calculus for the foliation $\what\maF'$. Moreover, the transverse principal symbol in our Beals-Greiner sense satisfies the relation
$$
\sigma (\what{D}^{\sign})^4 (\hm, \eta) = (\vert\eta_v\vert^4 + \vert\eta_n\vert^2)\; \Id.
$$
Therefore, $\what{D}^{\sign}$ is uniformly transversely elliptic in the Connes-Moscovici pseudo' calculus ${\Psi'}^{1} ({\what M}, \what\maF'; \what\maE)$. Moreover, the operator is self-adjoint by definition since the operators $Q$ and $Q^2$ are self-adjoint. Indeed, $Q$ is actually self-adjoint with the initial domain $C_c^\infty (\hM; \what\maE)$ and we have considered its unique self-adjoint extension in the definition of $\what{D}^{\sign}$ inspired by the same construction for suspensions given in \cite{CM95}. Finally, the transverse principal symbol of $(\what{D}^{\sign})^2$ is scalar and given by
$$
\sigma (\what{D}^{\sign})^2 (\hm, \eta) = {\sqrt{\vert\eta_v\vert^4 + \vert\eta_n\vert^2}}\; \Id.
$$
The proof is now complete.
\end{proof}

Denote by  $\hP\to \hM$ the Connes fibration for a smooth foliation $(\hM, \what\maF)$ of bounded geometry, that is the fiber bundle of all the Euclidean metrics on the normal to the foliation $\what\maF$.    Then $\hP$  is automatically endowed with the bifoliation $\what\maF_{\hP} \subset\what\maF'_{\hP}$ as given in Lemma \ref{ConnesFibration} and it is a strongly Riemannian bifoliation. 
 The proper action of the countable group $\Gamma$ yields a proper action on $\hP$ preserving the foliations. As a corollary of Theorem \ref{CrossedSign}, we have the following  well defined transverse signature morphism for any smooth foliation with bounded geometry. 

\begin{theorem}
Assume that $(\hM, \what\maF)$ is a smooth foliation of bounded geometry and let $\Gamma$ be a countable group which acts properly by $\what\maF$-preserving isometries.  Then the CM transverse signature operator on the Connes fibration $(\hP, \what\maF_{\hP} \subset\what\maF'_{\hP})$ gives a group morphism $\Sign_{\hM, (2)}^\perp$
$$
\Sign^\perp := \left< \Ch (\bullet)\ ,\; [({\what{D}}^{\sign, \hP})_\rtimes]\right> \circ \Thom \; : \; K_0(C^*(\what\maG\rtimes \Gamma) ) \longrightarrow \R.
$$
\end{theorem}

Here $\Thom: K_0(C^*(\what\maG\rtimes \Gamma) ) \rightarrow K_0(C^*(\what\maG (\hP, \what\maF_{\hP} )\rtimes \Gamma))$ is the crossed product version of the ``Thom" isomorphism as defined in \cite{ConnesTransverse}, while $\what{D}^{\sign, \hP}$ is the CM signature operator on the strongly Riemannian bifoliation $(\hP, \what\maF_{\hP} \subset\what\maF'_{\hP})$. 

\section{The invariant triple for Galois coverings of bifoliations}

Let $(\hM, \what\maF) \rightarrow (M, \maF)$ be a Galois $\Gamma$-covering of smooth bounded-geometry foliations. So, $\pi_1 \hM$ is a normal subgroup of $\pi_1M$ with quotient group the countable group $\Gamma$
which acts freely and properly on  ${\what M}$ with quotient the smooth manifold $M$. We also assume that this action preserves the foliation $\what\maF$ and induces the foliation $\maF$ on $M$. So $\Gamma$ acts by diffeomorphisms of $\what M$ which preserve the leafwise bundle $T\what\maF$. We shall see that the group $\Gamma$ is somewhat artificial when working with the monodromy groupoid, so that the results of the present section will be equivalent to the same results for the universal cover $\tM\to M$ with the pull-back bifoliation. For simplicity, the reader may assume that $\hM=\tM$ and that $\Gamma=\pi_1M$ although we don't explicitly make this assumption. 

Fix a fundamental domain $\what U\subset \what M$ for the free and proper action of $\Gamma$ and denote by $\chi$ the characteristic function of $\what U$. Recall that we  can assume that 
$$
{\what U} g \cap {\what U}g' \neq \emptyset \Rightarrow g=g' \text{ and } \bigcup_{g\in \Gamma} {\what U} g = \what M.
$$
These conditions are only required up to  negligible subsets of ${\what M}$, assuming if necessary that $\what{U}$ is open, as this will be enough and does not affect the proofs.  For instance, we will need that for almost all $\hm\in {\what M}$,  there is  a unique translate ${\what U}_{\hm} = {\what U} g$ of ${\what U}$ which contains $\hm$. 

\subsection{The rationality conjecture} 

As explained in the sequel,  the groupoid $\what{G}\rtimes \Gamma$ is equivalent to the monodromy groupoid $G$ of the foliation $(M, \maF)$.  Therefore, the transverse signature morphism $\Sign^\perp$ defined above induces, by composition with the Morita isomorphism $
K_0(C^*(G)) \longrightarrow K_0(C^*(\what{G}\rtimes \Gamma))$,
the group morphism
$$
\Sign_{M, (2)}^\perp \; : \; K_0(C^*(G)) \longrightarrow \R.
$$
Here again $C^*(G)$ is the maximal $C^*$-algebra of  $G$. 

\begin{conjecture}\label{Rational}\
Let $(\hM, \what\maF) \rightarrow (M, \maF)$ be a Galois $\Gamma$-covering of smooth bounded-geometry foliations. Then the transverse signature morphism  $
\Sign_{M, (2)}^\perp \; : \; K_0(C^*G) \longrightarrow \R$ is
always   rational.
\end{conjecture}

When one assumes furthermore that $G$ is torsion free, meaning that the fundamental groups of the leaves are torsion free, then using the main result of \cite{BenameurHeitsch17}, we shall see  that the transverse signature morphism  $
\Sign_{M, (2)}^\perp \; : \; K_0(C^*G) \longrightarrow \R$ is actually integer valued on the range of the Baum-Connes map for $G$. Therefore, the surjectivity of the Baum-Connes map will imply the integrality and hence a positive answer to the above conjecture. 

We have stated the above conjecture for the signature operator because of its implications in topology, but it should be clear that this conjecture can be stated for all the spectral triples constructed above, using transversely elliptic operators from the Connes-Moscovici calculus. We prove in \cite{BenameurHeitsch17} that when $G$ is torsion free, the morphism $\Sign_{M, (2)}^\perp$ is actually integer valued, when restricted to  the range of the maximal  Baum-Connes map for the monodromy groupoid $G$. So,  the above conjecture is true for all smooth foliations whose monodromy groupoid has a surjective maximal Baum-Connes map and is torsion free.

\begin{remark}\
 In the case of a foliation $\maF$ with a single leaf $M$, the conjecture becomes the claim that the regular trace induces an integer-valued group morphism of the $K$-theory of the maximal $C^*$-algebra of the fundamental group $\pi_1M$, when we assume that this latter is torison free, and is rational in general. But the torsion free case is well known to be  a consequence of Atiyah's $L^2$-index  theorem \cite{AtiyahCovering}, while the general case is a well known conjecture due to Baum and Connes.
\end{remark}

Now suppose that $(\hM, \what\maF \subset \what\maF') \rightarrow (M, \maF \subset \maF')$ is a Galois covering, where the proper and free action of the countable discrete group $\Gamma$ preserves the foliations $\what\maF$ and $\what\maF'$ and induces the bifoliation $\maF\subset \maF'$ of $M$. Notice that the   leaves of $\what\maF$ in ${\what M}$ are the connected components of the inverse images of the leaves of $\maF$ in $M$, and similarly for $\what\maF'$. An interesting situation occurs when $M$ is compact but we don't impose this condition in general. As in the previous sections, the transverse bundle  $\what\nu$ (resp. $\nu$) to the foliation $\what\maF$ (resp. $\maF$), is replaced by the $\Gamma$-equivariantly isomorphic bundle 
$$
T\what\maF'/T\what\maF\; \oplus\; T{\what M} /T\what\maF' \; \; \; \text{ (resp. }T\maF'/T\maF\;  \oplus\;  TM /T\maF').
$$ 
The bundle $T\what\maF'/T\what\maF \oplus T{\what M} /T\what\maF'$ is  endowed with the diagonal action of the group $\Gamma$ and the holonomy action of $\what\maF$ is well defined on this bundle.  We choose again a $\Gamma$-invariant metric on $\what M$ which allows us to identify $T{\what M} /T\what\maF'$ with the $\Gamma$-equivariant orthogonal bundle $\hN$ to $T\what\maF'$.  We may also  identify $T\what\maF'/T\what\maF$ with the $\Gamma$-equivariant orthogonal ${\what V}$ to $T\what\maF$ in $T\what\maF'$. The holonomy action then need not to be diagonal, but is supposed to be  triangular, see \cite{CM95}. The same situation occurs downstairs on $M$ with the bundle $T\maF'/T\maF \oplus TM /T\maF'$. 

\subsection{The monodromy Morita equivalence}

Recall that we are working here with the monodromy groupoids, which act on all the geometric data through the holonomy action by composing with the covering map from monodromy to holonomy. As the quotient $\what{G} /\Gamma$ is diffeomorphic to $G$, we also have the convenient identifications
$$
C^\infty_u (\what{G})^\Gamma \; \simeq \; C^\infty_u (G)\text{ and }C^\infty_{c-\Gamma} (\what{G})^\Gamma \; \simeq \; C^\infty_c (G),
$$
where  $\left(\bullet\right)^\Gamma$ means the $\Gamma$-invariant elements, and $C^\infty_{c-\Gamma} (\what{G})$ is the subspace of smooth functions on $\what{G}$ whose support is $\Gamma$-compact, that is whose support projects to a compact subspace of $G$.  Using this identification, we can make  $C^\infty_u (G)$ act on $L^2 (\hM; \what\maE)$ by $\Gamma$-invariant operators.  
There exists a smooth   function, 
$$
\varrho: \hM \rightarrow [0, 1]    \quad \text{ such that } \quad   \sum_{g\in \Gamma} g^*\varrho^2 = 1,
$$
and such that the restriction of the projection $\hM\to M$ to the support of $\varrho$ is proper. We denote by $A=C^*G$ the maximal $C^*$-algebra completion of $C_c^\infty (G)$  and by $B=C^*\what{G}\rtimes\Gamma$ the maximal $C^*$-algebra completion of $\maB=C_c^\infty (\what{G}\rtimes \Gamma)$.

\begin{proposition}
For  $\varphi\in C_c^\infty (G)$,  $\what\gamma\in \what{G}$,  $g\in \Gamma$, and $\gamma=p(\what\gamma)\in G$, set
$$
\Phi (\varphi) (\what\gamma ; g) := \varphi (\gamma) \varrho (r(\what\gamma)) \varrho (s(\what\gamma) g).
$$
Then  $\Phi (\varphi)$ is a smooth compactly supported function on $\what{G}\rtimes \Gamma$. Moreover,
\begin{enumerate}
\item  The map $\Phi$ is a morphism of  $*$-algebras.
\item  $\Phi$  induces a $K$-theory isomorphism from  $K_*(A)$ to $K_*(B)$ which does not depend on the choice of the cutoff function $\varrho$.
\end{enumerate}
\end{proposition}

\begin{proof}
Fix $\varphi, \varphi'\in C_c^\infty (G)$, then we have
\begin{eqnarray*}
(\Phi\varphi)(\Phi\varphi') (\what\gamma; g) & = & \sum_{g_1\in \Gamma} \int_{\what{G}^{r(\what\gamma)}} \;  (\Phi\varphi)(\what\gamma_1; g_1) (\Phi\varphi')((\what\gamma_1^{-1}\what\gamma) g_1; g_1^{-1} g)\;  d\what\eta^{r(\what\gamma)} (\what\gamma_1)\\
& = & \sum_{g_1\in \Gamma} \int_{\what{G}^{r(\what\gamma)}} \; \varphi (\gamma_1) \varrho (r(\what\gamma)) \varrho (s(\what\gamma_1) g_1) 
\varphi' (\gamma_1^{-1}\gamma) \varrho (s'(\what\gamma_1)g_1) \varrho (s(\what\gamma) g) \;  d\what\eta^{r(\what\gamma)} (\what\gamma_1)\\
& = &  \int_{\what{G}^{r(\what\gamma)}} \; \varphi (\gamma_1) \varphi' (\gamma_1^{-1}\gamma) \varrho (r(\what\gamma)) \varrho (s(\what\gamma) g) \left(\sum_{g_1\in \Gamma} \varrho^2 (s(\what\gamma_1) g_1) \right)  \;  d\what\eta^{r(\what\gamma)} (\what\gamma_1)\\
& = & \varrho (r(\what\gamma)) \varrho (s(\what\gamma) g)  \int_{\what{G}^{r(\what\gamma)}} \; \varphi (\gamma_1) \varphi' (\gamma_1^{-1}\gamma)  \;  d\what\eta^{r(\what\gamma)} (\what\gamma_1).
\end{eqnarray*}
But for any fixed $\what\gamma$ with projection  $\gamma$, and since the Haar system on $\what{G}$ is $\Gamma$-invariant, we have
$$
 \int_{\what{G}^{r(\what\gamma)}} \; \varphi (\gamma_1) \varphi' (\gamma_1^{-1}\gamma)  \;  d\what\eta^{r(\what\gamma)} (\what\gamma_1) =  \int_{G^{r(\gamma)}} \; \varphi (\gamma_1) \varphi' (\gamma_1^{-1}\gamma)  \;  d\eta^{r(\gamma)} (\gamma_1).
 $$ 
 This proves the multiplicicativity property.

Computing in the same way $(\Phi\varphi)^*$ we get
$$
(\Phi\varphi)^* (\what\gamma; g) = {\overline{\Phi\varphi ( \what\gamma^{-1}  g; g^{-1}})} 
=  {\overline{ \varphi (\gamma^{-1}) }} \varrho (s(\what\gamma) g) \varrho (r(\what\gamma))
=  (\Phi\varphi^*)  (\what\gamma; g).
$$

\medskip

The proof of (2) is  standard and we shall be brief. There is an equivalence bimodule $\maE$ which is defined as follows. The space $C_c(\what{G})$ is endowed with the structure of a  pre-Hilbert module over the algebra $C_c(G)$ which is given by the rules
\begin{itemize}
\item $\dd (f\varphi) (\what\gamma) := \int_{\what{G}^{r(\what\gamma)}} f(\what{\gamma}_1) \varphi (\gamma_1^{-1}\gamma)\; d\what\eta^{r(\what\gamma)} (\what\gamma_1)$, for $f\in C_c(\what{G})$ and $\varphi\in C_c(G)$;
\item $\dd \left\langle  f_1, f_2 \right\rangle (\gamma) := \sum_{g\in \Gamma} \int_{\what{G}^{r(\what\gamma) g}} \overline{f_1 (\what\gamma_1)} f_2 (\what\gamma_1^{-1} \what\gamma) \; d\what\eta^{r(\what\gamma)} (\what\gamma_1)$, for $f_1, f_2\in C_c(\what{G})$.
\end{itemize}
where $\gamma \in G$ is the class of $\what\gamma\in \what{G}$. The representation of the algebra $C_c(\what{G})\rtimes \Gamma$ as $C_c(G)$-linear operators of the above pre-Hilbert mopdule, is given by
$$
\Xi (f_1 \delta_g) (f_2) (\tm, \tm') := \int_{\what{G}^{r(\what\gamma)}}\; f_1 (\what\gamma_1) f_2 ((\what\gamma_1^{-1}\what\gamma)g) \; d\what\eta^{r(\what\gamma)} (\what\gamma_1).
$$
Moreover, we can view $C_c (\what{G})$ as a left pre-Hilbert module over the algebraic crossed product algebra $C_c(\what{G})\rtimes\Gamma$ using this left action and the scalar product given by
$$
\langle f_1 , f_2 \rangle (\what\gamma; g) := \int_{\what{G}^{r(\what\gamma)}}\; f_1 (\what\gamma_1) \overline{f_2 ((\what\gamma_1^{-1}\what\gamma)g)} \; d\what\eta^{r(\what\gamma)} (\what\gamma_1), \quad f_1, f_2\in C_c( \what{G}).
$$
The following fundamental relation is then easily verified
$$ 
\Xi (\langle f_1, f_2\rangle ) (h) = f_1 \langle f_2, h\rangle.
$$ 
Now, it is easy to see that the  completion $\maE$ with respect to the maximal norm of $C^*(G)$ yields a Hilbert bimodule which is a full imprimitivity bimodule, i.e. it is a full right Hilbert $C^* (G)$-module and the representation $\Xi$ identifies the maximal $C^*$-algebra $C^*(\what{G})\rtimes \Gamma$ with the compact operators of $\maE$.  See \cite{BenameurPiazza} for similar constructions in an easier situation but with the detailed proofs which can be immediately extended to the setting here. 
To finish this  proof, we point out that the range of the $*$-homomorphism $\Phi$ is identified through the representation $\Xi$ with a corner associated with the rank one projection  associated with the cutoff function $\varrho$. 
\end{proof}

\subsection{The $\Gamma$-invariant calculus and Atiyah's von Neumann algebra}

Fix an auxiliary  $\Gamma$-equivariant Hermitian vector bundle $\what\maE$ over $\what M$ which is the pull-back of some holonomy equivariant Hermitian bundle $\maE$ over $M$ for the foliation $\maF$. This means that the holonomy action preserves some Hermitian structure on $\maE$ so also the pull-back one on $\what\maF$.  Denote by $W$ this holonomy (and hence monodromy) action on $\maE$,  and by $\what W$ the pull-back  action of $\what\maG$ on $\what\maE$. So,
$$
{\what W}_{\what\gamma} : \what\maE_{s(\what\gamma)}\longrightarrow \what\maE_{r(\what\gamma)} \text{ is just } W_\gamma.
$$
Here and below removing the hat from $\what\gamma$ means that we take its projection to $\maG$. Note that the pull-back holonomy action is equal to the holonomy action upstairs for the pull-back foliation $\what\maF$.  Following Atiyah \cite{AtiyahCovering}, we consider the semi-finite von Neumann algebra $\maM$ of $\Gamma$-invariant bounded operators on the Hilbert space $\what\maH = L^2 (\what{M}; \what\maE)$ with respect to the induced unitary $\Gamma$-action. This von Neumann algebra is endowed with the semi-finite trace $\tau$ defined as follows. Recall that $\chi$ is the characteristic function of a fundamental domain  $\what U$.

\begin{definition}
For any $\Gamma$-invariant bounded operator $T$ which is nonnegative,  set
$$
\tau (T) := \Tr (M_\chi \circ T \circ M_\chi),
$$
where $M_\chi$ is the bounded operator on $\what\maH$ which is multiplication  by the Borel bounded function $\chi$.
\end{definition}

The following classical result was  proven by Atiyah in \cite{AtiyahCovering}.

\begin{proposition}\cite{AtiyahCovering}
The functional $\tau$ extends to a normal faithful  semi-finite positive trace on the von Neumann algebra $\maM$. In particular, $(\maM, \tau)$ is a semi-finite von Neumann algebra of type II$_\infty$. 
\end{proposition}

Since the Hermitian structure on $\what\maE$ is holonomy invariant, we were able to define the involutive representation $\pi$ of the algebra $C_u^\infty (\what\maG)$ of uniformly supported smoothly bounded functions on $\what\maG$, in the Hilbert space $\what\maH=L^2({\what M}; \what\maE)$. We thus obtain a well defined  involutive representation of $C^\infty_u(\maG)$ (and hence of $\maA$) in $\what\maH$ denoted again $\pi$,  obtained by using composition with the pull-back map corresponding to the projection $\what\maG \to \maG$. Notice that pull-backs respect the product and involution since our Haar system is $\Gamma$-invariant. More precisely, for any $k\in C_u^\infty (\maG)$ the operator $\pi (k)=\pi_{\what\maE}(k)$  is given by
$$
\pi (k) (u) (\what m) := \int_{ \gamma\in \maG^m} k (\gamma) W_{\gamma} u (s({ \widehat{\gamma}})) d\eta^{m} (\gamma),
$$
where $\what\gamma$ is the unique (leafwise equivalence class of the) path which covers $\gamma$ and satisfies $r(\what\gamma)=\what m$. Notice also that $\what\gamma$ is automatically contained in a leaf of $\what\maF$. It is then clear that the representation $\pi:C^\infty_u(\maG) \rightarrow B (\what\maH)$ is valued in the von Neumann algebra algebra $\maM$ of $\Gamma$-invariant operators. 
 
Denote by ${\Psi'}_{c-\Gamma}^{m,\ell} ({\what M}, \what\maF\subset  \what\maF'; \what\maE)$ the subspace of ${\Psi'}^{m,\ell} ({\what M}, \what\maF\subset  \what\maF'; \what\maE)$ which is composed of operators whose support projects to a compact subset of $M\times M$. Such operators are sometimes called $\Gamma$-compactly supported pseudodifferential operators. 

\begin{proposition}
The group $\Gamma$ acts on every ${\Psi'}^{m,\ell} ({\what M}, \what\maF\subset  \what\maF'; \what\maE)$ (resp. on every ${\Psi'}_{c-\Gamma}^{m,\ell} ({\what M}, \what\maF\subset  \what\maF'; \what\maE)$) and this induces a bifiltration-preserving action by $*$-automorphisms of the involutive algebra ${\Psi'}^{\infty, \infty}({\what M}, \what\maF\subset  \what\maF'; \what\maE)$ (resp. of the involutive algebra ${\Psi'}_{c-\Gamma}^{\infty, \infty}({\what M}, \what\maF\subset  \what\maF'; \what\maE)$),   which preserves the ideal $\Psi^{-\infty}({\what M}; \what\maE)$ (resp.  the ideal $\Psi_{c-\Gamma}^{-\infty}({\what M}; \what\maE)$) of uniformly (resp. $\Gamma$-compactly) supported smoothing operators. 
\end{proposition}

\begin{proof}
Since any element $g\in \Gamma$ acts as an isometry of ${\what M}$ and since it preserves each of the foliations $\what\maF$ and $\what\maF'$, we know that for any such $\Gamma$-equivariant bundle $\what\maE$, $g$ induces a linear isomorphism on ${\Psi'}^{m,\ell} ({\what M}, \what\maF\subset  \what\maF'; \what\maE)$ which moreover preserves  ${\Psi'}_{c-\Gamma}^{m,\ell} ({\what M}, \what\maF\subset  \what\maF'; \what\maE)$ and respects the product and $*$ structures on ${\Psi'}^{\infty,\infty} ({\what M}, \what\maF\subset  \what\maF'; \what\maE)$ and ${\Psi'}_{c-\Gamma}^{\infty,\infty} ({\what M}, \what\maF\subset  \what\maF'; \what\maE)$. The rest of the proof is clear. 
\end{proof}

Denote by ${\Psi'}^{m, \ell} ({\what M}, \what\maF\subset  \what\maF'; \what\maE)^\Gamma$ the subspace of ${\Psi'}^{m, \ell} ({\what M}, \what\maF\subset  \what\maF'; \what\maE)$ composed of $\Gamma$-invariant operators. Similarly, ${\Psi'}_{c-\Gamma}^{m,\ell} ({\what M}, \what\maF\subset  \what\maF'; \what\maE)^\Gamma$ is the subspace of $\Gamma$-compactly supported operators which are $\Gamma$-invariant. As an obvious corollary of the results of the previous sections, we have

\begin{proposition}
If $A \in {\Psi'}^{m, \ell} ({\what M}, \what\maF\subset  \what\maF'; \what\maE)^\Gamma$  and $B \in {\Psi'}^{m', \ell'} ({\what M}, \what\maF\subset  \what\maF'; \what\maE)^\Gamma$, then  $A\circ B$ $\in  {\Psi'}^{m+m', \ell+\ell'} ({\what M}, \what\maF\subset  \what\maF'; \what\maE)^\Gamma$,  and the formal adjoint $A^*$ belongs to ${\Psi'}^{m, \ell} ({\what M}, \what\maF\subset  \what\maF'; \what\maE)^\Gamma$. 
Moreover, if $A$ or $B$ belongs to ${\Psi'}_{c-\Gamma}^{*,*} ({\what M}, \what\maF\subset  \what\maF'; \what\maE)^\Gamma$ then $A\circ B$ belongs to ${\Psi'}_{c-\Gamma}^{m+m', \ell+\ell'} ({\what M}, \what\maF\subset  \what\maF'; \what\maE)^\Gamma$,  and the adjoint of a $\Gamma$-compactly supported operator is $\Gamma$-compactly supported. 
 \end{proposition}
 
 For $p\geq 1$,  denote by $\maL^p(\maM, \tau)$, or simply $\maL^p(\maM)$ the $p$-Schatten space associated with the semi-finite von Neumann algebra $\maM$ (with respect to its trace $\tau$) \cite{Benameur2003, BenameurFack}.  The following proposition is the exact counterpart of Proposition \ref{Properties-maN} in the Galois covering case. 

\begin{proposition}\label{SemiFinite}\
\begin{enumerate}
\item For any $k\in \maA$, the operator $\pi (k)$ belongs to ${\Psi'}_{c-\Gamma}^{0, -\infty} ({\what M}, \what\maF\subset  \what\maF'; \what\maE)^\Gamma$.
\item  If $T$ is an element of  ${\Psi'}^{m, \ell} ({\what M}, \what\maF\subset  \what\maF'; \what\maE)^\Gamma$ with $m\leq 0$ and $\ell\leq 0$, then it induces a bounded operator on $L^2 ({\what M}; \what\maE)$ which  belongs to the von Neumann algebra $\maM$.  
\item If $T$ is an element of  ${\Psi'}_{c-\Gamma}^{m, \ell} ({\what M}, \what\maF\subset  \what\maF'; \what\maE)^\Gamma$ with $m<-(v+2n)$ and $\ell<-p$, then the induced operator has continuous Schwartz kernel and belongs to the Schatten  ideal $L^1(\maM)\cap \maM$.
\item If $T$ is an element of  ${\Psi'}_{c-\Gamma}^{m, \ell} ({\what M}, \what\maF\subset  \what\maF'; \what\maE)^\Gamma$ with $m<0$ and $\ell<0$, then the induced operator belongs to the ideal $\maK (\maM)$ of $\tau$-compact operators in $\maM$.
\end{enumerate}
\end{proposition}

\begin{proof}\
(1)  The pull-back of $k$ is uniformly supported and smoothly bounded.    By Lemma \ref{pi(k)},  $\pi (k)$ belongs to ${\Psi'}^{0, -\infty} (\hM, \what\maF\subset  \what\maF'; \what\maE)$  and it is obviously $\Gamma$-invariant. Moreover, since $k\in \maA$, the support of its pull-back projects to its support which is a compact subset of $\maG$ and hence $\pi(k)$ belongs to ${\Psi'}_{c-\Gamma}^{0, -\infty} (\hM, \what\maF\subset  \what\maF'; \what\maE)^\Gamma$. 

(2)  Apply  Proposition \ref{SobolevBound} of  the Appendix to $(\hM, \what\maF\subset \what\maF')$ with $m\leq 0$ and $\ell\leq 0$ to  deduce that $T$ induces a bounded operator on  $L^2 (\hM; \what\maE)$. Since $T$ is $\Gamma$-invariant by hypothesis,  $T$ belongs to the von Neumann algebra $\maM$. 

(3) Since the polar decomposition of any element of $\maM$ holds in $\maM$, the argument in the proof of Lemma \ref{TraceClass} applies to reduce the problem to proving  the following implication
$$
T\in {\Psi'}_{c-\Gamma}^{m, \ell} ({\what M}, \what\maF\subset  \what\maF'; \what\maE)^\Gamma\text{ with }m<-(v+2n)/2\text{ and }\ell<-p/2 \; \Longrightarrow \; T \in \maL^2 (\maM, \tau).
$$
$\maL^2 (\maM, \tau)$ is the Schatten space of Hilbert-Schmidt $\tau$-measurable operators in the semi-finite von Neumann algebra $(\maM, \tau)$ \cite{BenameurFack}. Since $T$ is $\Gamma$-compactly supported, it can be written as a finite sum of operators acting, over distinguished trivializing open sets $p^{-1}V$ and $p^{-1}V'$, from $C_{c-\Gamma}^\infty (p^{-1}V, \C^d)$ to $C_{c-\Gamma}^\infty (p^{-1}V', \C^d)$.   They are $\Gamma$-invariant and live in the class ${\Psi'}_{c-\Gamma}^{m, \ell}$ with respect to the restricted trivial foliations on $p^{-1}V$ and $p^{-1}V'$.  Since the Schwartz kernel $K (\hm, \hm')$ of such an operator is $\Gamma$-invariant and $\Gamma$-compactly supported in $p^{-1}V\times p^{-1}V'$, it belongs to $L^2(\maM, \tau)$ if and only if $\vert K\vert^2$ is integrable over some fundamental domain, or equivalently is integrable over any compact subspace of $p^{-1}V\times p^{-1}V'$. Therefore, we need  to justify that any (single) elementary operator which is compactly supported and associated with $k(z, x, y; \zeta, \eta, \sigma)$ from the class ${\Psi'}^{m, \ell}$ with $m<-(v+2n)/2$ and $\ell<-p/2$, has square integrable Schwartz kernel. The proof of Lemma \ref{TraceClass} then applies here, mutatis mutandis.

(4)  Note that for any continuous function $f:[0, \infty)\to \R$ with $f(0)=0$ and any compact nonnegative operator $S$, the operator  $f(S)$ is  $\tau$-compact. On the other hand, the operator $T\in \maM$ is $\tau$-compact if and only if the operator $T^*T$ is $\tau$-compact. Now, $T^*T$ is easily seen to belong to ${\Psi'}_{c-\Gamma}^{m,\ell} (\hM, \what\maF\subset \what\maF'; \what\maE)^\Gamma$.  Since the extension of the formal adjoint $T^*$ to $\what\maH$ is obviously the adjoint of the extension of $T$, we may assume without lost of generality that $T$ extends to a nonnegative operator in $\maM$.  Since $m<0$ and $\ell <0$, for $k\geq 1$ large enough the nonnegative operator $T^k$ satisfies the conditions of item (3). Therefore the extension of $T^k$ to a bounded operator in $\maM$ is an element the Schatten ideal $L^1(\maM, \tau) \cap \maM$ and hence belongs to the ideal $\maK(\maM)$. But this implies that $T$ is $\tau$-compact,  by using for instance the continuous function $t^{1/k}$ which vanishes at zero.
\end{proof}

\begin{corollary}
 If $k\in \maA$ and  $T$ in ${\Psi'}^{m, \ell} (\hM, \what\maF\subset  \what\maF'; \what\maE)^\Gamma$ with $m < -(v+2n)$, then the composite operator $\pi (k)\circ T$ extends to an element of $\maM$ which is $\tau$-trace class. 
\end{corollary}
 
 \begin{proof}
We apply the previous proposition. As above, the composition of a $\Gamma$-compactly supported $\Gamma$-invariant operator with a uniformly supported $\Gamma$-invariant operator is automatically a $\Gamma$-compactly supported $\Gamma$-invariant operator.  The operator $\pi (k)$ was shown to belong to ${\Psi'}^{0, -\infty}_{c-\Gamma} (\hM, \what\maF\subset  \what\maF'; \what\maE)^\Gamma$, so the composite operator $\pi (k)\circ T$ belongs to ${\Psi'}_{c-\Gamma}^{m, -\infty} (\hM, \what\maF\subset  \what\maF'; \what\maE)^\Gamma$. Since $m< -v-2n$, we have that the operator $\pi (k)\circ T$, which lives in $\maM$, is actually $\tau$-trace class by Proposition \ref{SemiFinite}. 
 \end{proof}

\subsection{The Atiyah-Connes semi-finite triple}

For the next results, note that any leafwise uniformly supported pseudodifferential operator of order $\ell$ on $\hM$ with respect to the foliation $\what\maF$ and with coefficients in $\what\maE$  which is $\Gamma$-invariant lives in   ${\Psi'}^{0,\ell} (\hM, \what\maF\subset \what\maF'; \what\maE)^\Gamma$. In the same way, any $\Gamma$-invariant pseudodifferential operator on $\hM$ from the Connes-Moscovici calculus with respect to the foliation $\what\maF'$, and with coefficients in $\what\maE$,  yields an operator in the class ${\Psi'}^{m, 0}(\hM, \what\maF\subset\what\maF'; \what\maE)^\Gamma$. 

Recall the notion of transverse order $m\leq \ell$ of a pseudo differential operator from the Connes-Moscovici calculus ${\Psi'}^\ell (\what{M}, \what\maF'; \what\maE)$ defined above. The following is a consequence of Lemma \ref{TransversalOrder}. 

\begin{lemma}\label{TransversalOrderII}
Assume that $P\in {\Psi'}^\ell ({\what M}, \what\maF'; \what\maE)^\Gamma$ is a  uniformly supported order $\ell$ operator and is $\Gamma$-invariant with transversal order $m\leq \ell$. Then the operator $P$ belongs to the pseudodifferential  class  ${\Psi'}^{m, \ell-m} (\hM, \what\maF\subset \what\maF'; \what\maE)^\Gamma$. 
\end{lemma}

\begin{proof}
Applying Lemma  \ref{TransversalOrder}, we have that $P$ belongs to ${\Psi'}^{m, \ell-m} (\hM, \what\maF\subset \what\maF'; \what\maE)$. So $P$ is $\Gamma$-invariant, the proof is complete.
\end{proof}

We also have the following $\Gamma$-invariant corollary of Theorem \ref{Parametrix}.

\begin{proposition}\label{GammaParametrix}\
Let $P\in {\Psi'}^\ell(\hM, \what\maF'; \what\maE)^\Gamma$ be a $\Gamma$-invariant pseudodifferential operator from the Connes-Moscovici calculus with respect to the foliation $\what\maF'$ on $\hM$. Assume that $P$ is uniformly transversely elliptic. Then there exists a $\Gamma$-invariant operator $Q \in {\Psi'}^{-\ell} (\hM, \what\maF'; \what\maE)^\Gamma$  such that
$$
R=I-QP \text{ and  }S=I-PQ \;\; \in  \Psi^{-\infty} (\what\nu ^*, \what\maE)^\Gamma \;\; \; \; (\cap {\Psi'}^0(\hM, \what\maF'; \what\maE)^\Gamma).
$$
\end{proposition}

\begin{proof}
We only need to apply Theorem \ref{Parametrix}. Note that  in the construction of the parametrix $Q$, we can ensure that it is $\Gamma$-invariant.
\end{proof}

In particular,  using the previous results we get the following.

\begin{corollary}
Let $P\in {\Psi'}^\ell (\hM, \what\maF'; \what\maE)^\Gamma$ be a $\Gamma$-invariant pseudodifferential operator in the Connes-Moscovici calculus with respect to the foliation $\what\maF'$ on $\hM$. Assume that $P$ is uniformly transversely elliptic and let $Q \in {\Psi'}^{-\ell} (\hM, \what\maF'; \what\maE)^\Gamma$ be a parametrix as in Proposition \ref{GammaParametrix}. Then for any $k\in C_c^\infty (\maG)$ the operators $\pi (k) (I-QP)$ and $\pi (k) (I-PQ)$ belong to the von Naumann algebra $\maM$ and are $\tau$-trace class operators, i.e. belong to the Schatten ideal $\maM \cap L^1(\maM, \tau)$ in $\maM$.
\end{corollary}

\begin{proof}
This is the $\Gamma$-invariant version of Corollary \ref{TraceParametrix} and the proof follows the same lines. 
By Proposition \ref{GammaParametrix}, we know in that
$$
I-QP \text{ and  }I-PQ \;\; \in   {\Psi'}^0({\what M}, \what\maF'; \what\maE)^\Gamma \cap {\Psi'}^{-n-2v-1} (\what\nu ^*, \what\maE)^\Gamma.
$$
On the other hand, by Proposition \ref{TransversalOrderII}, we have that 
$$
I-QP \text{ and  }I-PQ \;\; \in   {\Psi'}^{-(n+2v)-1, n+2v+1} ({\what M}, \what\maF\subset \what\maF' ; \what\maE)^\Gamma.
$$
But for $k\in C_c^\infty (\maG)$, by  item (1) of Proposition \ref{SemiFinite},  $\pi (k) \in {\Psi'}_{c-\Gamma}^{0, -\infty} ({\what M}, \what\maF\subset\what\maF' ; \what\maE)^\Gamma$. Therefore, 
$$
\pi (k) \in {\Psi'}_{c-\Gamma}^{0, -(p+v+2n)-2} ({\what M}, \what\maF \subset \what\maF'; \what\maE)^\Gamma.
$$
As a consequence, we obtain
$$
\pi (k) (I-QP) \text{ and  }\pi (k) (I-PQ) \;\; \in   {\Psi'}_{c-\Gamma}^{-(v+2n)-1, -p-1} ({\what M}, \what\maF\subset \what\maF' ; \what\maE)^\Gamma.
$$
The proof is completed using item (3) of  Proposition  \ref{SemiFinite}.
\end{proof}

\begin{proposition}
Assume that $P\in {\Psi'}^1 (\hM, \what\maF'; \what\maE)^\Gamma$ is a $\Gamma$-invariant pseudodifferential operator from the Connes-Moscovici calculus with respect to the foliation $\what\maF'$. Assume that $P$ is uniformly transversely elliptic and  that it induces an (unbounded) invertible operator on $L^2 (\hM, \what\maE)$ (injective with dense image and bounded inverse). Then for any $k\in C_c^\infty (\maG)$ the operator $\pi (k) P^{-1}$ belongs to the ideal $\maK (\maM, \tau)$ of $\tau$-compact operators in $\maM$. More precisely, $\pi (k) P^{-1}$ belongs to the Schatten ideal $L^r(\maM, \tau)\cap \maM$ for any $r> v+2n$. 
\end{proposition}

\begin{proof}
Let $Q \in \Psi^{-1} (\hM, \what\maF'; \what\maE)^\Gamma$ be a parametrix for $P$ as in Proposition \ref{GammaParametrix}. Then we get
$$
\pi (k) P^{-1} = \pi (k) R P^{-1} + \pi (k) Q \text{ where } R=I-QP.
$$
From the previous corollary, we know that $\pi (k) R $ (and hence also $\pi (k) R P^{-1}$) is $\tau$-trace class. Since it is  $L^2$ bounded, it thus belongs to the Schatten ideal   $L^1(\maM, \tau) \cap \maM$. Therefore, we see that $\pi (k) R P^{-1}$ is in particular $\tau$-compact in $\maM$.

Since $Q\in {\Psi'}^{-1, 0} (\hM, \what\maF'; \what\maE)^\Gamma$,  $\pi (k) Q\in \Psi^{-1, -\infty}_{c-\Gamma} (\hM, \what\maF\subset\what\maF'; \what\maE)^\Gamma$. Applying  item (2) of Proposition \ref{SemiFinite}, we deduce that  $\pi (k) Q$ extends to a bounded $\Gamma$-invariant operator on $L^2(\hM, \what\maE)$.
\end{proof}

\begin{corollary}\label{ResolventII}
Let $D \in {\Psi'}^{1} (\hM, \what\maF'; \what\maE)^\Gamma$ be a uniformly transversely elliptic pseudodifferential operator from the Connes-Moscovici calculus with respect to the foliation $\what\maF'$. Assume that $D$ induces an essentially self-adjoint operator on the Hilbert space  $L^2(\hM, \what\maE)$. Then for any $k\in C_c^\infty(\maG)$ the operator $\pi (k) (D+i)^{-1}$ extends to a $\tau$-compact operator in $\maM$.  More precisley,  $\pi (k) (D+i)^{-1}$ belongs to the ideal $L^r(\maM, \tau)\cap \maM$ for any $r> v+2n$.
\end{corollary}

\begin{proof}
Simply apply the previous proposition to $P=D+i$.
\end{proof}

Using the above results, we can now prove  the main result of this sub-section.

\begin{theorem}\label{typeII}
Assume that ${\what D}$ is transversely elliptic $\Gamma$-invariant pseudodifferential operator from the Connes-Moscovici calculus ${\Psi'}^1 (\hM, \what\maF'; \what\maE)^\Gamma$ for the foliation $\what\maF'$, acting on the smooth sections of the Hermitian bundle $\what\maE:=p^*\maE$ over $\hM$, with  the holonomy invariant {\em{transverse principal symbol}}. Assume also that $\what{D}$ is essentially self-adjoint on $\what\maH$ with the initial domain $C_c^\infty (\hM; \what\maE)$. 
Then  the triple $(\maA, \maM, \what{D})$   is a semi-finite spectral triple which is finitely summable of dimension $v+2n$ in the sense of \cite{BenameurFack}. 
\end{theorem}

We have implicitely used in the above statement the involutive representation $\pi$ of $\maA=C_c^\infty (\maG)$, and the normal semi-finite trace $\tau$.  

\begin{remark}
If we assume moreover that $\what{D}^2$ has a scalar principal symbol then the semi-finite spectral triple  $(\maA, \maM, \hD)$  is  regular and  has simple dimension spectrum contained in the set  $\{n \in \N \, | \, n \leq q \}$. The  proof is long but straightforward with our tools, and is omitted.
\end{remark}

In the case $\what\maF=0$ and $T\what\maF'=T\hM$,  this theorem coincides with the semi-finite spectral triple for coverings as defined in \cite{BenameurFack}.

\begin{proof}\
The proof is an easy extension of the proof of Theorem \ref{CM-triple} to the semi-finite setting of Atiyah's von Neumann algebra. We first note that the essentially self-adjoint $\Gamma$-invariant operator $\what{D}$ is automatically affiliated with the von Neumann algebra $\maM$. We proved in Corollary \ref{ResolventII} that for any $k\in \maA= C_c^\infty(\maG)$ the operator $\pi (k) (\what{D}+i)^{-1}$ belongs to the Schatten ideal $\maL^r(\maM, \tau)\cap \maM$ for any $r > v+2n$. We also proved that $\pi (k)$ belongs to ${\Psi'}^{0, -\infty} (\what{M}, \what\maF \subset \what\maF')^\Gamma$.  Therefore, using Corollary \ref{Commutator} since $\what{D}$ has a holonomy invariant transverse principal symbol, we get that the commutator
$$
\left[ \what{D}, \pi (k)\right] = \what{D}\pi (k) - \pi (k)\what{D},
$$
belongs to  ${\Psi'}^{0, -\infty} (\what{M}, \what\maF \subset \what\maF'; \what\maE)^\Gamma$ and hence yields a bounded operator on the Hilbert space $L^2(\what{M}; \what\maE)$ which is $\Gamma$-invariant.  Thus the commutator $\left[ \what{D}, \pi (k)\right]$ belongs to the von Neumann algebra $\maM$. By classical arguments, we have that $(C_c^\infty (\maG), (\maM, \tau),  {\what D})$ is a (semi-finite) spectral triple with finite dimension $\leq v+2n$. It is easy to check using the local Laplacians for our $\Gamma$-invariant metric, that the dimension is precisely $v+2n$. 
\end{proof}

Note that any operator $\what{D}$ which satisfies the assumptions of Theorem \ref{typeII} induces downstairs an operator $D$ in the Connes-Moscovici calculus which yields a type I spectral triple  called the Connes Moscovici spectral triple and given by $(\maA, L^2 (M, \maE), D)$, see \cite{CM95}. The main result of \cite{BenameurHeitsch17} is that when $G$ is torsion free, the Connes-Chern characters of the two spectral triples in periodic cyclic cohomology coincide as morphisms on the range of the Baum-Connes map. As a corollary, we deduce the following.

\begin{theorem}
Assume that the Baum-Connes map for the monodromy groupoid $G$ is surjective and that $G$ is torsion free, then the morphism induced by the semi-finite spectral triple of Theorem \ref{typeII} 
$$
K_*(C^*(G))\longrightarrow \R,
$$
is integer valued. 
\end{theorem}

\subsection{The crossed product triple vs the invariant triple}

We now compare the two spectral triples associated with Galois coverings of bifoliations, via their Connes-Chern characters. Recall that working with the monodromy groupoids allows the use of the Morita morphism $\Phi:  C_c^\infty (\maG)\to \maB$ which induces an isomorphism between the $K$-theory groups of the completions. 
The main result of this sub-section is that the maps on $K$-theory which correspond to the Connes-Chern characters of both semi-finite spectral triples agree, using this Morita isomorphism. 

Recall that $\what\maH=L^2(\hM; \what\maE)$.  Define  linear maps $\lambda:\what\maH\to \what\maH\otimes \ell^2\Gamma$ and $\what\lambda: \what\maH\otimes\ell^2\Gamma \to \what\maH$ by setting
$$
\lambda (\eta) (\hm; g):= \varrho (\hm) \eta (\hm g)  \quad\quad \text{ and } \quad\quad\what\lambda (\xi) (\hm):=\sum_{g\in \Gamma} \varrho (\hm g) \xi (\hm g; g^{-1}).
$$
It is clear by definition  that  $\what\lambda \circ \lambda= \Id$, that $\what\lambda=\lambda^*$ and hence that  $\lambda\what\lambda$ is an orthogonal projection on $\what\maH\otimes \ell^2\Gamma$, given by
$$
(\lambda\what\lambda)(\xi) (\tv, x; \gamma):=\varrho (\tv, x) \sum_{\gamma_1\in \Gamma} \varrho(\tv\gamma_1, \gamma_1^{-1} x) \xi(\tv\gamma_1, \gamma_1^{-1}x; \gamma_1^{-1}\gamma).
$$
It is easy to check that $\lambda\what\lambda$ belongs to the von Neumann algebra $\maN$.

\begin{proposition}\label{Lambda}
For $T\in \maM$, 
$$
\Lambda (T):= \lambda\circ T \circ \what\lambda.
$$
is a $*$-monomorphism which identifies $\maM$ with the von Neumann sub algebra of $\maN$ which is the corner associated with the self-adjoint idempotent $\lambda\circ\what\lambda$ in $\maN$ 
In particular $\what\lambda \circ {\tilde \pi} (\Phi\varphi) \circ \lambda  = \pi (\varphi)$ and $\Lambda (\pi(\varphi)) := \lambda\circ \pi(\varphi)\circ \what\lambda = {\tilde\pi} (\Phi\varphi)$. 
\end{proposition}

\begin{proof}
We need to check that for any $T\in \maM$, we have $\Lambda (T)\circ R_\alpha=R_\alpha\circ \Lambda(T)$. But a straightforward calculation shows that
$$
[R_\alpha \circ \Lambda (T)](\xi)_\gamma = \Lambda (T)(\xi)_{\gamma\alpha}= \varrho \times (\gamma\alpha) T(\what\lambda \xi).
$$
Now, $\what\lambda (\xi)=\sum_{\gamma_1\in \Gamma} \gamma_1 (\varrho \times \xi_{\gamma_1^{-1}})$, and since $T$ is $\Gamma$-invariant,
$$
[R_\alpha \circ \Lambda (T)](\xi)_\gamma = \varrho \times \sum_{\gamma_1\in \Gamma} (\gamma\alpha\gamma_1) T(\varrho\times \xi_{\gamma_1^{-1}}).
$$
Computing $[\Lambda (T)\circ R_\alpha] (\xi)_\gamma$ we find
$$
[\Lambda (T)\circ R_\alpha] (\xi)_\gamma =\varrho\times \gamma [(T\circ \what\lambda\circ R_\alpha)(\xi)] = \varrho\times \gamma \sum_{\gamma_1\in \Gamma} \gamma_1[T(\varrho\times \xi_{\gamma_1^{-1}\alpha})].
$$
By setting $\gamma_1=\alpha\gamma_2$:, we get
$$
[\Lambda (T)\circ R_\alpha] (\xi)_\gamma = \varrho\times \sum_{\gamma_2\in \Gamma} (\gamma\alpha\gamma_2) T(\varrho\times \xi_{\gamma_2^{-1}}).
$$
This proves that $\Lambda$ sends $\maM$ to $\maN$. That this is a $*$-monomorphism is a consequence of the relations $\what\lambda \lambda=  \Id$ and $\what\lambda=\lambda^*$. 

If  $T=\pi(\varphi)$ for a smooth compactly supported function $\varphi$ then for $\xi \in \what\maH$, we have
$$
T(\varrho \times \xi_{\gamma_1^{-1}}) (\tv, x) = \int_{\tV} \varphi[\tv, \tv', x] \varrho (\tv',x) \xi (\tv', x; \gamma_1^{-1}) d\tv '.
$$
Therefore, 
$$
[(\gamma\gamma_1) T(\varrho \times \xi_{\gamma_1^{-1}})] (\tv, x) = \int_{\tV} \varphi[\tv\gamma\gamma_1, \tv', \gamma_1^{-1}\gamma^{-1} x] \varrho (\tv', \gamma_1^{-1}\gamma^{-1} x) \xi (\tv', \gamma_1^{-1}\gamma^{-1} x ; \gamma_1^{-1}) d\tv '.
$$
Setting $\tv'=\tv '_1\gamma\gamma_1$ we get
$$
[\Lambda (T)( \xi)] (\tv, x; \gamma)  = \varrho(\tv, x) \times\sum_{\gamma_1\in \Gamma} \int_{\tV} \varphi[\tv\, \tv'_1, x] \varrho (\tv'_1\gamma\gamma_1, \gamma_1^{-1}\gamma^{-1} x) \xi (\tv'_1\gamma\gamma_1, \gamma_1^{-1}\gamma^{-1} x ; \gamma_1^{-1}) d\tv '.
$$
But this coincides with
$$
\sum_{\gamma_1\in \Gamma} \int_{\tV} (\Phi\varphi) (\tv, \tv'_1,x;\gamma\gamma_1) \xi (\tv'_1\gamma\gamma_1, \gamma_1^{-1}\gamma^{-1} x; \gamma_1^{-1}) d\tv'_1.
$$
Setting $\gamma'_1=\gamma\gamma_1$ we  get ${\tilde\pi} (\Phi\varphi) (\xi) (\tv, x; \gamma)$. 

What remains is to notice  that for any ${\what T} \in \maN$, the operator $\what\lambda\circ {\what T}\circ \lambda$ belongs to $\maM$. This allows us to deduce that
$$
\Lambda: \maM \stackrel{\cong}{\longrightarrow} (\lambda\what\lambda) \maN (\lambda\what\lambda). 
$$
\end{proof}

\begin{proposition}\label{TR=tau}
$\TR \circ \Lambda = \tau.$
\end{proposition}

\begin{proof}
Fix $T\in \maM^+$ and an orthonormal basis $(f_i)_i$ of the Hilbert space $\maH$.   Then 
\begin{eqnarray*}
\TR (\Lambda (T)) &=&  \sum_i < (\lambda\circ T\circ \what\lambda )(f_i\otimes \delta_e) , f_i\otimes \delta_e >\\
& = &  \sum_i < (M_\varrho \circ T \circ M_\varrho) (f_i) , f_i >\\
& = & \Tr (M_\varrho \circ T \circ M_\varrho)
\end{eqnarray*}
We have used the relation $(\lambda \circ T\circ \what\lambda) (f\otimes \delta_\alpha) = \varrho \times (U_{\alpha^{-1}} \circ T) (\varrho f)$ but only for $\alpha=e$. But
$$
 \Tr (M_\varrho \circ T \circ M_\varrho) = \tau (T).
$$
and the proof is complete. 
\end{proof}

{{Gathering the previous results together, we can now  state the main result of this section.}}

\begin{theorem}\label{Versus}
Assume that ${\what D}$ is transversely elliptic $\Gamma$-invariant pseudodifferential operator from the Connes-Moscovici calculus ${\Psi'}^1 (\hM, \what\maF'; \what\maE)^\Gamma$ for the foliation $\what\maF'$, acting on the smooth sections of the Hermitian bundle $\what\maE:=p^*\maE$ over $\hM$, with  holonomy invariant {\em{transverse principal symbol}}. Assume also that $\what{D}$ is essentially self-adjoint on $\what\maH$ with the initial domain $C_c^\infty (\hM; \what\maE)$. 
Then the Connes-Chern characters of the semi-finite spectral triple  $(\maA, \maM, \what{D})$   coincides with the pull-back under the Morita map $\Phi$ of the Connes-Chern character of the semi-finite spectral triple $(\maB, \maN, \what{D}_\rtimes)$. 
\end{theorem}

\begin{proof}
We treat the odd case.   The even case is similar and only needs the use of  supertraces in place of  traces. A consequence of the local index theorem is that the Connes-Chern character of the Atiyah-Connes semi-finite spectral triple in periodic cyclic cohomology is invariant under bounded perturbations of the operator $\what{D}$ in the Atiyah von Neumann algebra $\maN$.  But the operator $\what\lambda \what{D}_\rtimes  \lambda$ is also affiliated with the von Neumann algebra $\maN$ and is a bounded perturbation of $\what{D}$ in $\maN$. Indeed, we have by direct computation 
$$
\what\lambda \what{D}_\rtimes  \lambda  - \what{D} = \sum_{g\in \Gamma} (g^{-1} \varrho)\;  [\what{D}, g^{-1}\varrho].
$$
Since the commutator $[\what{D}, g^{-1}\varrho]$ is bounded and belongs to $\maN$, the operator $\sum_{g\in \Gamma} (g^{-1} \varrho)\,  [\what{D}, g^{-1}\varrho]$ is well defined in $\maN$ as a strong limit of elements of $\maN$, so it belongs to $\maN$. Therefore, we may represent the Connes-Chern character of the Atiyah-Connes spectral triple by using the operator $\what\lambda \what{D}_\rtimes  \lambda$ in place of $\what{D}$. 
Note that the algebra $C_c^\infty (G)$ is not unital, but there is a classical trick which allows one to replace it by its unitalization.   Then the Connes-Chern character of the crossed product  semi-finite spectral triple, when pulled back under the Morita homomorphism $\Phi$, is represented for any $r>v+2n$ by the cyclic cocycle 
$$
(k_0, \cdots, k_r) \longmapsto  \TR \left(\what\pi (\Phi(k_0)) [\what{F}_\rtimes, \what\pi (\Phi (k_1)) ] \cdots [\what{F}_\rtimes, \what\pi (\Phi (k_r)) ]\right),
$$
where the operator $\what{F}$ is the symmetry built out of $\what{D}$ as usual.  It follows that  the corresponding symmetry constructed out of $\what{D}_\rtimes$ is nothing but $\what{F}_\rtimes$, thanks to the compatibility of the functional calculi that we have already explained.  From Proposition \ref{Lambda} we have for any $i$
$$
\what\pi (\Phi(k_i)) = \lambda\circ \pi (k_i)\circ \what\lambda.
$$
Thus, the pull-back under $\Phi$ of the Connes-Chern character of the crossed product spectral triple is represented by 
$$
(k_0, \cdots, k_r) \longmapsto  \TR \left(\lambda \pi (k_0)\what\lambda [\what{F}_\rtimes, \lambda \pi (k_1)\what\lambda ] \cdots [\what{F}_\rtimes, \lambda \pi (k_1)\what\lambda  ]\right),
$$
which equals
 $$
 \TR \left(\lambda \pi (k_0) [\what\lambda (\what{F}_\rtimes)\lambda, \pi (k_1) ] \cdots [\what\lambda (\what{F}_\rtimes) \lambda, \pi (k_1)] \what\lambda  \right).
 $$
 Using Proposition \ref{TR=tau}, we deduce that this latter is equal to 
 $$
  \tau \left(\pi (k_0) [\what\lambda \what{F}_\rtimes\lambda, \pi (k_1) ] \cdots [\what\lambda \what{F}_\rtimes \lambda, \pi (k_1)]  \right).
 $$
 Since $\what\lambda \what{F}_\rtimes\lambda$ is the symmetry which corresponds to the operator $\what\lambda \what{D}_\rtimes  \lambda$, and we are done.
\end{proof}

Using this compatibility result together with the main theorem of \cite{BenameurHeitsch17}, we have the claimed  integrality theorem for Riemannian foliations. More precisely, we have the following important integrality theorem.

\begin{theorem}\label{AtiyahTheorem}\
Assume that the foliation $(M, \maF)$ is Riemannian with torsion free monodromy groupoid $G$ and that the maximal Baum-Connes map for $G$ is surjective. Then the Connes-Chern character of any semi-finite spectral triple associated with a transversely hypo-elliptic  operator $\what{D}$ as in Theorem  \ref{CM-triple} is integral, that is induces an integer valued pairing with the $K$-theory of the maximal $C^*$-algebra $B=C^*(\what{G})\rtimes \Gamma$. 
\end{theorem}

\begin{proof}
In \cite{BenameurHeitsch17}, we proved that for Riemannian foliations the pairing of the Connes-Chern character of the semi-finite Atiyah-Connes spectral triple with the index classes of {\underline{leafwise elliptic}} operators on closed foliated manifolds coincides with the corresponding pairing for the type I Connes-Moscovici spectral triple. Therefore, this pairing is integer valued since it corresponds to a Fredholm index in the usual sense. On the other hand, when $G$ is torsion free, the range of the Baum-Connes map is built out of indices of such leafwise elliptic operators as a limit. So the proof is complete. 
\end{proof}

\begin{remark}\
\begin{itemize}
\item When the foliation $\maF$ is zero dimensional, the above theorem is a consequence of the Atiyah $L^2$-index theorem for coverings applied to elliptic first order operators with coefficients in compatcly supported virtual bundles on $M$. 
\item When the foliation $\maF$ is top dimensional, the groupoid $G$ is isomorphic to  ${\tM}\times_{\pi_1M} {\tM}$ and is Morita equivalent to the fundamental group $\pi_1M$. Then the theorem is the well-known statement that the regular trace and the averaging trace (i.e. the trivial representation) on the maximal $C^*$-algebra of $\pi_1M$, induce the same $K$-theory morphism on the range of the maximal Baum-Connes map for $\pi_1M$, provided this latter is torsion free. 
\end{itemize}
\end{remark}

\appendix

\section{Adapted Sobolev spaces}\label{Sobolev}

We give here the anisotropic Sobolev spaces adapted to our pseudodifferential operators. Suppose that $u \in C^{\infty}_c(\R^{\om}, \C^{a })$,  and denote its Fourier transform by $\what{u}$.   We shall use the decomposition of $\R^{\om}$ into $\R^p\times \R^v\times \R^{q-v}$ as above. For all $s,k \in \R$, the Sobolev $s,k$ norm of $u$ is defined in the classical way
$$
\|u\|^2_{s,k} \,\, = \,\, \int_{\zeta \in \R^p,  \eta \in \R^q} \; \; |\what{u}(\zeta,\eta)|^2(1 + \vert\xi\vert') ^{2s} (1 + |\zeta|)^{2k} d\zeta d\eta.
$$
Recall that ${\vert\xi \vert'}^2=\vert\zeta\vert ^2 + {\vert\eta\vert'}^2$. 
The case where $v=0$ was used in  \cite{K97} with the notation $\oH^{s,k}$.   Our modification allows us to prove results for our operators similar  to those used in \cite{K97}. See also \cite{GreenleafUhlmann}.  

\begin{definition}
The space ${\oH'}^{s,k}(\R^{\om},\R^p; \C^{a })$ is the completion of $C^{\infty}_c(\R^{\om}, \C^{a })$ under the norm $\| \cdot \|_{s,k}$. So, ${\oH'}^{s,k}(\R^{\om},\R^p; \C^{a })$ is a Banach space when endowed with  the norm $\|u\|_{s,k}$.
\end{definition}

Although the above spaces are not Hilbert spaces when $s\neq 0$, they will be convenient to exploit the pseudo' filtration in terms of the properties of the corresponding extended operator  on Sobolev spaces. Following the seminal references \cite{BealsGreiner} and \cite{CM95}, we have used the usual  quantization map.  The main results we will need will indeed be true with this quantization, as we shall see. Note first that by the Cauchy-Schwarz inequality, it is clear that the $L^2$ scalar product of any smooth compactly supported $u, v$ can be estimated as usual
$$
\vert <u, v>_{L^2} \vert \;\leq \; \vert\vert u\vert\vert_{s, k} \; \times \; \vert\vert v \vert\vert_{-s, -k}.
$$
This allows us to compute $\vert\vert u\vert\vert_{s, k}$  as the supremum over $\vert\vert v \vert\vert_{-s, -k}\leq 1$, of the $L^2$ expressions $\vert <u, v>_{L^2}\vert $. 

If $U$  is an open set in $\R^{\om}$ which is a product of open sets in $\R^p, \R^v$ and $\R^{q-v}$ respectively, and if $V$ is an extra open set in $\R^p$, then we define similarly the space ${\oH'}^{s,k}(U,V,\C^{a })$.
Let $(\hU_i, \hT_i)_{i\in I}$ be a  good open cover of the foliation $(\hM, \what\maF)$ with finite multiplicity and such that $\hU_i \simeq \R^p \times \hT_i$ and $\hT_i\simeq \R^v\times \R^{q-v}$ so that $\hU_i\simeq \R^p\times\R^v\times \R^{q-v}$. Using a classical lemma due to Gromov \cite{Gromov}, we know that such an open cover always exists. Moreover, we may assume that the open sets $\hU_i$ are products of metric balls in $\R^p$, $\R^v$ and $\R^{q-v}$ which are diffeomorphic ranges of the local exponential maps and such that any plaque of $\what\maF$ in any $\hU_i$ is the diffeomorphic range of the leafwise exponential map and similarly for the plaques of the foliation $\what\maF'$.  Let $\{\what\phi_i\}$  be a $C^\infty$-bounded partition of unity subordinate to the cover $\{\hU_i\}$ of $\hM$, see  \cite{Shubin92}. For $u \in C^{\infty}_c(\hM, \what\maE)$, and using the local trivializations of $\what\maE$ over the $\hU_i$, we define its $s,k$ norm as 
$$
\|u\|_{s,k} \,\, = \,\,  \sum_{i}   \| \what\phi_i \cdot u \|_{s,k},
$$
where on the right we are thinking of the product $\what\phi_i \cdot u$  as an element in $ C^{\infty}_c(\R^{\om}, \C^{a })$ using the trivializations, and the norm $ \| \cdot \|_{s,k}$ is pulled back from the norm of ${\oH'}^{s,k}(\R^{\om},\R^p,\C^{a })$.
\begin{definition}
The bigraded Sobolev space ${\oH'}^{s,k}(\hM,  \what\maF\subset  \what\maF';  \what\maE)$ is the completion of $C^{\infty}_c(\hM, \what\maE)$ under the  norm  $\| \cdot \|_{s,k}$. 
\end{definition}

Classical arguments show that although the norms depend on the choices, the bigraded Sobolev spaces ${\oH'}^{s,k}(\hM,  \what\maF\subset  \what\maF'; \what\maE)$ do not.  Recall the bigraded Hilbert space ${\oH}^{s,k}$ used in \cite{K97} given by
$$
{\oH}^{s,k} (\hM, \what\maF; \what\maE) = {\oH'}^{s,k} (\hM, \what\maF\subset T\hM; \what\maE).
$$
We naturally denote denote by ${\oH}^{s}$ the classical $s$-Sobolev space. 

\begin{remark}\label{Comparison}\
\begin{itemize}
\item When $s\geq 0$ then ${\oH}^{s,k}\subset {\oH'}^{s,k}$, while for $s\leq 0$, ${\oH'}^{s,k}\subset {\oH}^{s,k}$. 
\item If $k\geq 0$ then ${\oH}^{s+k, 0} \subset  {\oH}^{s,k}$ and ${\oH'}^{s+k, 0} \subset  {\oH'}^{s,k}$.
\item If  $s\geq 0$ and $k\geq 0$, then one has continuous inclusions
$$
{\oH}^{s+k}={\oH}^{s+k, 0} \subset  {\oH}^{s,k}\subset {\oH'}^{s,k}\subset {\oH}^{s/2,k} \subset {\oH}^{s/2}.
$$ 
\item If $s\leq 0$ and $k\leq 0$, then one has continuous inclusions
$$
{\oH}^{s,k} \subset {\oH'}^{2s, k}\subset {\oH}^{2s, k} \subset {\oH}^{2s+k,0}={\oH}^{2s+k}.
$$
There are similar inclusions for $s\geq 0$ or $k\leq 0$ or when $s\leq 0$ and $k\geq 0$. 
\end{itemize}
\end{remark}

\begin{proposition}\label{SobolevBound} (Compare \cite{K97})\
\begin{enumerate}
\item A uniform smoothing operator $T$ induces bounded operators
$$
T: {\oH'}^{s,k}(\hM, \what\maF\subset\what\maF' ; \what\maE) \longrightarrow {\oH'}^{s',k'}(\hM, \what\maF\subset\what\maF'; \what\maE), \quad \forall s, s', k, k'.
$$
\item Any operator $A\in {\Psi'}^{m,\ell}(\hM,  \what\maF\subset \what\maF'; \what\maE)$ extends to bounded operators
$$
A:{\oH'}^{s,k}(\hM,  \what\maF\subset  \what\maF'; \what\maE)\longrightarrow {\oH'}^{s-m,k-\ell}(\hM, \what\maF\subset  \what\maF';  \what\maE)), \quad \forall s, k.
$$
In particular, when $m\leq 0$ and $\ell\leq 0$, the operator $A$ extend to an $L^2$-bounded operator. 
\end{enumerate}
\end{proposition}

\begin{proof}\ 
It is well known that any $R\in \Psi^{-\infty} (\hM, \what\maE)$ induces a bounded operator between any two classical Sobolev spaces, see \cite{Shubin92}. This essentially gives the conclusion. More precisely, if we assume for instance that $s, s', k, k'$ are nonnegative, then given that such $R$ extends to a bounded operator from $\oH^{s/2}(\hM, \what\maE)$ to $\oH^{s'+k'}(\hM, \what\maE)$, we see that  $R$ extends to a bounded operator from ${\oH'}^{s,k}(\hM, \what\maF\subset\what\maF', {\what V} ; \what\maE)$ to ${\oH'}^{s',k'}(\hM, \what\maF\subset\what\maF'; \what\maE)$. Notice that the different extensions of $R$ are compatible with the inclusions of Remark \ref{Comparison}. 

Now given $A\in {\Psi'}^{m,\ell}(\hM,  \what\maF\subset\what\maF'; \what\maE)$,  we may use (1) and  a  partition of unity argument to reduce the proof of (2) to the case of an elementary local operator of type $(m, \ell)$. So we are reduced to  the  similar result   on $\I^{\om}=\I^p\times \I^v\times \I^{q-v}$ with $\what\maE$ being the trivial bundle.  Forgetting the constants, we then have (using the previous notations, e.\ g.\ $\xi= (\zeta, \eta)$ and $\xi'= (\zeta', \eta')$ etc)
$$
{\what{Au}} (\xi'=(\zeta', \eta')) = \int (Au) (x,y) e^{-i(x\zeta'+ y\eta')} dx dy.
$$
Replacing $Au$ by its local expression in terms of its symbol $k(z, x, y; \sigma, \zeta, \eta)$, one gets
$$
{\what{Au}} (\xi') =\int K(\xi, \xi-\xi') {\what u} (\xi) d\xi \text{ where } K(\xi_1, \xi_2) = \int k(z, x, y; \sigma, \xi_1) e^{iz(\sigma-\zeta_1)} e^{i(x\zeta_2+ y \eta_2)} dz d\sigma dx dy.
$$
Therefore,  
\begin{eqnarray*}
\left<  Au, v\right> & = & \int K(\xi, \xi-\xi') {\what u} (\xi) {\overline{{\what v} (\xi')}} d\xi d\xi'\\
& = & \int \rho_{m} (\xi, \xi') \left[ (1+\vert\xi\vert')^s (1+\vert \zeta\vert)^k {\what u} (\xi) \right] \left[ (1+\vert\xi'\vert)^{m-s} (1+\vert\zeta'\vert)^{\ell-k}{\what v} (\xi') \right] d\xi d\xi'\\
& \leq & \left[ \int \vert\rho_{m}(\xi,\xi')\vert \cdot \vert {\what u}_{s,k} (\xi)\vert^2 d\xi d\xi'\right]^{1/2} \times \left[\int \vert\rho_{m}(\xi,\xi')\vert \cdot\vert {\what v}_{m-s, \ell-k} (\xi') \vert^2 d\xi d\xi' \right]^{1/2}.
\end{eqnarray*}
We have denoted here 
$$
{\what u}_{s,k} (\xi) =(1+\vert\xi\vert')^s (1+\vert \zeta\vert)^k {\what u} (\xi),  \quad\quad {\what v}_{m-s, \ell-k} (\xi') = (1+\vert\xi'\vert)^{m-s} (1+\vert\zeta'\vert)^{\ell-k}{\what v} (\xi'), 
$$
and
$$
\rho_{m} (\xi, \xi') = K(\xi, \xi-\xi') (1+\vert \xi\vert')^{-s} (1+\vert\zeta\vert)^{-k} (1+\vert\xi'\vert')^{s-m} (1+\vert\zeta'\vert)^{k-\ell}.
$$
Now, for any $M, N\in \N$ there exists $C_1\geq 0$ such that
$$
\vert \partial_z^N \partial_{x, y}^M k(z, x, y; \sigma, \xi) \vert \leq C_1 (1+ \vert\xi\vert ')^m (1+\vert\sigma\vert)^\ell.
$$
Therefore,  
$$
\vert K(\xi, \xi-\xi') \vert \leq  C_2 (1+\vert\xi\vert')^m (1+\vert\xi -\xi'\vert) ^{-M} \int (1+\vert\sigma\vert)^\ell (1+\vert \sigma-\zeta\vert )^{-N} d\sigma.
$$
Note that the following relations hold
$$
(1+\vert\xi_1\vert)^{1/2} \leq 2 (1+\vert\xi_1\vert')  \leq 4 (1+\vert\xi_1\vert), \quad\forall \xi_1\in \R^p\times \R^q.
$$
Now,  apply Petree's inequality for the norm $\vert\cdot\vert$, as well as for our norm $\vert \cdot\vert'$, and the fact that $1+\vert X\vert' \leq 2 (1+\vert X\vert)$.   This easily gives the existence of $C_3\geq 0$ and then $C_4\geq 0$ such that the following inequality holds
\begin{eqnarray*}
\vert\rho_{m} (\xi, \xi')\vert  &\leq &C_3 (1+\vert\xi\vert')^{m-s} (1+\vert\zeta\vert)^{\ell-k} (1+ \vert \xi-\xi'\vert)^{-M}  (1+\vert\xi'\vert')^{s-m} (1+\vert\zeta'\vert)^{k-\ell}\\
& \leq & C_4 (1+\vert \xi'-\xi\vert')^{\vert s-m\vert + \vert k-\ell\vert - M}.
\end{eqnarray*}
This finishes the proof as for $M$ large enough, we know that 
$$
\int (1+\vert \xi'-\xi\vert')^{\vert s-m\vert + \vert k-\ell\vert - M} \; d\xi \; = \; \int (1+\vert \xi'-\xi\vert')^{\vert s-m\vert + \vert k-\ell\vert - M} \; d\xi'  \; <\; +\infty.
$$
\end{proof}

\end{document}